\documentclass[12pt]{amsart}
\usepackage{amsmath,amssymb,amsthm,enumerate,bm,bbm,mathrsfs,xcolor,hyperref,theoremref,tabularx,array,makecell,upgreek}
\usepackage{lipsum}
\setcellgapes{2pt}
\usepackage{tikz}
\usetikzlibrary{arrows,shapes,positioning,shadows,trees}

\tikzset{
	basic/.style  = {draw, text width=7cm, rectangle,align=center},
	root/.style   = {basic, thin, align=center,
		fill=white!30},
	level 2/.style = {basic, thin,align=center, fill=white!60,
		text width=8em},
	level 3/.style = {basic, thin, align=left, fill=white!60, text width=6.5em}
}

\usepackage{comment}
\newcommand{\hlabel}{\phantomsection\label}

\hypersetup{allcolors = cyan,
anchorcolor=red,
    colorlinks=true,
    linkcolor=blue,
    urlcolor=cyan,
    pdfpagemode=FullScreen,
    }
\newcommand{\hptg}[2]{\hypertarget{#1}{#2}}

\usepackage[margin=2.5cm]{geometry}
\numberwithin{equation}{section}
\newcommand{\bb}{\mathbb}
\newcommand{\cc}{\mathcal}
\newcommand{\ccP}{\mathscr{P}}
\renewcommand{\ge}{\geqslant}
\renewcommand{\le}{\leqslant}
\renewcommand{\geq}{\geqslant}
\renewcommand{\leq}{\leqslant}
\renewcommand{\Re}{\textrm{Re}}
\renewcommand{\Im}{\textrm{Im}}
\newcommand{\hpt}[2]{#2}

\newcommand{\tred}{}
\newcommand{\tredd}{\textcolor{red}}
\newcommand{\tb}{}
\definecolor{gr}{rgb}{0,.7,0}

\newcommand{\supp}{\operatorname{Supp}}

\newcommand{\tr}{\operatorname{Tr}}

\newcommand{\nab}{\nabla}

\newcommand{\ov}{\overline}

\newcommand{\mf}{\mathfrak}
\newcommand{\bbm}{\mathbbm}
\newcommand{\ttt}{\mathtt}
\newcommand{\bmrm}[2][2]{\mathbf{#2}^{#1}}

\newcommand{\Span}{\mathrm{Span}}
\newcommand{\Range}{\operatorname{Range}}

\newcommand{\fou}{\mathcal{F}}
\newcommand{\fouh}{\mathcal{F}^{\mathbf{e}}}

\newcommand{\FsemiM}{\bm{P}^M[\bm{\psi}]}
\newcommand{\FsemiMd}{\widehat{\bm{P}}^M[\bm{\psi}]}
\newcommand{\AM}{\bm{A}^M}
\newcommand{\HMd}{\bm{H}^{-1}_{\bm{\psi},M}}
\newcommand{\Mpsi}{\mm^{\mathbf e}_{{\bm{\psi}}}}

\newcommand{\Jp}{J_{\psi}}
\newcommand{\Jpo}{J_{\psi}}
\newcommand{\m}{\mathfrak{m}}
\newcommand{\phin}{\varphi^{(y)}_n}
\newcommand{\sqpi}{\frac{1}{\sqrt{2\pi}}}
\newcommand{\Hpsi}{H_\psi}
\newcommand{\Hpsinv}{H^{-1}_\psi}
\newcommand{\Hpsid}{\widehat{H}_{\ov{\psi}}}
\newcommand{\mpsi}{m_{\Hpsi}}
\newcommand{\Wphip}{W_{\phi_+}\left(\frac{1}{2}-\i\xi\right)}
\newcommand{\Wphim}{W_{\phi_-}\left(\frac{1}{2}+\i\xi\right)}
\newcommand{\WBpsi}{W^{(\beta)}_\psi}
\newcommand{\Wphi}{\frac{W_{\phi_+}\left(\frac{1}{2}+\mathrm{i}\xi\right)}{W_{\phi_-}\left(\frac{1}{2}+\mathrm{i}\xi\right)}}

\newcommand{\fouho}{\cc{F}^{e}}
\newcommand{\C}{\mathbf{C}}
\newcommand{\B}{\mathbf{B}}
\newcommand{\D}{\mathbf{D}}
\renewcommand{\SS}{\mathbf{S}}
\newcommand{\A}{\mathbf{A}}
\newcommand{\eps}{\epsilon}
\newcommand{\LL}{\Lambda}
\newcommand{\heb}{\mf{h}_{\eps,\beta}}
\renewcommand{\i}{\mathrm{i}}
\newcommand{\LRe}{\bmrm{L}(\bb{R},e)}
\newcommand{\LLRe}{\bmrm{L}(\bb R^d,{\bf e})}
\newcommand{\NN}{\mathbf{N}}
\newcommand{\mm}{\mathbf{m}}
\newcommand{\f}{\mathbf{f}}
\newcommand{\g}{\mathbf{g}}
\renewcommand{\tau}{\uptau}
\newcommand{\ccA}{\mathscr{A}}
\newcommand{\cD}{\mathcal{D}}

\newcommand{\rH}{\mathrm{H}}
\newcommand{\LLReM}{\bmrm{L}(\bb R^d,\mathbf{e}_M)}

\renewcommand{\tredd}{}

\newtheorem{thm}{Theorem}[section]
\newtheorem{prop}[thm]{Proposition}
\newtheorem{lem}[thm]{Lemma}
\theoremstyle{definition}
\newtheorem{definition}[thm]{Definition}
\theoremstyle{assump}

\newtheorem{corr}[thm]{Corollary}
\theoremstyle{remark}
\newtheorem{rem}[thm]{Remark}
\title[WS orbit of (log)-self-similar Markov semigroups]{Weak similarity orbit of (log)-self-similar Markov semigroups on the Euclidean space}

\author{P.~Patie}\thanks{The authors are grateful to two anonymous referees for useful comments and a careful reading that lead to an improvement of the quality of the paper.}
\address{School of Operations Research and Information Engineering, Cornell University, Ithaca, NY 14853, USA.}
\email{pp396@cornell.edu}

\author{R.~Sarkar}
\address{Department of Mathematics, University of Connecticut, Storrs, CT 06269, USA.}
\email{rohan.sarkar@uconn.edu}
\date{}

\begin{document}

\maketitle
\begin{abstract}
We start by identifying a class  of pseudo-differential operators, generated by  the set of continuous negative definite functions,  that are in the weak similarity (WS) orbit of the self-adjoint log-Bessel operator on the Euclidean space. These WS relations turn out to be  useful to first characterize a core for each operator in this class, which enables us to show that they generate  a class, denoted by $\ccP$, of non-self-adjoint $\cc{C}_0$-contraction positive semigroups. Up to a homeomorphism, $\ccP$  includes, as fundamental objects in probability theory, the family of self-similar Markov semigroups on $\bb{R}_+^d$. Relying on the WS orbit, \hptg{C4}{we  characterize the nature of the spectrum of each element in $\ccP$ that is used in their spectral representation} which depends on analytical properties of the Bernstein-gamma functions defined from the associated  negative definite functions, and, it is  either the point, residual, approximate or continuous spectrum. We proceed by providing a spectral representation of each element in $\ccP$ which is expressed in terms of Fourier multiplier operators and valid, at least, on a dense domain of a natural weighted $\bmrm{L}$-space.   Surprisingly, the domain is the full Hilbert space when the spectrum is the residual one, something which seems to be noticed for the first time in the literature. We end up the paper by presenting a series of examples for which all spectral components are computed explicitly in terms of special functions or recently introduced power series.

\end{abstract}

\section{Introduction}
The spectral analysis of self-adjoint pseudo-differential operators (PDO) in Hilbert space  is by now well established and has offered many fascinating and deep ramifications in several branches of mathematics, see the monograph of Shubin \cite{Shubin}.
However,  although generic, the theory of non-self-adjoint ones is much less understood and unified, something which seems to be attributed to the variety of phenomena, such as the  instability of the spectrum under small perturbation that one encounters when studying such operators. We refer to the recent monograph of  Sj\"ostrand \cite{Sjos} for a thorough account on non-self-adjoint differential operators and a detailed study of their  spectral asymptotic.

In this  paper, we  develop an  in-depth and detailed analysis of the non-self-adjoint and non-local PDOs in the class \begin{equation*}
\mathscr{A}=\{\bm{A}^M_{\ttt{PDO}}[\bm{\psi}]; \: \bm{\psi}\in\NN^d_b(\bb R), M\in\mathrm{GL}_d(\bb{R})\} \end{equation*}
of densely defined operators in the weighted Hilbert space $\bmrm{L}(\bb{R}^d,{\bf e}_M)$, see \eqref{eq:defeb} for definition, and whose symbols take the form
\begin{equation} \hlabel{eq:symb}
\bbm{a}(\bm{x},\bm{\xi})=-\langle \hptg{C5}{e(-M^{-1}\bm{x})},\bm{\psi}(M^\top\bm{\xi})\rangle, \: \bm{x},\bm{\xi}\in\bb{R}^d 
\end{equation}
where $e(\bm{x})=(e^{x_1},\ldots, e^{x_d})$, $\langle.,.\rangle$ is the Euclidean inner product, $\bm{\psi}=(\psi_1,\ldots,\psi_d)\in\mathbf{N}_b^d(\bb{R})$, a vector of  continuous definite functions defined in  \eqref{eq:defNbd}, and $M \in {\rm{GL}}_d(\bb{R})$, the group of $d\times d$  invertible matrices. This includes the study of their domains, a classification scheme by weak similarity orbit, a notion that we introduce in Section \ref{sec:ws}, the generation of $\mathcal{C}_0$- positivity-preserving contraction semigroups, the spectral theory and representation of the latter,
and representation as integro-differential operators.

The first motivation  to investigate  this class of (non local) PDOs stems on their intimate connection with the family of self-similar Markov processes on $\bb{R}_+^d$, which have been intensively studied over the
last two decades from both theoretical and applied perspectives, see e.g.~\cite{Lamperti1972,Caballero2006,Patie-abs-08}.
These interests seem to be  attributed to their role played in limit theorems in probability theory.  This probably also \hptg{C7}{explains} the appearance of self-similar Markov processes in
many different studies, such as coalescence-fragmentation \cite{Bertoin_Frag}, random planar maps \cite{Bertoin-Igor-16}, and also in the study of fractional operators \cite{Patie-Simon,patie_zhao2017} to name but a few.  We shall prove that indeed the Dynkin generator of such processes, up to the homeomorphism $\bm{x}\mapsto e(\bm{x})$, coincide, on a core that we identify, with a pseudo-differential operator whose symbol is of the form \eqref{eq:symb} with $M$ the identity matrix.  The current work presents a very detailed spectral analysis of non-local and non-self-adjoint linear operators, for which few instances can be found in the literature.
We also  point out that in the recent paper \cite{MPS}, the authors introduce and develop in-depth analysis a class of continuous-time integer-valued Markov chains which are discrete self-similar with respect to some random discrete dilation operator.

We also aim at pursuing  the program initiated by the first author with M. Savov \cite{patiesavov1,patiesavov} and later, with other co-authors,  on the development of a spectral theory for non-self-adjoint Markov semigroups based on the concept of intertwining or related classification schemes, see  \cite{Choi-Patie, CP, CherPat, Patie-Miclo, MPS, PZS}.   From this perspective, this paper deepens the understanding of such ideas by considering on the one hand pseudo-differential operators combined with the Fourier operators to identify the intertwining relations, in the spirit of the work of Egorov \cite{Ego}. On the other hand, it also reveals that such an approach is efficient to deal with all different components of the spectrum, including the continuous, point, residual and approximate spectrum.

Let us now  describe  our main results and the strategy we have implemented to get them, and, we refer to the diagram \ref{fig:organization} for a visual description of the path we have followed. Our approach relies on the concept of weak similarity (WS) relation. More specifically, we say  that the two linear operators $P, Q$ have a  WS relation if
\begin{equation}\hlabel{eq:wso}
  P\Lambda =\Lambda Q\ \mbox{ on } \D(\Lambda )
\end{equation}
where   $\Lambda \in \mathscr{D}(H)$, i.e. it is a densely defined injective linear operator  with a dense range in a Hilbert space $H$, and we denote its domain by $\D(\Lambda )$,  see Section \ref{sec:ws} for a refined definition. In such case, we say that $P$ is in the WS orbit of $Q$. Since we shall show that, in our context, WS is an equivalence relation,  $Q$ may also be seen as a member of the WS orbit of $P$.

Coming back to our setting, we start  by showing, by a combination of Fourier and complex analysis,  that, under very mild condition, see Remark \ref{rem:cond} for a discussion,  the class of  symbols of the form \eqref{eq:symb} is in the WS orbit of the self-adjoint PDO whose symbol is given by
\begin{equation} \hlabel{eq:symbb}
\bbm{a}_0(\bm{x},\bm{\xi})=-\langle e(-\bm{x}),\bm{\psi}_0(\bm{\xi})\rangle, \: \bm{x},\bm{\xi}\in\bb{R}^d, 
\end{equation}
where $\bm{\psi}_0(\bm{\xi})=(\xi^2_1,\ldots,\xi^2_d)$. Let us point out  that in \cite{patiesavov}, in the one-dimensional case, relying on  original factorization of some random variables,   intertwining relations (this terminology means that $\Lambda$ in \eqref{eq:wso} is a bounded Markov operator) has been obtained between self-similar semigroups on $\bb{R}_+$ under some conditions on the negative definite functions $\psi$.  Although, by  the  isomorphism aforementioned, these relations could easily be transferred to our framework, the required conditions, which are  rather restrictive,  could not be relaxed. Moreover, an attempt to obtain weak similarity relations from the representation of the  infinitesimal generator as an integro-differential operator or Dynkin characteristic operator given by Lamperti \cite{Lamperti1972} also leads to some restrictive conditions such as, at least, the requirement of  analytical extension to some strip of the L\'evy-Khintchine exponent.

\hpt{R1}{These are the main reasons that have forced us to consider this question from a different and fresh perspective by starting instead} from the class of PDOs with symbols of the form \eqref{eq:symb} to first characterize the WS orbit of log-Bessel semigroups for almost all the class of log-Lamperti semigroups. However, this approach gives rise to several issues, such as the existence of a $\cc{C}_0$-contraction semigroup associated to this PDO's, the positivity preserving property of the latter as well as their relation to log-Lamperti semigroups. We emphasize that the latter is not a trivial exercise as, to the best of our knowledge, there does not exist an equivalent of Volkinski's formula for PDO's. We proceed by explaining the paths we follow to overcome these difficulties.

First, establishing a link between the PDO with symbol  \eqref{eq:symbb} and the Laplacian acting on the space of symmetric functions, and, resorting to  classical results from the theory of self-adjoint operators, we show that it generates the self-adjoint $\cc{C}_0$-contraction  semigroup, on  the weighted Hilbert space $\bmrm{L}(\bb{R}^d,{\mathbf{e}}_M)$, where ${\mathbf{e}}_M$ is defined in \eqref{eq:defeb}, of the log-squared Bessel process, which is also positivity preserving. Unfortunately, these classical  results do not carry over to the class of symbols of the form \eqref{eq:symb}.

To deal with the generation of a, in general non-self-adjoint, $\cc{C}_0$-contraction semigroup on the same  weighted $\bmrm{L}$-space, we first observe that the class of symbols of the form \eqref{eq:symb} is closed under conjugation allowing us to split the analysis of this class of PDO's into two. Then, using, in the one-dimensional case, the WS relation mentioned above  we manage to identify a core for one of these two classes as well as the dissipativity property which guarantees the existence of a $\cc{C}_0$-contraction semigroup. For the other remaining class, we resort to classical results from the theory of $\cc{C}_0$-contraction semigroup.
We also manage to lift the WS relations between the PDO's to the level of the $\cc{C}_0$-contractions semigroups.

Based on the WS relation and the well-known diagonalisation of the Bessel semigroup, see e.g.~\cite{MuS}, we proceed by giving, in Theorem \ref{th:spectral_decomp}, the spectral representation of the non-self-adjoint $\cc{C}_0$-contraction semigroups generated by these PDO's. It  is valid, at least, on a domain, defined in terms of objects related to the negative definite function $\bm{\psi}$, which is either the full or a dense subset of the Hilbert space $\bmrm{L}(\bb{R}^d,{\mathbf{e}}_M)$. The integral representation is expressed in terms of Fourier multiplier operators which are, themselves,  defined in terms of the so-called Bernstein-gamma functions associated to the Wiener-Hopf factors of the negative definite functions $\bm{\psi}$, see \eqref{eq:BG_receq} for definition. In this direction, we mention that spectral properties of some special instances  of self-similar Markov semigroups have already been studied in the literature. For instance, spectral expansion of self-similar semigroups corresponding to fractional derivative operators have been derived in \cite{patie_zhao2017}, the spectral representation of the transition density of stable processes on the half-line and related objects have  been  obtained in \cite{kK, kw,kw1, Mu}.

Next, to prove the positivity preserving property, thanks again to the characterization of the core by the WS relation, we show that each  PDO coincide with the Dynkin  operator of  a positive self-similar Markov process, see  \eqref{eq:dynk} for definition.  
The $d$-dimensional case is then deduced by means of a tensorization argument combined with a similarity transform.

We proceed by characterizing the nature of the spectrum which depends on the asymptotic behavior of the Bernstein-gamma functions, see \eqref{eq:BG_receq}, which we manage to reduce to easy-to-check properties on the Wiener-Hopf factors or, directly, on the negative definite functions. The class spans all parts of the spectrum including  the point, residual, approximate or continuous ones, which reveals the flexibility of the concept of WS in this context.  It is already worth pointing out that the spectral representation is valid on  the full Hilbert space (resp.~on a dense domain) when the spectrum is the residual one (resp.~point one), something which, to the best of our knowledge, seems to have been unnoticed in the literature.

The remaining part of the paper is organized as follows. After providing some general notations, we present our main results, including the WS orbit, the nature of the spectrum, the spectral representation  in Section \ref{sec:main}. Before providing the proofs of these results in Section \ref{sec:proof}, we recall, in Section \ref{sec:prelim} some substantial preliminary results and define some tools  that will be essential in the course of the proofs. The last Section is devoted to the detailed description of some new examples for which we express the spectral components in terms of known special functions or recently introduced power series generalizing the later. In Figure~\ref{fig:organization}, we provide a chart to indicate how the main results are connected to each other.

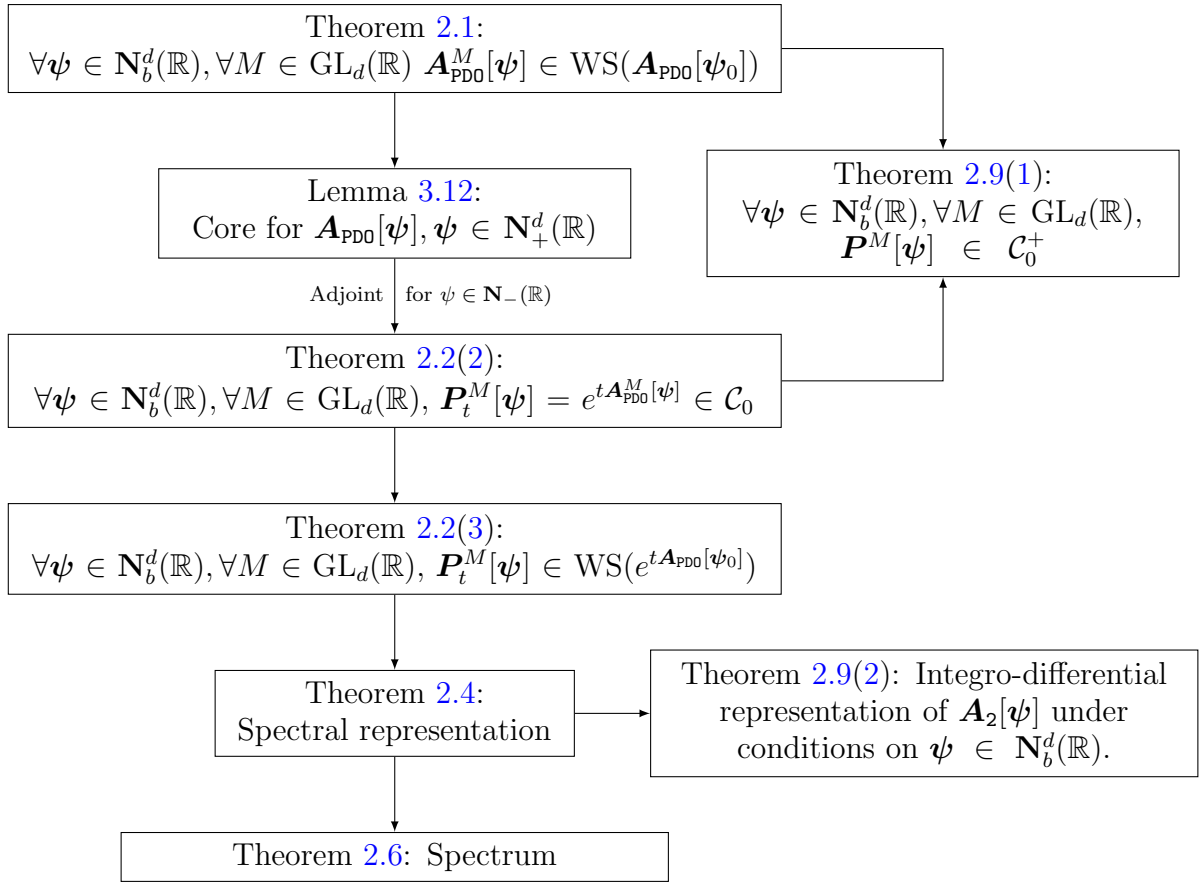
\begin{figure}[ht]
	\centering
	\begin{tikzpicture}[
		level 1/.style={sibling distance=20mm},
		edge from parent/.style={->,draw},
		>=latex]
		
		\node[minimum width=10.2cm, basic, text width=10cm, text centered, fill=white] (c1)  {Theorem~\ref{thm1:main_thm_1}:\\ $\forall\bm{\psi}\in\NN^d_b(\bb R), \forall M\in\mathrm{GL}_d(\bb R)$  $\bm{A}^M_{\ttt{PDO}}[\bm\psi]\in\mathrm{WS}(\bm{A}_{\ttt{PDO}}[\bm{\psi}_0])$};
		\node[basic, below=of c1, minimum width=6.2cm, basic, text width=6cm, text centered, fill=white] (c2) {Lemma~\ref{lem:generator_intertwining}: \\ Core for $\bm{A}_\ttt{PDO}[\bm\psi], \bm{\psi}\in\NN^d_+(\bb R)$};
		\node [basic,below=of c2, minimum width=10.2cm, basic, text width=10cm, text centered, fill=white]  (ws1) {Theorem~\ref{thm:semigroup_generation}\eqref{it:bijection}: \\ $\forall{\bm\psi}\in\NN^d_b(\bb R), \forall M\in\mathrm{GL}_d(\bb R)$, $\bm{P}^M_t[\bm\psi]=e^{t{\bm{A}}^M_{\ttt{PDO}}[\bm\psi]} \in  \cc{C}_0$}; 
		\node [basic,below=of ws1, minimum width=10.2cm, basic, text width=10cm, fill=white] (sg) {Theorem~\ref{thm:semigroup_generation}\eqref{it:ws_semigroups}: \\ $\forall{\bm\psi}\in\NN^d_b(\bb R), \forall M\in\mathrm{GL}_d(\bb R)$, $\bm{P}^M_t[\bm\psi] \in \mathrm{WS}(e^{t{\bm{A}}_{\ttt{PDO}}[\bm\psi_0]})$}; 
		\node[basic,text width=4.5cm, below= of sg] (sp1){Theorem~\ref{th:spectral_decomp}: \\ Spectral representation};
		\node[basic, below= of sp1] (sp2) {Theorem~\ref{thm:spectrum}: Spectrum};
		\node [basic,  minimum width=3.2cm, text width=6cm, right=of c2] (lamp) {Theorem~\ref{thm:core_multi_d}\eqref{it:feller_semigroup}: \\
$\forall{\bm\psi}\in\NN^d_b(\bb R), \forall M\in\mathrm{GL}_d(\bb R)$, \\ $\bm{P}^M[\bm\psi] \in \cc{C}^+_0$};
		\node [basic, right=of sp1] (core) {Theorem~\ref{thm:core_multi_d}\eqref{it:core}: Integro-differential representation of $\bm{A}_{\ttt{2}}[\bm\psi]$ under  conditions on $\bm{\psi}\in\NN^d_b(\bb R)$.};

		\draw [->] (c1) -- (c2);
		\draw[->]  (c2) -- node[left]{\tiny{Adjoint}}node[right]{\tiny{for ${\psi}\in\NN_-(\bb R)$}} (ws1) ;
		\draw[->] (ws1) -- (sg);
		\draw[->] (sg) -- (sp1);
		\draw[->] (sp1) -- (core);
		\draw[->] (ws1.east) -| (lamp.south);
		\draw[->] (sp1) -- (sp2);
		\draw[->] (c1.east) -| (lamp.north);
	\end{tikzpicture}
	\caption{User guide on the connection between the main results}\label{fig:organization}
\end{figure}

\subsection{Notations and preliminaries}
\subsubsection{Functional Spaces and operators}  Throughout this paper, we denote by $\bmrm{L}(\bb{R}^d,\bm{\mu})$ the class of all square integrable functions $\f$ with respect to the non-negative Borel measure $\bm{\mu}$ on $\bb{R}^d$ endowed with the natural inner-product $\langle \f,\g\rangle_{\bm{\mu}}=\int_{\bb{R}^d} \f(\bm{x})\overline{\g(\bm{x})}\bm{\mu}(d\bm{x})$. In particular, when $\bm{\mu}(d\bm{x})=e^{\langle \bm{x},\bm{1}\rangle}d\bm{x}$ (resp.~$\bm{\mu}(d\bm x)=e(x)dx=e^x dx, x \in \bb{R}$) we denote the Hilbert space by $\bmrm{L}(\bb{R}^d,\mathbf{e})$ (resp.~$\bmrm{L}(\bb{R},e)$). Throughout the paper, we use the notation ``$\perp$'' to indicate orthogonality relation on the Hilbert spaces.

For a univariate function $f$ and a $d$-dimensional vector $\bm{x}=(x_1,x_2,\ldots, x_d)\in\bb R^d$, we denote $f(\bm{x})=(f(x_1),f(x_2),\ldots, f(x_d))$. For any $k\in\bb{N}\cup \{0\}$, we denote by $\C^k(\bb{R})$ the class of all functions defined on $\bb{R}$ that are $k$-times continuously differentiable. When $k=0$, we write $\C^0(\bb{R})$ simply as $\C(\bb{R})$. By $\C_0(\bb{R})$, we denote the class of all continuous functions defined on $\bb{R}$ that vanishes at infinity. We also consider the spaces $\C^k_b(\bb{R})$, the class of all $k$-times continuously differentiable and bounded functions with bounded derivatives. Finally, we denote by $\C^\infty_c(\bb{R})$, the class of all compactly supported smooth functions on $\bb{R}$.  When the functions are assumed to be smooth with all the derivatives vanishing at infinity, we denote them by $\C^\infty_0(\bb{R})$.

For any two Banach spaces $\mathrm{X}_1, \mathrm{X}_2$, we denote by $\mathscr{B}(\mathrm{X}_1, \mathrm{X}_2)$  the space of all bounded linear operators defined from $\mathrm{X}_1$ to $\mathrm{X}_2$. When $\mathrm{X}_1=\mathrm{X}_2=\mathrm{X}$, we simply denote the unital algebra  by $\mathscr{B}(\mathrm{X})$. For any $T\in\mathscr B(\mathrm{H}_1,\mathrm{H}_2)$ where $\mathrm{H}_1,\mathrm{H}_2$ are Hilbert spaces, the adjoint of $T$ is denoted by $\widehat{T}$, which is an element in $\mathscr B(\mathrm{H}_2, \mathrm{H}_1)$.  Denoting $\D(\Lambda)$ the domain of an operator $\Lambda$, we also use the set \[\mathscr{D}_0(\mathrm{X}_1,\mathrm{X}_2)=\{\Lambda: \D(\Lambda)\subset\mathrm{X}_1 \mapsto  \mathrm{X}_2 \textrm{ linear, densely defined, injective with dense range} \}\] and we write $\mathscr{D}(\mathrm{X}_1,\mathrm{X}_2)$ to denote the set of all densely defined operators from $\mathrm{X}_1$ to $\mathrm{X}_2$. When $\mathrm{X}_1=\mathrm{X}_2=\mathrm{X}$, we use the notations $\mathscr{D}_0({\rm X})$ and $\mathscr{D}({\rm X})$ respectively.

Throughout the paper, we use the notation $\cc{C}_0(\bmrm{L}(\bb{R}^d,{\bm \mu}))$ (resp.~$\cc{C}^+_0(\bmrm{L}(\bb{R}^d,{\bm \mu})))$ to denote the set of (resp.~positivity-preserving) strongly continuous semigroups on the Hilbert space $\bmrm{L}(\bb{R}^d,{\bm \mu})$ where ${\bm \mu}$ is $\sigma$-finite measure on $\bb{R}^d$. We recall that $(\bm P_t)_{t\geq0} \in \cc{C}_0(\bmrm{L}(\bb{R}^d,{\bm \mu}))$ is said to be positivity-preserving if for all $t\geq0$, $\bm P_t\mathbf f\ge 0$ a.s. whenever $\mathbf f\ge 0$ a.s.

\subsubsection{Complex plane, strips and analytic functions}
We use $\bb{C}$ to denote the complex plane. For any $-\infty\le a<b\le \infty$, we use $\bb{C}_{(a,b)}$ to denote the vertical strip $\{z\in\bb{C}; \ \Re(z)\in (a,b)\}$. For the horizontal strip we use the notation $\bb{S}_{(a,b)}$. In the same manner, we use the notations $\bb{C}_{[a,b]}$, $\bb{S}_{[a,b]}$ to denote the strips including their boundaries. Again, for $-\infty\le a<b\le\infty$, we use $\A_{(a,b)}$ to denote the class of functions that are analytic on the strip $\bb{S}_{(a,b)}$. We use $\A_{[a,b]}$ to denote the class of analytic functions on $\bb{S}_{(a,b)}$ that extend continuously to the boundary of the strip. We also use the classical notation ${\rm{i}}^2=-1$.

\subsubsection{Fourier transform in $\bmrm{L}(\bb{R}^d,\mathbf{e})$ and multipliers}
Let us denote by $\fouh$ the shifted Fourier transform   
defined for functions $\f$ such that $\sqrt{\mathbf{e}}\f$ is integrable as follows
\begin{align}
	\fouh_{\f}(\bm{\xi})=(2\pi)^{-\frac{d}{2}}\int_{\bb{R}^d}e^{-{\rm{i}}\langle\bm{\xi}, \bm{x}\rangle}\sqrt{\mathbf{e}(\bm{x})}\f(\bm{x})d\bm{x}=\fou_{\!\!\sqrt{\mathbf{e}}\f}(\bm{\xi}), \: \bm{\xi}\in\bb{R}^d,
\end{align}
where $\fou$ is the usual Fourier transform and we recall that
$\mathbf{e}(\bm{x})=e^{\langle \bm{x},\bm{1}\rangle}, \: \bm{x}\in\bb{R}^d, \hptg{C10}{\bm{1}=(1,1,\ldots, 1)}$.  $\fouh$ (resp.~$\fou$) is  an unitary operator from $\bmrm{L}(\bb{R}^d,\mathbf{e})$ (resp.~$\bmrm{L}(\bb{R}^d)$) into $\bmrm{L}(\bb{R}^d)$, which explains the notation.  
\hptg{C11}{
Next, for a function $\mathbf{m}^{\mathbf{e}}_{\bm{\LL}}:\bb{R}^d\to\bb{C}$, we define the \emph{shifted Fourier  operator} $\bm{\LL}:\bmrm{L}(\bb{R}^d,\mathbf{e})\to\bmrm{L}(\bb{R}^d,\mathbf{e})$ with multiplier $\mm^{\mathbf{e}}_{\bm{\LL}}$ by
	\begin{equation}
		\fouh_{\bm{\LL} \f}(\bm{\xi})=\mm^{\mathbf{e}}_{\bm{\LL}}\left(\bm{\xi}\right)\fouh_{\f}(\bm{\xi}), \ \bm{\xi}\in\bb{R}^d.
	\end{equation}
	When the multiplier $\mm^{\mathbf{e}}_{\bm{\LL}}$ is analytic in the cylinder \hpt{R2}{$\bb{S}^d_{(-\frac12,\frac12)}$}, we will  use the notation $\mm_{\bm{\LL}}\left(\bm{\xi}+\frac{{\rm{i}}}{2}\bm{1}\right)=\mm^{\mathbf{e}}_{\bm{\LL}}\left(\bm{\xi}\right)$, which corresponds to the (shifted) multiplier associated to the classical Fourier transform.}
	$\bm{\LL}$ is a densely defined operator with domain
	\begin{equation}
		\D(\bm{\LL})=\left\{\f\in\bmrm{L}(\bb{R}^d,\mathbf{e}); \: \bm{\xi} \mapsto \mm^{\mathbf{e}}_{\bm{\LL}}\left(\bm{\xi}\right)\fouh_ {\f}(\bm{\xi})\in\bmrm{L}(\bb{R}^d)\right\}.
	\end{equation}
\noindent
When $\mm^{\mathbf{e}}_{\bm{\LL}}$ is a.e.~non-zero, $\bm{\LL}\in \mathscr{D}_0(\bmrm{L}(\bb R^d,{\bf e}))$. In  Subsection~\ref{ss:fourier_mult}, we will discuss properties of these operators in more details.
\subsubsection{Boundedness of ratios}
For two functions $f,g$ defined on the complex plane, we use the following notations
\begin{eqnarray*}
	f&\asymp &g  \text{ means that $\exists  c>0$ such that  $c^{-1}\le\frac{f}{g}\le c$,} \\
	f&\overset{a}{\sim}& g \text{ means that $\lim_{x\to a}\frac{f(x)}{g(x)}=1$ for some $a\in[0,\infty]$}, \\
	f(x)&\overset{a}{=}&\mathrm{O}(g(x)) \text{ means that $\limsup_{x\to a}\left|\frac{f(x)}{g(x)}\right|<\infty$},\\
	f(x)&\overset{a}{=}&\mathrm{o}(g(x)) \text{ means that $\lim_{x\to a}\frac{f(x)}{g(x)}=0$}.
\end{eqnarray*}

\subsubsection{Continuous negative definite functions, Wiener-Hopf factorization and Bernstein-gamma functions}
Let $\mathbf{N}(\bb{R})$  be the class of all continuous negative definite functions defined on $\bb{R}$, i.e.~for any $\psi\in\mathbf{N}(\bb{R})$  there exists a unique quadruple  $(\psi(0),\sigma^2,\ttt{b},\mu)$ such that
\begin{align} \hlabel{eq:defLKe}
\psi(\xi)=\psi(0)-\mathrm{i}\ttt{b}\xi+\sigma^2\xi^2+\int_{\bb{R}}\left(1-e^{{\rm{i}}\xi   y}+{\rm{i}}y   \xi\mathbbm{1}_{\{|y|\le 1\}}\right)\mu(dy), \: \xi \in \bb{R},
\end{align}
where $\psi(0)\ge 0, \ttt{b}\in\bb{R},\sigma^2\ge 0$ and $\mu$ is a non-negative Radon measure such that $\int_{\bb{R}}(y^2\wedge 1)\mu(dy)<\infty$, see e.g.~\cite{Jacob} and \cite[Chap.~4]{SchillingSongVondracek10}. It is  well-known that the set $\mathbf{N}(\bb{R})$ is stable by conjugation, that is  $\psi\in\mathbf{N}(\bb{R})$ if and only if its conjugate function $\ov{\psi}\in \mathbf{N}(\bb{R})$. Moreover,  each $\psi\in\mathbf{N}(\bb{R})$ admits  the following analytic Wiener-Hopf factorization
\begin{align} \hlabel{eq:WH}
\psi(\xi)=\phi_+(-{\rm{i}}\xi  )\phi_-({\rm{i}}\xi  ), \ \ \xi\in\bb{R},
\end{align}
where $\phi_+, \phi_-\in\B$,  the set of Bernstein functions,  which are functions $\phi :[0,\infty) \mapsto [0,\infty)$  of the form
\begin{align}\hlabel{eq:phi_def}
\phi(u)=\phi(0)+\ttt{d}u+\int_0^\infty (1-e^{-uy})\nu(dy), \ \ \ u\geq0,
\end{align}
with $\phi(0)\ge 0, \ttt{d}\ge 0$ and $\nu$ is a non-negative Radon measure such that $\int_0^\infty (1\wedge y) \nu(dy)<\infty.$  Any $\phi\in\B$ can be extended analytically on $\bb{C}_{(0,\infty)}$ that remains continuous on ${\rm{i}}\bb{R}$.
The Wiener-Hopf factors $\phi_+$ and $\phi_-$ are the Laplace exponents of the so-called ascending and descending ladder height processes corresponding to the L\'evy process associated with $\psi$. Also, existence of jump measures in $\phi_+$ (resp.~$\phi_-$) indicates existence of positive (resp.~negative) jumps of the L\'evy process determined by \eqref{eq:levy_kh}.
Note that if $\psi$ is of the form \eqref{eq:defLKe}, it follows immediately that its conjugate admits the factorization \[ \ov{\psi}(\xi)=\phi_-(-\mathrm{i}\xi)\phi_+(\mathrm{i}\xi), \ \ \xi\in\bb{R}. \]
 For an excellent account  on the Wiener-Hopf factorization of L\'evy processes, we refer to   \cite[Chapter VI]{bertoin1996levy} and  \cite[Chapter 6]{kyprianou2014fluctuations}.

 Next, for any $\phi\in\B$, let us denote by $W_{\phi}$ the so-called Bernstein-gamma function associated to $\phi$, that is the  unique positive definite function, i.e.~the Mellin transform of a positive random variable, satisfying the functional equation
\begin{equation} \hlabel{eq:BG_receq}
W_{\phi}(z+1)=\phi(z)W_\phi(z), \ \  z\in\bb{C}_{(0,\infty)}, \textrm{ and } W_\phi(1)=1.
\end{equation}
We refer to \cite{patie2018} for a thorough account on this function and in particular from
\cite[Sec.~4]{patie2018}, one gets that
\begin{equation}\hlabel{eq:W_Zerofree}
W_{\phi} \mbox{ is analytic  and zero-free (at least) on }  \bb{C}_{[0,\infty)}.
\end{equation}

\subsection{The notion of weak similarity} \hlabel{sec:ws}
The cornerstone of our  approach to  obtain a detailed spectral theory for the semigroups generated by $\AM_\ttt{PDO}[\bm{\psi}]$ stems on its membership in the  weak similarity orbit of the self-adjoint PDO $A_\ttt{PDO}[\bm\psi_0]$ given by \eqref{eq:symbb}, a concept that we now define. For a Hilbert space $\rH$, let us first recall that
 \[\mathscr{D}_0(\rH)=\{\Lambda: \D(\Lambda)\subseteq \rH \mapsto  \rH \textrm{ linear, densely defined, closed, injective with dense range} \}\]
 and note that it is not, in general, a group of operators. 
 \begin{definition}
For two Hilbert spaces $\rH_1,\rH_2$ and $A\in\mathscr{D}(\rH_1), B\in\mathscr{D}(\rH_2)$, we say that \emph{$A$ is weakly similar to $B$} if there exists $\Lambda \in \mathscr{D}_0(\rH_2,\rH_1)$ such that
\begin{align}\hlabel{eq:ws0}
	A\Lambda =\Lambda B  \ \mbox{ on } \ \cD
\end{align}
with $\cc{D},\LL(\cD)$ being dense in $\rH_2, \rH_1$ respectively, and
\begin{align}\hlabel{eq:domains}
	\cD\subset\D(B)\cap \D(\LL), \ \LL(\cD)\subset\D(A) \ \text{ and } \ B(\cD)\subset\D(\LL).
\end{align}
$\Lambda$ is called a \emph{weak similarity operator}.
\end{definition}
From now on we implicitly assume \eqref{eq:domains} whenever we have identities of the type \eqref{eq:ws0}. In the above definition we note that when $A,B$ are bounded and $\cD$ is a core for $\LL$, the identity \eqref{eq:ws0} extends to $\D(\Lambda)$.

\begin{definition}
	Let $\rH_1,\rH_2$ be Hilbert spaces and $B\in\mathscr{D}(\rH_2)$. We define the \emph{weak similarity orbit} of $B$ in $\mathscr{D}\subset\mathscr{D}(\rH_1)$ as the set of all operators $A\in\mathscr{D}$ such that $A$ is weakly similar to $B$. We write
	\begin{align*}
		\mathrm{WS}_{\mathscr{D}}(B)=\{A\in\mathscr{D}; \: \text{$A$ is weakly similar to $B$}\}.
	\end{align*}
When $\mathscr{D}=\mathscr{D}({\rm H}_1)$, we simply denote $\mathrm{WS}(B)=\mathrm{WS}_{\mathscr{D}}(B)$.
\end{definition}

\subsection{Standing assumptions}
We now set up the assumptions and notation that will be in force throughout the paper, and we refer to the remarks below for discussions about their generality.
Let
\begin{align}\hlabel{N_+}
 \mathbf{N}_+(\bb{R})=
&\left\{\psi\in\mathbf{N}(\bb{R});  \ \liminf_{|\xi|\to\infty}\frac{|\phi_+({\rm{i}}\xi  )|}{|\xi|^{-\kappa}}>0 \ \mbox{for some $\kappa>0$},  \ \frac{W_{\phi_+}(\frac{1}{2}+{\rm{i}}\xi  )}{W_{\phi_-}(\frac{1}{2}+{\rm{i}}\xi  )}={\rm{O}}(1)\right\}
 \end{align}
\begin{align}\hlabel{N_-}
\mathbf{N}_-(\bb{R})=\left\{\psi\in\mathbf{N}(\bb{R});  \ \ov{\psi}\in\mathbf{N}_+(\bb{R})\right\}
\end{align} and
\begin{equation} \label{eq:defNb}\mathbf{N}_b(\bb{R})=\mathbf{N}_+(\bb{R})\cup\mathbf{N}_-(\bb{R}).
\end{equation}
Finally, for any $d\in \mathbb{N}, d\geq2,$ we write  $\bm{\psi}=(\psi_1,\ldots,\psi_d)\in\mathbf{N}^d_b(\bb{R})$ where
\begin{equation} \label{eq:defNbd}
\mathbf{N}^d_b(\bb{R})=\mathbf{N}_b(\bb{R})\times \ldots \times \mathbf{N}_b(\bb{R}) \quad (d \textrm{ times})
\end{equation}
and, for any $k=1,\ldots,d$, $(\psi_k(0),\sigma_k^2,\ttt{b}_k,\mu_k)$ is the characteristic quadruplet of $\psi_k$, see \eqref{eq:defLKe}.
  \begin{rem}
Since for all $\phi\in\B$, the conjugate of $W_\phi$ is such that $\ov{W}_{\!\phi}(z)=W_\phi(\ov{z}), z\in\bb{C}_{(0,\infty)},$  we note that if $\psi\in\mathbf{N}_-(\bb{R})$ then  \[ \liminf\limits_{|\xi|\to\infty}{{|\xi|^{\kappa}}|\phi_-({\rm{i}}\xi  )|}>0 \ \mbox{for some $\kappa>0$ and }   \ \frac{W_{\phi_-}(\frac{1}{2}+{\rm{i}}\xi  )}{W_{\phi_+}(\frac{1}{2}+{\rm{i}}\xi  )}={\rm{O}}(1), \]
justifying the notation.
\end{rem}
\begin{rem}
We point out that, for $\phi\in\B$, the condition $\liminf\limits_{|\xi|\to\infty}|\xi|^\kappa|\phi({\rm{i}}\xi  )|>0$ has been termed as the \emph{weak non-lattice} property in \cite[(2.29)]{patie2018}, and in particular, it implies that
\begin{equation}\hlabel{eq:zerofree}
  \phi \ \mbox{is zero-free on ${\rm{i}}\bb{R}\setminus\{0\}$}.
\end{equation}
In this paper, the weak non-lattice property is used to ensure the moderate growth of certain functions along the lines $\Re(z)=a$, $0\le a\le1$, which plays an important role in the proof of our results. We refer to Section \ref{bernstein} for more details.
\end{rem}
\begin{rem}\label{rem:cond}
Although the weak non-lattice property disregards only very specific and even degenerate cases,  it seems to be the main restriction in the definition of the set $\mathbf{N}_b(\bb{R})$ as we believe that for any $\psi \in \mathbf{N}(\bb{R})$, the ratio of its associated Bernstein-gamma functions is either bounded from above or below.
\end{rem}
 \begin{rem}
 We also emphasize that, in many instances, the sets $\mathbf{N}_+(\bb{R})$ and $\mathbf{N}_-(\bb{R})$ can be expressed in terms of the characteristic quadruple of the negative definite function $\psi$ and/or in terms of the parameters defining its  Wiener-Hopf factors. For more details, we refer  to Proposition~\ref{prop:ratio_criterion} below.
 \end{rem}


\section{Main results}\hlabel{sec:main}

\subsection{The weak similarity orbits $\mathrm{WS}_{\ccA_M}$}\hlabel{ss:WS_PDO}
For any $\bm{\psi}=(\psi_1,\ldots,\psi_d)\in\mathbf{N}^d_b(\bb{R})$ and  $M\in\mathrm{GL}_d(\bb{R})$, the group of $d\times d$ invertible matrices, and, writing
\begin{equation}\hlabel{eq:defeb}
	{\mathbf{e}}_M(\bm{x})={\mathbf{e}}(M^{-1}\bm{x})=e^{\langle M^{-1}\bm{x},\bm{1}\rangle}, \quad \bm{x} \in \bb{R}^d,
\end{equation}
we consider the  pseudo-differential operator
\begin{equation*}
  \bm{A}^M_{\ttt{PDO}}[\bm{\psi}] : \D(\bm{A}^M_{\ttt{PDO}}[\bm{\psi}]) \subset \bmrm{L}(\bb{R}^d,{\bf e}_M) \to \bmrm{L}(\bb{R}^d,{\bf e}_M)
\end{equation*}
with symbol $\bbm{a}(\bm{x},\bm{\xi})=-\langle e(-M^{-1}\bm x),\bm{\psi}(M^\top\bm{\xi})\rangle, \ \bm{x},\bm{\xi}\in\bb R^d$. That is, for $\f\in\D(\bm{A}^M_{\ttt{PDO}}[\bm{\psi}])$,
\begin{align*}
	\bm{A}^M_{\ttt{PDO}}[\bm{\psi}]\f(\bm x)=(2\pi)^{-\frac{d}{2}}\int_{\bb{R}^d} e^{\i\langle \bm x, \bm \xi\rangle} \bbm{a}(\bm{x},\bm{\xi})\fou_{\f}(\bm\xi) d\bm{\xi}
\end{align*}
When  $M=\mathrm{Id}$, we simply write $\bm{A}_\ttt{PDO}[\bm\psi]=\AM_\ttt{PDO}[\bm\psi]$ and $\bmrm{L}(\bb{R}^d,{\bf e})=\bmrm{L}(\bb{R}^d,{\bf e}_M)$. Let us also define
\begin{equation*}
 \cc{D}(\bb R) = \{f\in \LRe;\: \cc{F}_f \textrm{ is entire and satisfies \eqref{eq:udec}}\}
\end{equation*}
\begin{equation} \label{eq:udec}
\hpt{R4}{\textrm{ For any compact } C \subset \bb{R},\: \sup_{a \in C}|\cc{F}_f(\xi+\i a)|=\mathrm{O}(e^{-(\frac{\pi}{2}+\eps)|\xi|}),\: \epsilon>0,}
	\end{equation}
and we write $\bm{\cD}(\bb R^d)=\otimes_{k=1}^d\cD(\bb R)$. We are now ready to state our first main result.
	\begin{thm}\hlabel{thm1:main_thm_1} For any $M\in\mathrm{GL}_d(\bb R)$, writing $ \mathscr{A}_M=\{\bm{A}^M_{\ttt{PDO}}[\bm{\psi}]; \: \bm{\psi}\in\NN^d_b(\bb R)\}$, we have
			\begin{align*}
				\mathrm{WS}_{\ccA_M}(\bm{A}_\ttt{PDO}[\bm\psi_0])=\ccA_M
			\end{align*}
where $\bm{\psi}_0(\bm{\xi})=(\xi^2_1,\ldots,\xi^2_d)$. More specifically, for any $\bm{\psi}\in\NN^d_b(\bb R)$ and $M\in\mathrm{GL}_d(\bb{R})$, we have
\begin{align} \hlabel{eq:wsob}
	\AM_\ttt{PDO}[\bm\psi]\bm{\LL}_{\bm\psi,M}=\bm{\LL}_{\bm\psi,M}\bm{A}_\ttt{PDO}[\bm\psi_0] \ \text{ on } \ \bm{\cc{D}}(\bb R^d)
\end{align}
where $\bm{\Lambda}_{\bm \psi,M}:\D(\bm{\Lambda}_{\bm \psi,M}) \subset \LLRe\to \bmrm{L}(\bb R^d,{\bf e}_M)$ is such that $\bm{\Lambda}_{\bm\psi,M}{\bf f}(\bm x)=\bm{\Lambda_\psi}{\bf f}(M\bm x)$, and $\bm{\Lambda}_{\bm \psi}$ is a shifted Fourier multiplier operator associated to
\begin{align*}
	\mathbf{m}^{\mathbf{e}}_{\bm{\Lambda}_{\bm\psi}}(\bm{\xi})=\prod_{k=1}^d \frac{W_{\phi_{+,k}}(\frac{1}{2}-\i\xi_k)}{W_{\phi_{-,k}}(\frac{1}{2}+\i\xi_k)}\frac{\Gamma(\frac{1}{2}+\i\xi_k)}{\Gamma(\frac{1}{2}-\i\xi_k)}.
\end{align*}
Moreover, for any $\bm{\psi}\in\NN^d_+(\bb R)$, $\bm{\LL}_{\bm\psi,M}\in\mathscr{B}(\bmrm{L}(\bb{R}^d,{\bf e}),\bmrm{L}(\bb{R}^d,{\bf e}_M))$.
	\end{thm}
This theorem is proved in Subsection~\ref{ss:main_thm_1}.
\subsection{Generation of the set  $\ccP$ of $\mathcal{C}_0$-contractions semigroups and $\mathrm{WS}_\ccP$}
Let us  introduce the linear differential operator
\begin{eqnarray}\hlabel{eq:fourier_int_0}
	\bm{A}_0\f(\bm{x})&=&\tr\left(\Sigma(\bm{x})\nab^2 \f(\bm{x})\right), \quad  \bm{x} \in \bb{R}^d,
\end{eqnarray}
where $\tr$ and $\nab$ stand for the trace and  the gradient respectively, and, $\Sigma=\Sigma_{\bm{1}}$ with, for any $(\sigma_1,\ldots, \sigma_d)\in [0,\infty)^d$ and writing $\bm{x}=(x_1,\ldots,x_d)$, the diagonal matrix $\Sigma_{\bm{\sigma}}(\bm{x})$ is defined as
\begin{align} \hlabel{eq:defQ}
	\Sigma_{\bm{\sigma}}(\bm{x})=\textrm{diag}(
		\sigma_1e^{-x_1}, \ldots, \sigma_de^{-x_d}).
\end{align}
\begin{thm}\hlabel{thm:semigroup_generation}
		\begin{enumerate}
\item \hlabel{it:bessel_generation}
The closure of  $(\bm{A}_\ttt{PDO}[\bm\psi_0],\C^\infty_c(\bb R^d))$ in $\LLRe$, denoted by $\bm{A}_2[\bm\psi_0]$, generates a self-adjoint $\cc{C}_0$-semigroup $\bm{Q} \in \mathcal{C}^+_0(\bmrm{L}(\bb{R}^d,{\mathbf{e}}))$.
More specifically, $\bm{Q}$ restricted to  $\C_0(\bb R^d)\cap \LLRe$  extends to the Feller semigroup of the $d$-dimensional log-Bessel process of dimension 2 whose generator is the closure of $(\bm{A}_0, \C^\infty_c(\bb R^d))$.
	\item \hlabel{it:bijection}
For any $M\in\mathrm{GL}_d(\bb R)$, there exists a one-to-one mapping between the set $\ccA_M$ and the set
\[ \ccP_M=\left\{\FsemiM; \: \bm{\psi}\in\mathbf{N}^d_b(\bb{R})\right\} \subset \mathcal{C}_0(\bmrm{L}(\bb{R}^d,{\mathbf{e}}_M))\]
such that $\AM_\ttt{2}[\bm\psi]=\AM_\ttt{PDO}[\bm\psi]$ on a dense subset, where  $\AM_\ttt{2}[\bm\psi]$ is the  $\LLReM$-generator of $\FsemiM$.

More specifically,  for any ${\psi}\in \NN_+(\bb R)$, $A_\ttt{2}[\psi]$ is the closure of $(A_\ttt{PDO}[\psi],\Lambda_{\psi}(\C^\infty_c(\bb R))$, where $\Lambda_{\psi}\in  \mathscr{B}(\bmrm{L}(\bb{R},{e}))$ is defined in \eqref{eq:wsob}, and, otherwise, for any ${\psi}\in \NN_b(\bb R)\setminus \NN_+(\bb R)$, $\widehat{{P}}[{\psi}]={P}[{\ov{{\psi}}}]$. The general case is then obtained by tensorization and isomorphism of Hilbert spaces.
	\item \hlabel{it:ws_semigroups} 
For any $M\in\mathrm{GL}_d(\bb R)$, we have
\hptg{C16}{
	\begin{equation} \hlabel{eq:ws}
		\mathrm{WS}_{\ccP_M}(\bm{Q})=\ccP_M.      
	\end{equation}}
In fact, the set of weak similarity operators is a multiplicative group which entails that $\mathrm{WS}_{\ccP}$ forms an equivalence class, where $\ccP=\bigcup_{M\in\mathrm{GL}_d(\bb R)}\ccP_M$. 
			\item \hlabel{it:thm2adj} For any $\bm\psi\in\NN^d_b(\bb R)$, we have $\FsemiMd=\bm{P}^{M}[\ov{\bm{\psi}}]\in \ccP$, where $\FsemiMd$ stands for  the $\bmrm{L}(\bb{R}^d,{\mathbf{e}}_M)$-adjoint of $\FsemiM$. Thus, $\FsemiM$ is self-adjoint in $\bmrm{L}(\bb{R}^d,{\mathbf{e}}_M)$ if and only if $\ov{\bm{\psi}}=\bm{\psi}$. \vspace{.2cm}
		\end{enumerate}
	
	\end{thm}
This theorem is proved in Subsection~\ref{ss:pf_thm:semigroup_generation}.
\begin{rem}
	When $d=1$, $M=\rm{Id}$, the identity matrix, writing simply $P[\psi]=\bm{P}^{\rm{Id}}[\bm{\psi}]$, the identity \eqref{eq:ws} reads as follows
	\begin{align}
		{\rm WS}_\ccP(Q)=\{P[\psi]; \: \psi\in\NN_b(\bb R)\}
	\end{align}
	where $Q$ is the semigroup on $\bmrm{L}(\bb{R},e)$, $e(x)=e^x, x \in \bb{R}$, of the log-squared Bessel process of dimension $2$. Moreover, $P[\psi]$ is the $\bmrm{L}(\bb{R},e)$-extension of the log-self-similar Feller semigroup as reviewed in Section \ref{sec:lamp}.
\end{rem}

\subsection{Spectral representation of the class $\mathscr{P}$}
In this part, we use  the weak similarity relation  combined with the fact that $\bm{Q}$ is diagonalisable as it is self-adjoint in $\LLRe$, to obtain the spectral decomposition of the semigroups in $\ccP$. 
To this end, we denote the multiplication semigroup on $\bmrm{L}(\bb{R}^d,\mathbf{e})$ by $(\bm{e}_t)_{t\geq 0}$, i.e.~for any $t\geq0$,
\begin{align}
	\bm{e}_t \f(\bm{y})={e}^{-t\langle e(-\bm{y}),\bm{1}\rangle}\f(\bm{y}), \;\; \bm{y}\in\bb{R}^d, 
\end{align}
where we recall that  $e(-\bm{y})=\left(e^{-y_1}, \ldots, e^{-y_d}\right)$.
We are now ready to state the following.
\begin{thm}\hlabel{th:spectral_decomp}
	Let $\bm{P}^M[\bm{\psi}] \in \ccP_M$ for some $M\in\mathrm{GL}_d(\bb R)$ and $\bm{\psi}=(\psi_1,\ldots,\psi_d)\in\mathbf{N}^d_b(\bb{R})$ such that $\psi_k(\xi)=\phi_{+,k}(-\mathrm{i}\xi)\phi_{-,k}(\mathrm{i}\xi)$ for all $1\le k\le d$ and $\xi\in\bb{R}$. Then, for all $t\ge 0$, 
		\begin{align}\hlabel{eq:spectral_decomp}
			\bm{P}^M_t[\bm{\psi}] =\bm{H}_{\bm{\psi},M}\bm{e}_t \bm{H}^{-1}_{\bm{\psi},M} \ \mbox{ on } \ \D(\bm{H}^{-1}_{\bm{\psi},M})
		\end{align}
		where  $\D(\bm{H}^{-1}_{\bm{\psi},M})=\Range(\bm{H}_{\bm{\psi},M})$ is dense in $\bmrm{L}(\bb{R}^d,\mathbf{e}), \
		\bm{H}_{\bm{\psi},M}\f(\bm x)=\bm{\Hpsi}\f(M \bm x)$ and $\bm{H}_{\bm{\psi}}$ being the shifted Fourier  operator with  multiplier denoted by $\mm^{\mathbf{e}}_{\bm \psi}(\bm \xi)$, where
		\begin{equation}\hlabel{eq:Mult}
			\mm^{\mathbf{e}}_{\bm \psi}(\bm \xi)=\prod_{k=1}^d\frac{W_{\phi_{+,k}}(\frac{1}{2}-{\rm{i}}\xi_k)}{W_{\phi_{-,k}}(\frac{1}{2}+{\rm i}\xi_k)},\ \ \bm{\xi}=(\xi_1, \ldots,\xi_d)\in\bb{R}^d.
		\end{equation}
		In fact,
		\[\D(\bm{H}^{-1}_{\bm{\psi},M})=\left\{\f\in\bmrm{L}(\bb{R}^d,{\mathbf{e}}_M);\ \bm{\xi}=(\xi_1,\ldots,\xi_d) \mapsto\fouh_{\f\circ M^{-1}}(\bm{\xi})\prod_{k=1}^d\frac{W_{\phi_{-,k}}(\frac{1}{2}-{\rm{i}}\xi  _k)}{W_{\phi_{+,k}}(\frac{1}{2}+{\rm{i}}\xi  _k)}\in\bmrm{L}(\bb{R}^d)\right\}. \]
		Under some conditions (see Proposition~\ref{prop:ratio_criterion}), $\D(\bm{H}^{-1}_{\bm{\psi},M})$ is the full Hilbert space and thus $\HMd$ extends as a bounded operator on the entire Hilbert space. As a result, \eqref{eq:spectral_decomp} holds on $\bmrm{L}(\bb{R}^d,{\mathbf{e}}_M)$.
		
	\end{thm}
	This theorem is proved in Subsection~\ref{sec:proof_th:spectral_decomp}.
	\begin{rem}
		It can be easily checked that the inverse of $\bm{H}_{\bm{\psi},M}$ is $\widehat{\bm{H}}_{\ov{\bm{\psi}},M}$. This implies that when $\bm{\psi}=\ov{\bm{\psi}}$, $\bm{H}_{\bm{\psi},M}$ is an unitary operator, see  Subsection \ref{ss:fourier_mult}. In Theorem~\ref{thm:spectrum}\eqref{it:residual_spect} below, we show that the boundedness of $\HMd$ along with the condition $\mm^{\mathbf e}_{\bm{\psi}}\in\bmrm{L}(\bb{R}^d,\mathbf{e})$ ensure the existence of residual spectrum of the semigroup $\FsemiM$. Hence, Theorem~\ref{th:spectral_decomp} indicates that existence of residual spectrum, which is empty for self-adjoint linear operators, enables the semigroup to have spectral expansion on the entire Hilbert space. 
	\end{rem}

\subsection{About the spectra of the $\mathcal{C}_0$-semigroups in $\ccP$}
The next result provides criteria  to determine the nature of the spectrum of the family of semigroups  $\ccP$. To this end, we recall that for a linear operator $A$ acting on a Hilbert space $\rH$, with identity operator $I$, we write $\sigma(A)=\{\lambda \in \bb{C}; \: A-\lambda I \textrm{ is not invertible} \},$ $\sigma_p(A)=\{\lambda \in \sigma(A); \: A-\lambda I \textrm{ is not injective} \},$ $\sigma_c(A)=\{\lambda \in \sigma(A)\setminus\sigma_p(A); \: A-\lambda I \textrm{ has a dense range} \}$  and $\sigma_r(A)=\sigma(A)\setminus (\sigma_p(A) \cup \sigma_c(A))$ which corresponds to the spectrum, point spectrum, continuous spectrum and residual spectrum of $A$ respectively, see e.g.~the monograph of Dunford and  Schwartz \cite[XV.8]{DS}. Additionally, we denote the approximate point spectrum of $A$ by \[\sigma_{ap}(A)=\{\lambda\in\bb{C}; \exists \:  (v_n)_{n\geq0} \subset \rH \mbox{ such that $\|v_n\|_{\rH}=1$, $\|Av_n-\lambda v_n\|_{\rH}\to 0$}\}\]
where $\|.\|_{\rH}$ stands for the norm in $\rH$.

\begin{thm}\hlabel{thm:spectrum} Let $\FsemiM \in \ccP$ for some $M\in\mathrm{GL}_d(\bb R)$. 
\begin{enumerate}
\item\hlabel{it:spect} For all $t\ge 0$,
\begin{align*}
e^{t\bb{R}_-}\subseteq\sigma(\bm{P}^{M}_t[\bm{\psi}]).
\end{align*}
More precisely, if $\mm^{\mathbf{e}}_{\bm{\psi}}$ is bounded then $e^{t\bb{R}_-}\subseteq\sigma_{ap}(\bm{P}^{M}_t[\bm{\psi}])$, where $\Mpsi$ is defined in \eqref{eq:Mult}. If $1/\mm^{\mathbf{e}}_{\bm\psi}$ is bounded, then $e^{t\bb{R}_-}\subseteq\sigma_{ap}(\widehat{\bm{P}}^M_t[\bm{\psi}])$. \vspace{.2cm}
\item\hlabel{it:point_spect} If $\mm^{\mathbf{e}}_{\bm{\psi}}\in\bmrm{L}(\bb{R}^d)$ then \begin{equation*}
    e^{t\bb{R}_-}\subseteq\sigma_p(\bm{P}^M_t[\bm{\psi}])  \mbox{  and } \sigma_r(\bm{P}^M_t[\bm{\psi}])=\emptyset
     \end{equation*}
     with, for any $\bm{q}\in\bb{R}_+^d$, 
 \begin{equation*}
     \bm{P}^M_t[\bm{\psi}]  \tau_{\ln \bm{q}}\mathbf{J}^M_{\bm{\psi}}(\bm{x})= \mathbf{e}(-t\bm{q}) \: \tau_{\ln \bm{q}}\mathbf{J}^M_{\bm{\psi}}(\bm{x})
     \end{equation*}
     where $\tau_{\bm{q}}\f(\bm{x})=\f(\bm{x}+\bm{q})$,  $\mathbf{J}^M_{\bm{\psi}}(\bm{x})= \mathbf{J}_{\bm{\psi}}(M^{-1}{\bm{x}})$,   and \[\mathbf{J}_{\bm{\psi}}(\bm{x})=\frac{\mathbf{e}(-\bm{x}/2)}{(2\pi)^{d/2}}\int_{\bb{R}^d}e^{{\rm{i}}\langle\bm{x},\bm{\xi}\rangle}\mm^{\mathbf{e}}_{\bm \psi}(\bm \xi) d\bm{\xi} \]
      where the integral is understood in the ${\mathbf{L}}^2$-sense.
    \item \hlabel{it:residual_spect}If $1/\mm^{\mathbf e}_{\bm\psi}\in\bmrm{L}(\bb{R}^d)$ then
         \begin{equation*}
    e^{t\bb{R}_-}\subseteq\sigma_r(\bm{P}^M_t[\bm{\psi}]) \mbox{ and } \sigma_p(\bm{P}^M_t[\bm{\psi}])=\emptyset
     \end{equation*}
     with,  for any $\bm{q}\in\bb{R}_+^d$, $t>0$ and $\f\in\bmrm{L}(\bb{R}^d,{\mathbf{e}}_M)$,
     \begin{equation*}
     \left\langle \bm{P}^M_t[\bm{\psi}] \f,\tau_{\ln \bm{q}}\mathbf{J}^M_{\ov{\bm{\psi}}}\right\rangle_{\bmrm{L}(\bb{R}^d,{\mathbf{e}}_M)}   = \mathbf{e}(-t\bm{q}) \left\langle \f,\tau_{\ln \bm{q}}\mathbf{J}^M_{\ov{\bm{\psi}}}\right\rangle_{\bmrm{L}(\bb{R}^d,{\mathbf{e}}_M)}.
     \end{equation*}
\item \hlabel{it:cont_spect} If $\mm^{\mathbf e}_{\bm\psi}$ is both bounded from above and below then \begin{equation} \hlabel{eq:contSpec}
    \sigma(\bm{P}^M_t[\bm{\psi}])=\sigma_c(\bm{P}^M_t[\bm{\psi}])=e^{t\bb{R}_-}.
     \end{equation}
     In particular, when $\bm{P}^M_t[\bm{\psi}]$ is self-adjoint, i.e.~$\bm{\psi}=\ov{\bm{\psi}}$, then \eqref{eq:contSpec} holds. 
\end{enumerate}
\end{thm}

This theorem is proved in Subsection~\ref{sec:proof4}.
\begin{rem}
 For the point and residual spectrum, there are plenty of natural examples, and,  we detail some of them in Sections \ref{ex:spectrally_neg} and \ref{ex:two_sided}. More generally, for any $\phi\in\B$, defining
\begin{align}\hlabel{eq:Thetas}
 \underline{\Theta}_\phi=\liminf_{|\xi|\to\infty}\Theta_\phi(|\xi|) \mbox{ and } \ov{\Theta}_\phi=\limsup_{|\xi|\to\infty}\Theta_\phi(|\xi|)
\end{align}
where $\Theta_\phi(|\xi|)=\frac{\int_{1/2}^\infty\ln\left(\frac{|\phi(a+\i|\xi|)|}{\phi(a)}\right)\,da}{|\xi|}$, one easily deduces from \cite[Theorem 6.2(1)]{patiesavov}, that,  for any $\xi \in \bb{R}$,
\begin{equation}\hlabel{eq:ration-gene}
		\left|\frac{W_{\phi_+}\left(\frac{1}{2}+{\rm{i}}\xi\right)   }{W_{\phi_-}\left(\frac{1}{2}+{\rm{i}}\xi\right)   }\right|\asymp\frac{\sqrt{|\phi_-\left(\frac{1}{2}+{\rm{i}}\xi\right) |}}{\sqrt{|\phi_+\left(\frac{1}{2}+{\rm{i}}\xi\right)   |}}e^{-|\xi|(\Theta_{\phi_+}(|\xi|)-\Theta_{\phi_-}(|\xi|))}.
	\end{equation} 
Thus, if $\underline{\Theta}_{\phi_+}>\ov{\Theta}_{\phi_-}$ (resp.~$\underline{\Theta}_{\phi_-}>\ov{\Theta}_{\phi_+}$), then ${P}[\psi]$ has point spectrum (resp.~has residual spectrum). For a proof of the last two statements and also sufficient conditions for both point and residual spectrum  in the case $\underline{\Theta}_{\phi_+}=\ov{\Theta}_{\phi_-}$, we refer to Proposition~\ref{prop:ratio_criterion}. In Section \ref{sec:cont_spect_ex}, we  provide examples of $\psi'$s for which ${P}[\psi]$ is non-self-adjoint and has continuous spectrum. However, in general, though we have always $e^{t\bb{R}_-}\subseteq\sigma(\bm{P}^M_t[\bm{\psi}])$, the weak similarity relation does not allow to check the reverse inclusion.
\end{rem}
\begin{rem}
 When the function $\Mpsi$ is polynomially bounded, i.e.~there is some positive $p>0$ such that \[\left|\Mpsi(\bm \xi)\right|\le C_p(1+\|\bm{\xi}\|)^p\] we can still have the Fourier inverse of $\Mpsi$ in the sense of Schwartz distribution, that could be used to define the functions associated to the residual spectrum. However, the growth of $\Mpsi$ along the real line,  being, in general,  exponential, this analysis requires deeper techniques  that we plan to develop in a subsequent work.
\end{rem}

\subsection{Positivity-preserving and integro-differential representation of the generators} In this section we provide an integro-differential representation of  the $\bmrm{L}(\bb R^d,\mathbf{e}_M)$-generator of the semigroup $\bm{P}^M[\bm{\psi}]$. We start by introducing the following subset of $\bmrm{L}(\bb R^d,\mathbf{e})$
\begin{equation}
		\bm{\cc{E}}=\bigotimes_{k=1}^d\cc{E} \mbox{ with }  \cc{E}=\Span\{\heb;\: \epsilon,\beta>0\}
\end{equation}
where $\heb: \bb{R} \to \bb{R}_+$ is defined by $\heb(x)=e^{-(\frac{1}{2}+\epsilon)x}e^{-\beta e^{-x}}$, and $\otimes$ stands for the tensor product of the univariate functions. Next, for any $\bm{\psi}=(\psi_1,\ldots,\psi_d)\in\mathbf{N}^d_b(\bb{R})$, where, for any $k=1,\ldots d$, $(\psi_k(0),\sigma_k^2,\ttt{b}_k,\mu_k)$ is the characteristic quadruplet of $\psi_k$,  and $M \in {\rm{GL}}_d(\bb{R})$, we consider the linear integro-differential operator acting on smooth and well behaved functions $\f: \bb{R}^d \to \bb{R}$ as follows
\begin{eqnarray}\hlabel{eq:fourier_int}
	\AM[\bm{\psi}]\f(\bm{x})&=&\tr\left({M}^\top \Sigma_{\bm{\sigma}}({M}^{-1}\bm{x}) {M}\nab^2 \f(\bm{x})\right)+\left\langle {M}e_{\bm{\mathrm b}}(-{M}^{-1}\bm{x}),\nab \f(\bm{x})\right\rangle  \nonumber \\
	& &+\int_{\bb{R}^d}\left(\f(\bm{x}+\bm{y})-\f(\bm{x})-\langle\bm{y},\nab \f(\bm{x})\rangle\bbm{1}_{\{\|{M}^{-1}\bm{y}\|\le 1 \}}\right)\langle e(-{M}^{-1}\bm{x}),\bm{\mu}_M(d\bm{y})\rangle \\
&& - \left\langle e(-\bm{x}),\bm{\psi}(0)\right\rangle\f(\bm{x})  \nonumber 
\end{eqnarray}
where $\tr$ stands for the trace, $\Sigma_{\bm{\sigma}}$ is defined in \eqref{eq:defQ}, $\nab$ is the gradient, $\mathbf{b}=(\ttt{b}_1,\ldots,\ttt{b}_d)$, $e_{\bm{\mathrm b}}(-\bm{x})=(\ttt{b}_1e^{-x_1},\ldots, \ttt{b}_de^{-x_d})$ with $e_\mathbf{1}(-\bm{x})=e(-\bm x)=(e^{-x_1},\ldots, e^{-x_d})$. $\langle .,.\rangle$ and $\|\cdot\|$ are  the inner product and the Euclidean norm in $\bb{R}^d$ respectively, and,
for any $\bm{B} \in\cc{B}(\bb{R}^d)$, $\bm{\mu}_M(\bm{B})=\bm{\mu}({M}^{-1}\bm{B})$ with $\bm{\mu}(\bm{B})=(\mu_1(B_1),\ldots, \mu_d(B_d))$ where for all $1\le k\le d$, $B_k=\left\{x\in\bb{R}: x\mf{e}_k\in\bm{B}\right\}$, $(\mf{e}_1,\ldots,\mf{e}_d)$ being the canonical basis of $\bb{R}^d$.
%
Let us now introduce the following condition regarding the Wiener-Hopf factors of $\bm{\psi}$.
\begin{equation}\hlabel{cond:finite_mean}
	\textrm{For any $1\le k\le d$, either $\int_{|y|> 1}\!\!\hptg{C15}{|y|}\mu_k(dy)<\infty$ \; or \; $\phi_{+,k}(0)>0$}
\end{equation}
\begin{equation}\hlabel{eq:schwarz_decay}
	\lim_{|\xi|\to\infty}|\xi|^ne^{-\frac{\pi}{2}|\xi|}\sup_{1\le k\le d}\left|\frac{W_{\phi_{+,k}}(\frac{1}{2}+\i\xi)}{W_{\phi_{-,k}}(\frac{1}{2}+\i\xi)}\right|=0 \ \textrm{for all $n\in\bb N$}
\end{equation}
When $\phi_{+,k}(0)=0$, the integrability assumption in \eqref{cond:finite_mean} combined with an application of Taylor's formula imply that the operator $\AM[\bm\psi]$ in \eqref{eq:fourier_int} is well defined on the set
\begin{align}
	\bm{\cc{C}}=\{\f\in \C^2(\bb{R}^d); \: \nab\f\in \C^1_b(\bb{R}^d)\}.
\end{align}
\begin{thm}\hlabel{thm:core_multi_d}
Let $M\in\mathrm{GL}_d(\bb R)$.
\begin{enumerate}
\item\hlabel{it:feller_semigroup} We have \[ \ccP_M \subset \mathcal{C}^+_0(\bmrm{L}(\bb{R}^d,{\mathbf{e}}_M)).\] More specifically, $\FsemiM$,  when restricted to $\C_0(\bb{R}^d) \cap \bmrm{L}(\bb{R}^d,{\mathbf{e}}_M)$, extends to the semigroup of the Feller process on $\bb{R}^d$ defined, for any $t\geq0$, as $\bm{Y}_t = M \bm{X}_t,$ where $\bm{X}_t=(\ln X^{(1)}_t, \ldots,\ln X^{(d)}_t)$, the stochastic processes $(X^{(k)}_t)_{t\geq0}, k=1,\ldots,d,$ are mutually independent  and each of them is a positive self-similar  Feller process on $(0,\infty)$, see Section \ref{sec:lamp} for more details on these processes.
	\item\hlabel{it:core} If $\bm{\psi}$ satisfies the conditions \eqref{cond:finite_mean} and \eqref{eq:schwarz_decay}, then $\bm{\cc{E}}^M_{\bm \psi}:=\bm{H}^M_{\bm\psi}(\bm{\cc{E}})$ is a core for $\bm{A}_{\ttt 2}[\bm \psi]$, and its restriction on $\bm{\cc{E}}^M_{\bm \psi}$  coincides with $\bm{A}^M[\bm\psi]$ defined in \eqref{eq:fourier_int}.
\end{enumerate}
\end{thm}
This theorem is proved in Section \ref{sec:thm29}.
\begin{rem}
	The condition \eqref{cond:finite_mean} is needed to ensure that $\bm{\cc{E}}^M_{\bm \psi}\subset\bm{\cc C}$, which entails that the integro-differential operator in \eqref{eq:fourier_int} is well-defined on $\bm{\cc{E}}^M_{\bm \psi}$. In order to weaken this assumption and get that $\bm{\cc{E}}^M_{\bm \psi}\subset\C^2_b(\bb R^d)$ whenever $\bm{\psi}$ satisfies \eqref{eq:schwarz_decay}, one would need to develop a refined analysis of the asymptotic behavior of the ratio $\left|\frac{\Wphip}{\Wphim}\right|$.
\end{rem}

\subsection{Conditions for point or residual spectrum in terms of the Wiener-Hopf factors}
The next proposition provides some sufficient conditions on the characteristic triplet of the Wiener-Hopf factors for the characterization of the presence of point or residual spectrum, which according to Theorem \ref{thm:spectrum}, boils down to identify whether or not $m_{\psi}=m_{\psi_k} \in \bmrm{L}(\bb{R})$ for  $k=1, \ldots,d$.  For $\phi\in\B$, we recall that $\ov{\Theta}_{\phi}$ and  $\underline{\Theta}_{\phi}$ are defined  in \eqref{eq:Thetas}.
\begin{prop}\hlabel{prop:ratio_criterion}
	Let $\psi\in\mathbf{N}(\bb{R})$ where $\psi(\xi)=\phi_+(-{\rm{i}}\xi  )\phi_-({\rm{i}}\xi  )$ for all $\xi\in \bb{R}$ with  $\phi_\pm(z)=\phi_\pm(0)+\ttt{d}_\pm z+\int_0^\infty (1-e^{-zy})\nu_\pm(dy)$ for all $z\in\bb{C}_{[0,\infty)}$. Then, for all $\xi\in \bb{R}$,
	
	\begin{enumerate}[(i)]
		\item \hlabel{it:ratio_bdd} If $\underline{\Theta}_{\phi_+}>\ov{\Theta}_{\phi_-}$  and
		\begin{align*}
			\sup\left\{\kappa>0;\: \liminf_{|\xi|\to\infty}|\xi|^\kappa |\phi_+(\mathrm{i}\xi)|>0\right\}<\infty
		\end{align*}
		then $\psi\in\mathbf{N}_+(\bb{R})$,  $m_{\psi}\in\bmrm{L}(\bb{R})$ and $m_{\overline{\psi}}\notin\bmrm{L}(\bb{R})$. Table \ref{tab:pos} provides sufficient conditions expressed in terms of the characteristic triplets of $\phi_+$ and $\phi_-$ for $\underline{\Theta}_{\phi_+}> \ov{\Theta}_{\phi_-}$, where $\ov{\nu}_\pm(r)=\nu_\pm(r,\infty)$, and, for any $\alpha>0$, $\mathbf{RV}(\alpha)$ (resp.~q-m) denotes the set of all regularly varying (resp.~quasi-monotone) functions of index $\alpha$, see \cite[Chapter 1 and Section 2.7]{BGT}
\begin{table}[h]
		\begin{center}{
				\centering\makegapedcells
			\begin{tabular}{|c|c|c|c|c|}
				\hline
				\makecell{$\ttt{d}_+$} & \makecell{$\ttt{d}_-$} & \makecell{$\ov{\nu}_+$} & \makecell{$\ov{\nu}_-$} & \makecell{$\left|\Wphi\right|$} \\
				\hline
				\makecell{$0$} & \makecell{$0$} & \makecell{$\ov{\nu}_+\in\mathbf{RV}(\alpha_+),$ \\  $y\mapsto\frac{\overline{\nu}_+(y)}{y^{\alpha_+}} \text{ is q-m}$} & \makecell{$\ov{\nu}_-\in\mathbf{RV}(\alpha_-),$ \vspace{.2 cm}\\ $y\mapsto\frac{\overline{\nu}_-(y)}{y^{\alpha_-}} \text{ is q-m},$} & \makecell{$={\rm{O}}\left(e^{-(\alpha_+-\alpha_-)\frac{\pi}{2}|\xi|}\right)$\vspace{.2 cm} \\ $0<\alpha_-<\alpha_+<1$} \\
				\hline
				\makecell{$>0$} & \makecell{$0$} &  &\makecell{$\ov{\nu}_-\in\mathbf{RV}(\alpha)$, \\ $y\mapsto\frac{\overline{\nu}_-(y)}{y^{\alpha}} \text{ is q-m},$} & \makecell{$={\rm{O}}\left(e^{-(1-\alpha)\frac{\pi}{2}|\xi|}\right)$\vspace{.2 cm} \\ $0<\alpha<1$} \\
				\hline
			\end{tabular}}
\caption{Conditions for $m_{\psi}\in\bmrm{L}(\bb{R})$ when $\underline{\Theta}_{\phi_+}>\ov{\Theta}_{\phi_-}$} \hlabel{tab:pos} 
		\end{center} 
\end{table}	
	
		\item \hlabel{it:p3} If $\underline{\Theta}_{\phi_+}=\ov{\Theta}_{\phi_-}$, we give in Table~\ref{tab:zero}  a set  of sufficient conditions on the characteristic triplets of $\phi_+, \phi_-$  to ensure  that $m_\psi\in\bmrm{L}(\bb R)$ and $m_{\overline{\psi}}\notin\bmrm{L}(\bb{R})$. 
	\begin{table}[h]
		\begin{center}{
			\centering\makegapedcells
			\begin{tabular}{|c|c|c|c|c|}
				\hline
				$\ttt{d}_+$ & $\ttt{d}_-$ & $\ov{\nu}_+$ & $\ov{\nu}_-$ & $\makecell{\left|\Wphi\right|}$ \\  
				\hline
				$>0$ & $0$ & $\ov{\nu}_+(0)<\infty$ & & $={\rm{O}}(|\xi|^{-u}) \ \forall u>0$ \\
				\hline
				$>0$ & $>0$ & $\ov{\nu}_+(0)<\infty$ & $\ov{\nu}_-(0)=\infty$ & $={\rm{O}}(|\xi|^{-u}) \ \forall u>0$ \\
				\hline
			\end{tabular}}
		\caption{Conditions for $m_{\psi}\in\bmrm{L}(\bb{R})$ when $\underline{\Theta}_{\phi_+}=\ov{\Theta}_{\phi_-}$} \hlabel{tab:zero}
	\end{center}
	\end{table}
\end{enumerate}
\end{prop}
This proposition is proved in Section~\ref{sec:prop2.2}.
We point out that the study of the asymptotic behavior of the ratio $\left|\Wphi\right|$ for any pairs of Bernstein functions, in particular any pairs of L\'evy measures, is a delicate issue. Our results identify several large  classes of Bernstein functions for which such estimate is attainable. We mention that, in the recent paper \cite{Sav-M}, asymptotic estimates of this ratio when $W_{\phi_+}$ is the classical gamma function, has been obtained under other type of conditions than ours on the tail of the L\'evy measure $\ov{\nu}_-$.

\section{Preliminaries and auxiliary results} \hlabel{sec:prelim}
In this section we gather some general facts about several ideas and tools that will be used throughout the remaining part of the paper, with an emphasis  to the  theory of shifted Fourier multipliers and the analytical properties of certain ratios of Bernstein-gamma functions that will be important in the sequel.

\subsection{Fourier transforms, shifted Fourier transforms and some classical results}
For any function $f$ in $\bmrm[1]{L}(\bb{R})$ or $\bmrm{L}(\bb{R})$, we denote its Fourier transform by $\mathcal{F}_{f}$, i.e.~
\begin{equation*}
\fou_f(\xi)=\frac{1}{\sqrt{2\pi}}\int_{\bb{R}} e^{-{\rm{i}}x\xi} f(x)\,dx,\:  \xi\in\bb{R}.
\end{equation*} As before we define $\fou^{e}$ the shifted Fourier transform, i.e.~for any $f\in\bmrm{L}(\bb{R},e)$, writing $e(x)=e^{x}$,
\begin{align}
\fou^{e}_f(\xi)=\fou_{\sqrt{e}f} (\xi)
\end{align}
in the $\bmrm{L}$-sense. Clearly, $\fou^{e}: \bmrm{L}(\bb{R},e)\to\bmrm{L}(\bb{R})$ is an isometry whose inverse is denoted by $\widehat{\fou}^{e}$. Next, we mention a very useful result due to Wiener that will be used frequently in this paper.
\begin{thm}[Wiener's Tauberian Theorem] \hlabel{thm:WT}
Let $f\in\bmrm{L}(\bb{R})$ be such that $\mathcal{F}_{f}$ is almost everywhere (a.e.)~non-zero. Then, recalling that $\tau_af(.)=f(.+a)$, the set $\Span\{\tau_af; a\in\bb{R}\}$ is dense in $\bmrm{L}(\bb{R})$.
\end{thm}
This result is very standard and can be found in \cite{BGT}. In the following corollary, we mention a variant of Wiener's theorem which will be needed in the subsequent results.
\begin{corr}\label{corr:wiener}
Let $f\in\bmrm{L}(\bb{R},e)$  be such that $\fou^{e}_f$ is a.e.~non-zero. Then, \hpt{R5}{$Span\{\tau_a f; a\in\bb{R}\}$ is dense in $\bmrm{L}(\bb{R},e)$}.
\end{corr}
\begin{proof}
Note that $f\in\bmrm{L}(\bb{R},e)$ implies that $\sqrt{e}f\in\bmrm{L}(\bb{R})$, where we recall that $\sqrt{e}f=f(x)e^{\frac{x}{2}}, x\in \bb{R}$. By means of  the Wiener's Tauberian theorem, we have that $\tredd{\Span\{\tau_a \sqrt{e}f; a\in\bb{R}\}}$ is dense in $\bmrm{L}(\bb{R})$, which proves the result.
\end{proof}
We close this part with the following variant of Cauchy contour integration formula that will also be useful later.
\begin{lem}\hlabel{prop:contour} Let $f\in\A_{[0,\gamma]}$ for some $\gamma>0$ and there exists $1\le p \leq \infty$ such that\hpt{R6}{, for $p<\infty$,
\begin{align*}
\sup\limits_{b\in [0,\gamma]}\left(\int_{\bb R}\left|f(x+{\rm{i}}b)\right|^p dx\right)^{\frac{1}{p}}<\infty,
\end{align*}
(resp.~ the supremum norm is finite for $p=\infty$).}
Then, one can choose a subsequence $(n_j)_{j\ge 1}$ of natural numbers such that for any $0\le b\le\gamma$,
\begin{align*}
\lim\limits_{j\to\infty}\int_{-n_j}^{n_j} [f(x+\i b)-f(x)]dx=0.
\end{align*}
In particular, when $p=1$, $\int_\bb{R} f(x)dx=\int_{\bb R} f(x+\i b)dx$ for all $0\le b\le \gamma$, and when $p=\infty$, $\lim\limits_{n\to\infty}\int_{-n}^n [f(x+\i b)-f(x)]dx=0$ for all $0\le b\le\gamma$.
\end{lem}
\begin{proof}
	Let us first assume that $1\le p<\infty$. From the analyticity of $f$ in the strip $\bb{S}_{[0,\gamma]}$ and applying the Cauchy integral formula, for any $x>0$ and $0<b\le\gamma$ we have
	\begin{align*}
		\int_{R^b_x} f(z) dz=0
	\end{align*}
where $R^b_x$ is the rectangle formed by the points $-x,x,x+\i b, -x+\i b$. Let us estimate the integrals on the vertical lines of $R^b_x$. By H\"older's inequality, we have
\begin{align*}
		\left(\int_{0}^b |f(\pm x+\i y)|dy\right)^p\le\left(\int_{0}^{\gamma} |f(\pm x+\i y)|dy\right)^p\le \gamma^{\frac{p}{q}}\int_0^\gamma |f(\pm x+\i y)|^p dy
\end{align*}
where $\frac{1}{p}+\frac{1}{q}=1$ with $q=\infty$ when $p=1$. Applying Fubini's theorem combined with the assumption of the lemma, we get
\begin{align*}
	\int_\bb{R}\left(\int_{0}^{\gamma} |f(x+\i y)|dy\right)^p dx\le\gamma^{\frac{p}{q}}\int_0^\gamma\int_{\bb R} |f(x+\i y)|^p dx\,dy<\infty
\end{align*}
which implies that there exists a subsequence $(n_j)_{j\ge 1}$ of natural numbers such that
\begin{align*}
\lim_{j\to\infty}\int_{0}^{\gamma} |f(\pm n_j+\i y)|dy=0.
\end{align*}
As a result, for any $0\le b\le \gamma$,
	$\lim\limits_{j\to\infty}\left(\int_{-n_j}^{n_j} [f(x+\i b)-f(x)]-\int_{R^b_{n_j}} f(z)dz\right)=0$, which proves the lemma. When $p=\infty$, for any $0<b\le\gamma$, the integral on the vertical lines of $R^b_x$ goes to $0$ as $x\to\infty$. This completes the proof of the lemma for all $1\le p\le\infty$.
\end{proof}

\subsection{Negative definite functions on $\bb{R}^d$ and pseudo-differential operators}
\subsubsection{Negative definite functions}
A function $\psi:\bb{R}^d\to\bb{C}$ is called negative definite if for any choice of $p\in\bb{N}$ and $(\bm{\xi}_1,\ldots, \bm{\xi}_p)\in\bb{R}^{d\times p}$, the matrix $$(\psi(\bm{\xi}_i)+\overline{\psi}(\bm{\xi}_j)-\psi(\bm{\xi}_i-\bm{\xi}_j))_{1\leq i,j\leq p}$$ is non-negative Hermitian. It is a well known fact that any continuous negative definite function $\psi$ has the following representation
\begin{align*}
\psi(\bm{\xi})=\psi(\bm{0})-\mathrm{i}\langle\mathbf{b},\bm{\xi}\rangle+\langle\bm{\xi},\Sigma\bm{\xi}\rangle+\int_{\bb{R}^d}\left(1-e^{{\rm{i}}\langle\bm{\xi},\bm{y}\rangle}+{\rm{i}}\langle\bm{\xi},\bm{y}\rangle\mathbbm{1}_{\{\|\bm{y}\|\le 1\}}\right)\mu(d\bm{y})
\end{align*}
where $\psi(\bm{0})\ge 0, \mathbf{b}\in\bb{R}^d,$ $\Sigma$ is non-negative definite and $\mu$ is a positive Radon measure such that $\int_{\bb{R}^d}(\|\bm{y}\|^2\wedge 1)\mu(d\bm{y})<\infty$. We denote the set of all continuous negative definite functions by $\mathbf{N}(\bb{R}^d)$.
The class $\mathbf{N}(\bb{R}^d)$ comes naturally in the  study of L\'evy processes. Indeed, for any (possibly killed) L\'evy process $Z=(Z_t)_{t\ge 0}$ on $\bb{R}^d$, there is a unique $\psi\in\mathbf{N}(\bb{R}^d)$ such that
\begin{align}\hlabel{eq:levy_kh}
\bb{E}[e^{{\rm{i}}\langle\bm{\xi}, Z_t\rangle}]=e^{-t\psi(\bm{\xi})}, \ t\ge 0, \ \  \bm{\xi}\in\bb{R}^d,
\end{align}
and the infinitesimal generator of $Z$ is given, for a smooth function \hptg{C20}{$\f \in \mathbf{S}(\bb{R}^d)$}, by
\begin{align*}
L[\psi]\f(\bm{x})=&\langle\nabla\f,\Sigma\nabla \f(\bm{x})\rangle+\langle\mathbf{b},\nabla\f(\bm{x})\rangle +\int_{\bb{R}^d}(\f(\bm{x}+\bm{y})-\f(\bm{x})-\langle\bm{y}, \nabla\f(\bm{x})\rangle\bbm{1}_{\{\|\bm{y}\|\le 1\}})\mu(d\bm{y})\\ &-\psi(\bm{0})\f(\bm{x}).
\end{align*}
We refer the monograph of Jacob \cite{Jacob} for a thorough account on this set of functions and its connection to L\'evy processes.
\subsubsection{Pseudo-differential operators and connection with Markov Processes}\hlabel{ss:lamperti}
Let $\bbm{a}:\bb{R}^d\times\bb{R}^d\to\bb{C}$ be a function which is continuous and, for all $\bm{x}\in \bb{R}^d$, the mapping $\bm{\xi} \mapsto \bbm{a}(\bm{x},\bm{\xi})$ is negative definite. From the general theory of \emph{pseudo-differential operators} (PDO in short) for Markov processes, we know that $|\bbm{a}(\bm{x},\bm{\xi})|\le\gamma(\bm{x})(1+\|\bm{\xi}\|^2)$ for all $\bm{x}$, $\gamma$ being a locally finite function. We define the following linear operator
\begin{equation*}
    -\bbm{a}(\bm{x},D)f(x)=-(2\pi)^{-\frac{d}{2}}\int_{\bb{R}^d}e^{\mathrm{i}\langle\bm{\xi},\bm{x}\rangle}\mathcal{F}_{\f}(\bm{\xi})\bbm{a}(\bm{x},\bm{\xi})\,d\bm{\xi}~, \  \  \f\in\SS(\bb{R}^d),
\end{equation*}
where $\SS(\bb{R}^d)$ denotes the Schwartz space on $\bb{R}^d$. Clearly, the above integral is well defined and we say that $\bbm{a}(\bm{x},D)$ is a PDO with  \emph{symbol} $\bbm{a}(\bm{x},\bm{\xi})$. The class of pseudo-differential operators with negative definite symbols plays an important role in the theory of Markov processes. Courr\`ege~\cite{courrege1966} showed that if $(\bm{A},\D(\bm{A}))$ is the generator of a Feller semigroup on $\bb{R}^d$ such that $\C^\infty_c(\bb{R}^d)\subseteq\D(\bm{A})$, then $A$ is a pseudo-differential operator with symbol $\bbm{a}(\bm{x},\bm{\xi})$ of the form
$$\bbm{a}(\bm{x},\bm{\xi})=\bbm{a}(\bm{x},\bm{0})-{\rm{i}}\langle\bm{\xi},\mathbf{b}(\bm{x})\rangle+\langle\bm{\xi}, Q(\bm{x})\bm{\xi}\rangle+\int_{\bb{R}^d\setminus\{0\}}(1-e^{{\rm{i}}\langle\bm{\xi},\bm{y}\rangle}+{\rm{i}}\langle\bm{\xi},\bm{y}\rangle\bbm{1}_{\{\|\bm{y}\|\le 1\}})\mu(\bm{x},d\bm{y})$$
where,  for all $\bm{x}$, $\bbm{a}(\bm{x},\bm{0})\ge 0$, $Q(\bm{x})$ is non-negative definite, $\mathbf{b}:\bb{R}^d\to\bb{R}^d$ is measurable and  $\mu(\bm{x},d\bm{y})$ is a L\'evy measure, i.e.~ $\int_{\bb{R}^d}(\|\bm{y}\|^2\wedge 1)\mu(\bm{x},d\bm{y})<\infty$.
On the other hand, if we are given a pseudo-differential operator with a negative definite symbol, the question of existence of a unique Feller process corresponding to it, is more subtle. Some results in this direction can be found in \cite[Theorem 5.7]{hoh1998pseudo}, but the required conditions on the symbol $\bbm{a}(\cdot,\cdot)$ are quite stringent. For instance, already in the one-dimensional case, they are not satisfied by the symbol $\ttt{a}(x,\xi)=e^{-x}\psi(\xi)$ where $\psi$ is a continuous negative definite function. However, in Theorem~\ref{thm:intertwining_pdo} below, we present an original approach based on the concept of  intertwining to overcome this issue. Indeed, we will show that under some mild conditions on $\psi$, the PDO with symbol $\ttt{a}(x,\xi)=e^{-x}\psi(\xi)$ restricted on a certain dense subset of functions, extends uniquely to the generator of a $\cc{C}_0$-contraction positivity-preserving semigroup on $\bmrm{L}(\bb{R},e)$.

 \subsection{Lamperti mapping, duality, the $\log$-transformation and generators} \hlabel{sec:lamp}
 In his seminal paper,  Lamperti \cite{Lamperti1972} established a one-to-one correspondence between the class of all L\'evy processes and self-similar Feller processes on the positive real line up to their absorption time at $0$. More specifically, for any positive $\alpha$-self-similar, $\alpha>0$, Feller process $X=(X_t(x))_{t\ge 0}$ issued from $x>0$, there is a unique L\'evy process $Z=(Z_t)_{t\ge 0}$ issued from the origin such that,  for all  $0\le t <\zeta(x)=\inf\{t\ge 0; X_t(x)\le 0\}$,
 \begin{align*}
 X_t(x)= x\exp(Z_{\varphi (x^{-\alpha}t)})
 \end{align*}
 where $  \varphi (t)=\inf\left\{s>0; \: \int_0^s e^{\alpha Z_r}\,dr>t\right\}$. Almost sure finiteness of the first hitting time $\zeta(x)$ depends on the L\'evy process $Z$ as well as its lifetime. This is known as the Lamperti mapping and enables to associate to ${\rm X}$ a unique L\'evy-Khintchine exponent and for more details, we refer to  \cite[Theorem 4.1]{Lamperti1972}. It is not hard to see that for any $\alpha$-self-similar Feller  process $(X_t)_{t\ge 0}$ on $(0,\infty)$, the process $( X^{1/\alpha}_t)_{t\ge 0}$ is a $1$-self-similar Feller process on $(0,\infty)$. Since $x\mapsto x^{1/\alpha}$ is a homeomorphism on $(0,\infty)$, the corresponding semigroups are similar to each other. For a 1-self-similar Feller semigroup ${P}^\ttt{F}_t[\psi]$ corresponding to the L\'evy-Khintchine exponent $\psi$, the adjoint semigroup is also a Feller one and corresponds to the conjugate L\'evy-Khintchine exponent $\ov{\psi}$, see \cite[Lemma 2]{BY02}. That is, for any $f,g\ge 0$,
 \begin{align} \hlabel{eq:adj}
 \int_{\bb{R}_+} {P}^\ttt{F}_t[\psi]f(x) g(x)\,dx=\int_{\bb{R}_+}f(x){P}^\ttt{F}_t[\ov{\psi}]g(x)\,dx.
 \end{align}
 This ensures that the Lebesgue measure on $\bb{R}_+$ is an excessive measure for the semigroup ${P}^\ttt{F}[\psi]$ and thus, the latter has a natural extension to $\bmrm{L}(\bb{R}_+)$, which we denote by ${P}[\psi]$. ${P}[\psi]$ is a $\cc{C}_0$-contraction semigroup on $\bmrm{L}(\bb{R}_+)$ and its adjoint is ${P}[\ov{\psi}]$. Since our  approach stems on the theory of  pseudo-differential operators which are defined on the entire real line $\bb{R}$, we consider   the logarithm of the $1$-self-similar Feller processes. Since the logarithm is  a homeomorphism from $(0,\infty)$ to $\bb{R}$, the resultant semigroup is still similar to the original one. Due to the $\log$-transformation, the resulting semigroup has $e(x)dx=e^x\,dx, x\in \bb{R},$ as an excessive measure and extends to a $\cc{C}_0$-contraction semigroup on $\bmrm{L}(\bb{R},e)$. Therefore, all the spectral properties remain invariant under this $\log$-transformation. Next, from \cite{Lamperti1972}, it is known that the Dynkin characteristic operator of a $1$-self-similar process associated to the L\'evy-Khintchine exponent $\psi$ is given by
\begin{align} \hlabel{eq:dynk}
\overline{A}_\ttt{D}[\psi]f(x)=\frac{1}{x}L[\psi](f\circ\exp)(\ln x), \ \ x\in\bb{R}
\end{align}
where $L[\psi]$ is the generator of the L\'evy process with L\'evy-Khintchine exponent $\psi$ and the set
$\{f:\bb{R}_+\to\bb{R}_+; \: x\mapsto f(x), xf'(x), x^2f''(x)\in \C_b(\bb{R}_+)\}\subseteq\D(\overline{A}_\ttt{D}[\psi])$.
If now $A_\ttt{D}[\psi]$ denotes the Dynkin characteristic operator of the log-self-similar process, then
\begin{align}\hlabel{eq:dynkin=pdo}
A_\ttt{D}[\psi]f(x)=e^{-x}L[\psi]f(x), \ \ x\in\bb{R},
\end{align}
and $
\{f:\bb{R}\to\bb{R}; \: f\in \C^2_b(\bb{R})\}\subseteq\D(A_\ttt{D}[\psi])$.
Also, for any $f\in\SS(\bb{R})$, from \eqref{eq:dynkin=pdo}, we deduce that $A_\ttt{D}[\psi]f= A_\ttt{PDO}[\psi]f$  where
\begin{align*}
A_\ttt{PDO}[\psi]f(x)=-\frac{1}{\sqrt{2\pi}}\int_\bb{R} e^{-x}\psi(\xi)\mathcal{F}_{f}(\xi)e^{{\rm{i}}\xi   x}\,d\xi
\end{align*}
which is a PDO with symbol $\ttt a(x,\xi)=e^{-x}\psi(\xi)$.
From now on, we only consider the log-self-similar processes and we denote their semigroups by $P[\psi]$. 
\subsection{Shifted Fourier multiplier operators}\hlabel{ss:fourier_mult}
We recall that an operator $\Lambda:\LRe\to\LRe$ is called a shifted Fourier multiplier operator if there exists a measurable function $m^e_\LL:\bb{R}\to\bb{C}$ such that, for all $\xi\in\bb{R}$,
\begin{align*}
	\fouho_{\Lambda f}(\xi)=m^e_\LL(\xi)\fouho_f(\xi)
\end{align*}
for all $f\in\LRe$ with $m^e_\LL\fouho_{\!f}\in\bmrm{L}(\bb{R})$. The operator $\Lambda$ is bounded if and only if $\tredd{m^e_{\Lambda}}$ is measurable and essentially bounded. In this work, we use a subclass of the shifted Fourier multiplier operators for which the multiplier is analytic on a strip. This class is introduced formally in the following definition.
\begin{definition}
Let $\mathscr{M}$ be the collection of all shifted Fourier multiplier operators defined on the weighted Hilbert space $\bmrm{L}(\bb{R},e)$ such that
for any $\LL\in\mathscr{M}$, there exists a function $m^e_{\LL}$, the Fourier multiplier, such that  writing \hpt{R10}{$m_{\LL}\left(\cdot+\frac{{\rm{i}}}{2}\right)=m^e_{\LL}(\cdot)$}, $m_{\LL}$ is defined on $\bb{S}_{(0,1)}$  and satisfies
\begin{enumerate}[(i)]
\item $m_{\LL}$ is analytic in the strip $\bb{S}_{(0,1)}$
\item $m_\LL$ is \hptg{C23}{zero-free} on $\bb{S}_{(0,1)}$
and
\begin{align}\hlabel{fourier_mult}
\fouho_{\LL f}(\xi)=m_{\LL}\left(\xi+\frac{{\rm{i}}}{2}\right)\fouho_f(\xi), \ \ \xi\in\bb{R},
\end{align}
with $\D(\LL)=\{f\in \bmrm{L}(\bb{R},e); \: m_{\LL}\left(\xi+\frac{{\rm{i}}}{2}\right)\fouho_{f}(\xi)\in \bmrm{L}(\bb{R})\}$.
\end{enumerate}
\end{definition}

\noindent
If $f\in\D(\LL)$ is regular enough and the multiplier $m_\LL$ satisfies some integrability conditions, then on a subspace of $\LRe$, $\LL$ can be expressed as a classical Fourier multiplier operator on $\bmrm{L}(\bb R)$.
\begin{prop}\hlabel{multiplier_result}
\hpt{R11}{Consider $\Lambda\in\mathscr{M}$} such that $m_\LL$ extends continuously on $\bb R$. Let $f\in\D(\LL)$ be such that for some $\frac{1}{2}\le\gamma\le 1$,
\begin{enumerate}[(i)]
\item $\fou_f\in\A_{[0,\gamma]}$ and $\fou_f(\cdot+\i b)\in\bmrm{L}(\bb R)$ for all $0\le b\le \gamma$
\item \hlabel{it:sup}$\sup\limits_{b\in [0,\gamma]}\int_{\bb{R}}|m_{\LL}(\xi+{\rm{i}}b)\mathcal{F}_{f}(\xi+{\rm{i}}b)|^2\,d\xi<\infty$.
\end{enumerate}
Then, $\Lambda f\in\mathbf{L}^2(\bb R, e^{2bx})$ for any $0\le b\le\gamma$, and for all $\xi\in\bb R$,
\begin{align}\hlabel{eq:fourier_b}
	\fou_{\Lambda f}(\xi+\i b)=m_{\Lambda}(\xi+\i b)\fou_f(\xi+\i b).
\end{align}
\end{prop}
\begin{proof} We first note that condition~\eqref{it:sup} ensures that for any $0\le b\le \gamma$, there exists $g_b\in\bmrm{L}(\bb R)$ such that $\fou_{\! g_b}(\xi)=m_{\Lambda}(\xi+\i b)\fou_f(\xi+\i b)$ for all $\xi\in\bb R$. Now, from the standard theory of Fourier transform, it is known that for any $0\le b\le \gamma$,
\begin{align*}
	g_b(x)=\lim_{n\to\infty}\frac{1}{\sqrt{2\pi}}\int_{-n}^n m_{\Lambda}(\xi+\i b)\fou_f(\xi+\i b) e^{\i\xi x} d\xi
\end{align*}
where the above convergence holds in $\bmrm{L}(\bb R)$. Now, for each $x\in\bb R$, since the integrand above satisfies the condition of Lemma~\ref{prop:contour} with $p=2$, we can therefore choose a subsequence $(n_j)_{j\ge 1}$ such that for all $x\in\bb R$ and $0\le b\le\gamma$,
\begin{align*}
	\frac{1}{\sqrt{2\pi}}\int_{-n_j}^{n_j}m_\Lambda(\xi)\fou_f(\xi) e^{\i\xi x}d\xi-\frac{1}{\sqrt{2\pi}}\int_{-n_j}^{n_j} m_{\Lambda}(\xi+\i b)\fou_f(\xi+\i b) e^{\i\xi x} e^{-bx} d\xi\to 0
\end{align*}
as $j\to\infty$. Note that the first integrand above converges to $g_0(x)$ and the second integral converges to $e^{-bx}g_b(x)$ in $\bmrm{L}(\bb R)$, which implies that $g_b(x)=e^{bx} g_0(x)$ for all $0\le b\le \gamma$, in particular when $b=\frac{1}{2}$. Now, from the definition of $\Lambda f$, it is clear that $\Lambda f(x)=e^{-\frac{x}{2}}g_{\frac{1}{2}}(x)=g_0(x)$. The above computation shows that $g_0\in\bmrm{L}(\bb R, e^{2bx})$ for all $0\le b\le \gamma$, which proves the first statement of the result. By means of Cauchy-Schwarz inequality, for any $0<b<\gamma$ we have
\begin{align*}
	\left(\int_{-\infty}^0 |g_0(x)|e^{bx} dx\right)^2&\le\int_{-\infty}^0 |g_0(x)|^2dx\int_{-\infty}^0 e^{2bx}dx<\infty \\
	\left(\int_{0}^\infty |g_0(x)| e^{bx}dx\right)^2&\le\int_{0}^\infty |g_0(x)|^2 e^{2\gamma x}dx\int_0^\infty e^{-2(\gamma-b)x} dx<\infty
\end{align*}
which shows that $\Lambda f=g_0\in \mathbf{L}^1(\bb R, e^{bx})$ for all $0<b<\gamma$. Hence, for any $0<b<\gamma$ one has
\begin{align*}
	\fou_{\Lambda f}(\xi+\i b)=\frac{1}{\sqrt{2\pi}}\int_{\bb R} g_0(x)e^{bx}e^{-\i \xi x} dx=\frac{1}{\sqrt{2\pi}}\int_{\bb R} g_b(x)e^{-\i \xi x} dx=m_{\Lambda}(\xi+\i b)\fou_f(\xi+\i b).
\end{align*}
Since $\Lambda f\in\bmrm{L}(\bb R)\cap \bmrm{L}(\bb R, e^{2\gamma x})$ and the right-hand side of the above equation extends continuously to the boundary of $\bb{S}_{(0,\gamma)}$, \eqref{eq:fourier_b} follows for all $0\le b\le \gamma$.
\end{proof}
\begin{prop}\hlabel{dense_range_intertwining}
Any $(\LL,\D(\LL))\in\mathscr{M}$ is densely defined with dense range in $\bmrm{L}(\bb{R},e)$ 
 and $\widehat{\cc F}^e_{\!\C^\infty_c(\bb{R})}$ is a core.
\end{prop}
\begin{proof}
Let $f\in\D(\LL)$ be such that  $\fouho _{f}(\xi)\neq 0$ a.e.. Then, for any $a,\xi\in\bb{R}$,  \[ \fouho_{\tau_a\LL f}(\xi)=e^{{\rm i}a-\frac{a}{2}}m_{\LL}\left(\xi+{\rm{i}}/2\right)\fouho_f (\xi)=m_{\LL}\left(\xi+{\rm{i}}/2\right)\fouho_{\tau_a f}(\xi)=\fouho_{\LL\tau_a f}(\xi),\] that is $\tau_a\LL f\in\Range(\LL)$. An application of Wiener's Tauberian theorem yields that the set $\Span\{\tau_a\LL f; \ a\in\bb{R}\}$ is dense in $\bmrm{L}(\bb{R},e)$. Next, we observe  that
\begin{align}
\D(\LL)=\left\{f\in\bmrm{L}(\bb{R},e); \: \fouho _{f}\in\bmrm{L}\left(\bb{R}, \left(1+\left|m_\LL\left(\xi+{\rm{i}}/2\right)\right|^2\right)\,d\xi\right)\right\}.
\end{align}
Since the function $1+\left| m_\LL\left(\cdot+\i/2\right)\right|^2$ is locally integrable, the corresponding weighted $\mathrm{L}^2$-space has $\C^\infty_c(\bb{R})$ as a dense subset. Also, the graph norm of $\LL$, which is given by $\|f\|_{\LL}:=\|f\|_{\bmrm{L}(\bb{R},e)}+\|\LL f\|_{\bmrm{L}(\bb{R},e)}$, is equivalent to the norm of $\bmrm{L}\left(\bb{R}, \left(1+\left|m_\LL\left(\xi+\i/2\right)\right|^2\right)\,d\xi\right)$. Thus, $\fouho_{\C^\infty_c(\bb{R})}$  is dense in $\D(\LL)$ with respect to the graph norm $\|\cdot\|_{\LL}$ and hence is a core of $(\LL,\D(\LL)).$
\end{proof}
The class $\mathscr{M}$ also has some nice algebraic properties that are given in the following.
\begin{prop}\hlabel{m-group}
The class $\mathscr{M}$ forms an abelian group with respect to operator multiplication (\hpt{R12}{in the sense that one considers the closure of the composition}) and it is closed under adjoints.
\end{prop}
\begin{proof}
We note that any $\LL\in\mathscr{M}$ is injective because from \eqref{fourier_mult}, $\LL f\equiv 0$ implies that $\fouho_{f}=0$ a.e., as $m_{\LL}(\xi+{\rm{i}}/2)\neq 0$ a.e., which yields $f\equiv 0$. From the definition, it is also clear that any operator in $\mathscr{M}$ is closed. Furthermore, any operator $(\LL,\D(\LL))\in\mathscr{M}$ admits an inverse $(\LL^{-1},\D(\LL^{-1}))$ which is also closed, and $\D(\LL^{-1})=\Range(\LL)$, $\Range(\LL^{-1})=\D(\LL)$. It is easy to see that $(\LL^{-1},\D(\LL^{-1}))\in\mathscr{M}$ with $m_{\LL^{-1}}=\frac{1}{m_{\LL}}$. For any two operators $\{(\LL_k,\D(\LL_k))\}_{k=1}^2$, \hpt{R13}{the closure $\ov{\LL_1\LL_2}$} of $\Lambda_1\Lambda_2$ is densely defined with $m_{\ov{\LL_1\LL_2}}=m_{\LL_1}m_{\LL_2}$, which proves that $(\ov{\LL_1\LL_2},\D(\ov{\LL_1\LL_2}))\in\mathscr{M}$. For $\LL\in\mathscr{M}$, we claim that $\widehat{\LL}$, the adjoint of $\LL$ is the Fourier multiplier operator with  multiplier  \hpt{C22}{$m_{\widehat{\LL}}(\xi+\i a)=\ov{m}_{\LL}(\xi+\i a)=m(-\xi+\i(1-a))$}. To see this, let $g\in\bmrm{L}(\bb{R},e)$ be such that the map
$$f\mapsto\langle \LL f, g\rangle_e=\left\langle m_{\LL}\left(\cdot+{\rm{i}}/2\right)\fouho_f, \fouho_g\right\rangle_{\bmrm{L}(\bb{R})}=\left\langle\fouho_f,\ov{m}_{\LL}\left(\cdot+{\rm{i}}/2\right)\fouho_g\right\rangle_{\bmrm{L}(\bb{R})}$$ is continuous. This is possible only if $\overline{m}_{\LL}\left(\cdot+{\i/2}\right)\fouho_g\in\bmrm{L}(\bb{R})$, which completes the proof of the proposition.
\end{proof}
Let $\mathscr{M}_e$ be the class of operators in $\mathscr{M}$ with the following additional property
\begin{align}\hlabel{eq:mult_strip}
\mbox{For all} \ (\LL,\D(\LL))\in\mathscr{M}_e, \ \left|m_{\LL}\left(\xi+{\rm{i}}/2\right)\right|\le Ce^{k|\xi|}, \  \mbox{ for all } \xi\in\bb R,
\end{align}
where $C,k$ are positive constants depending on $\LL$ but does not depend on $\xi$. Finally, we define $\overline{\mathscr{M}}$ (resp.~$\overline{\mathscr{M}}_e$) to be the class of operators $\LL$ in $\mathscr{M}$ (resp.~$\mathscr{M}_e$) such that the function $m_{\LL}$ is analytic on $\bb{S}_{(0,1)}$ and extends continuously to $\bb{S}_{[0,1]}$. One  notes that $(\LL,\D(\LL))\in\mathscr{M}$ is bounded if and only if $m_{\LL}(\cdot+{\rm{i}}/2)$ is bounded. We now show that Fourier multiplier operators satisfy the transitivity property when they act as intertwining operators.
\begin{prop}\hlabel{intertwining_equiv} Let $A,B\in\mathscr{B}(\bmrm{L}(\bb{R},e))$ and $\Lambda\in\mathscr{M}_e$ be such that $A\Lambda f=\Lambda Bf$ for all $f\in\D(\Lambda)$. Then $B\Lambda^{-1} f=\Lambda^{-1} Af$ for all $f\in\D(\Lambda^{-1})$. Similarly, if for $A,B,C\in\mathscr{B}(\bmrm{L}(\bb{R},e))$, there exist $\Lambda_1,\Lambda_2\in\mathscr{M}_e$ such that $A\Lambda_1=\Lambda_1 B$ on $\D(\Lambda_1)$ and $B\Lambda_2=\Lambda_2 C$ on $\D(\Lambda_2)$, then $A\ov{\Lambda_1\Lambda_2}=\ov{\Lambda_1\Lambda_2} \hpt{R15}{C}$ on $\D(\ov{\Lambda_1\Lambda_2})$, where $\ov{\Lambda_1\Lambda_2}$ is the closure of $\Lambda_1\Lambda_2$.
\end{prop}
\begin{proof}
If $\hpt{R16}{A\Lambda=\Lambda B}$ on $\D(\Lambda)$, then by injectivity of $\Lambda$, we have $B\Lambda^{-1}=\Lambda^{-1} A$ on $\Range(\Lambda)=\D(\Lambda^{-1})$. To prove the transitivity property, we consider  $A,B,C\in\mathscr{B}(\bmrm{L}(\bb{R},e))$ such that $A\Lambda_1=\Lambda_1 B$ on $\D(\Lambda_1)$ and $B\Lambda_2=\Lambda_2 C$ on $\D(\Lambda_2)$. Then, we already have $A\Lambda_1\Lambda_2=\Lambda_1\Lambda_2 B$ on $\{f\in\bmrm{L}(\bb{R},e); \: f\in\D(\Lambda_2), \ \Lambda_2 f\in\D(\Lambda_1)\}$.
By Proposition \ref{dense_range_intertwining}, the set $\widehat{\cc F}^e_{\!\C^\infty_c(\bb{R})}$ is a core for $\Lambda_1, \Lambda_2$, and $\ov{\Lambda_1\Lambda_2}$, where
\begin{align*}
	\D(\ov{\Lambda_1\Lambda_2})=\left\{f\in\bmrm{L}(\bb{R},e); \: \xi \mapsto m_{\Lambda_1}\left(\xi+{\rm{i}}/2\right)m_{\Lambda_2}\left(\xi+{\rm{i}}/2\right)\fouho _{f}(\xi)\in\bmrm{L}(\bb{R},e)\right\}.
\end{align*}
 Moreover, $\widehat{\cc F}^e_{\!\C^\infty_c(\bb{R})}\subset\{f\in\bmrm{L}(\bb{R},e); \: f\in\D(\Lambda_2), \ \Lambda_2 f\in\D(\Lambda_1)\}$. Hence, the identity $A\ov{\Lambda_1\Lambda_2}=\ov{\Lambda_1\Lambda_2} B$ holds on $\D(\ov{\Lambda_1\Lambda_2})$ by a density argument.

%

\end{proof}
\subsection{Some basic facts about generators}\hlabel{subsec:generator} In the theory of Markov semigroups, there are several notions for defining the generators. It depends on the underlying Banach space and the topology one is interested about. In this article we are concerned about the $\bmrm{L}$-generators of the Markov semigroups. The main purpose of this section is to connect the three types of generators, namely, the Feller generator, the Dynkin characteristic operator, and the $\bmrm{L}$-generator. We begin with the following result due to  Dynkin that relates the characteristic operator and the infinitesimal generator of a Feller process.
\begin{lem}\cite[Theorem 5.5, Chapter V.3]{ebdynkin} \hlabel{lem:Dynkin}
Let $(A_\ttt{F},\D(A_\ttt{F}))$ and $(A_\ttt{D},\D(A_\ttt{D}))$ be respectively the Feller infinitesimal generator and Dynkin characteristic operator of a Feller process with state space $E$. If for some $\hpt{R18}{f\in\D(A_\ttt{D})\cap \C_0(E)}$, $A_\ttt{D}f\in \C_0(E)$, then $f\in\D(A_\ttt{F})$ and $A_\ttt{F}f=A_\ttt{D}f$.
\end{lem}
The next result is about the equality of the strong and weak generator of a Markov semigroup on a Hilbert space and the proof can be found in \cite{Pazy}.
\begin{lem}\cite[Theorem 1.3]{Pazy}\hlabel{lem:weak=strong}
Let $(P_t)_{t\geq0}$ be a $\cc{C}_0$-contraction semigroup 
of bounded operators on a Banach space ${\bf{Y}}$. Let $(\tilde{A},\D(\tilde{A}))$ and $(A,\D(A))$ denote the weak and strong generator of this semigroup respectively. Then, $(\tilde{A},\D(\tilde{A}))=(A,\D(A))$.
\end{lem}
Finally, we have the following result that relates the Feller infinitesimal generator with the $\mathbf L^2$-generator of a Markov process.
\begin{lem}\hlabel{lem:feller=L2}
For a Feller semigroup $(P_t)_{t\ge 0}$ \hpt{R19}{on $\bb{R}$} with an excessive measure  $\varepsilon$, i.e.~$\varepsilon P_t \leq \varepsilon $, let $(P^{(\ttt{2})}_t)_{t\ge 0}$ denote its $\bmrm{L}(\bb{R},\varepsilon)$-extension. Let $(A_\ttt{F},\D(A_\ttt{F}))$ and $\ (A_\ttt{2},\D(A_\ttt{2}))$ be the generators of $P$ and $ \ P^{(\ttt{2})}$ respectively. If $f\in\D(A_\ttt{F})\cap\bmrm{L}(\bb{R},\varepsilon)$ is such that $A_\ttt{F}f\in \bmrm{L}(\bb{R},\varepsilon)$, then $f\in\D(A_\ttt{2})$ and $A_\ttt{F}f=A_\ttt{2}f$.
\end{lem}
\begin{proof}
Consider any $f\in\hptg{C28}{\D(A_\ttt{F})}\cap\bmrm{L}(\bb R,\varepsilon)$ such that $A_\ttt{F}f\in\bmrm{L}(\bb{R},\varepsilon)$ and $g\in\bmrm{L}(\bb{R},\varepsilon)$. Then, we know that, for all $t\ge 0$,
\hpt{R20}{$P_t$ is a contraction on $\C_0(\bb{R})\cap \bmrm{L}(\bb{R},\varepsilon)$
with respect to the $\bmrm{L}(\bb{R},\varepsilon)$ norm}, and,
\begin{equation}\hlabel{weak}
\begin{aligned}
 \left\langle(P^{(\ttt{2})}_tf-f)/t,g\right\rangle_{\bmrm{L}(\bb{R},\varepsilon)}=\left\langle (P_tf-f)/t,g\right\rangle_{\bmrm{L}(\bb{R},\varepsilon)}.
\end{aligned}
\end{equation}
Since $A_\ttt{F}$ is the generator of the Feller semigroup $(P_t)_{t\ge 0}$, we have
\begin{align*}
\lim_{t\to 0}\left\|(P_t f-f)/t-A_\ttt{F} f\right\|_{\infty}= 0
\end{align*}
where $\|\cdot\|_{\infty}$ stands for the supremum norm in $\C_0(\bb R)$. Since $A_\ttt{F} f\in\bmrm{L}(\bb{R},\varepsilon)$, we observe that for any $t>0$,
\begin{equation*}
\left\|\frac{P_t f-f}{t}\right\|_{\hpt{R21}{\bmrm{L}(\bb{R},\varepsilon)}}=\frac{1}{t}\left\|\int_0^t P_s A_\ttt{F} fds\right\|_{\bmrm{L}(\bb{R},\varepsilon)}\le\frac{1}{t}\int_0^t \|P_s A_\ttt{F} f\|_{\bmrm{L}(\bb{R},\varepsilon)}ds\le \|A_\ttt{F} f\|_{\bmrm{L}(\bb{R},\varepsilon)}.
\end{equation*}
Now, \eqref{weak} implies that if $g\in\mathbf{L}^1(\bb{R},\varepsilon)\ \cap\ \mathbf{L}^2(\bb{R},\varepsilon)$, then
$
\lim_{t\to 0}\left\langle\frac{P_t f-f}{t},g\right\rangle_{\bmrm{L}(\bb{R},\varepsilon)}=\langle A_\ttt{F} f,g\rangle_{\bmrm{L}(\bb{R},\varepsilon)}.
$
On the other hand, since $\{(P_tf-f)/t ; \: t>0\}$ is bounded in $\mathbf{L}^2(\bb R,\varepsilon)$, it is therefore weakly compact in $\mathbf{L}^2(\bb R,\varepsilon)$. Let $f^*$ be a weak subsequential limit of $(P_tf-f)/t$ as $t\to 0$. Then, by \eqref{weak} and the discussion above, we have
$
\langle f^*,g\rangle_{\bmrm{L}(\bb{R},\varepsilon)}=\langle f,g\rangle_{\bmrm{L}(\bb{R},\varepsilon)} \ \mbox{ for all }  \ g\in\mathbf{L}^1(\bb{R},\varepsilon)\cap\bmrm{L}(\bb{R},\varepsilon).
$
Since $\mathbf{L}^1(\bb{R},\varepsilon)\cap\bmrm{L}(\bb{R},\varepsilon)$ is dense in $\bmrm{L}(\bb{R},\varepsilon)$, we infer that
$\frac{P^{(2)}_tf-f}{t}\to A_\ttt{F} f$ weakly for all $f\in\D(A_\ttt{F})\cap\bmrm{L}(\bb{R},\varepsilon)$, as $t\to 0$. Finally, invoking Lemma~\ref{lem:weak=strong}, we conclude that $A_\ttt{F} f$ is also the strong limit of $(P^{(\ttt{2})}_t f-f)/t$, which completes  the proof.
\end{proof}
We close this section with the following lemma that shows that the $\cc{C}_0$-contraction semigroup generation property of operators is preserved under the weak similarity relation.
\begin{lem}\hlabel{lem:generator_intertwining}
	Let the closure of $(B,\cc D)$ be the generator of a $\cc{C}_0$-contraction semigroup on a Banach space ${\rm X}$. Let  $(A,\D(A))$ be a  densely defined operator on ${\rm X}$ such that
	\begin{align}\hlabel{eq:weak_sim_AB}
		A\Lambda =\Lambda B  \ \mbox{ on } \ \cc D
	\end{align}
	where $\Lambda \in \mathscr{B}({\rm X})$ with dense range, and $A$ is dissipative on $\Lambda(\cc D)$. Then, the closure of $(A,\Lambda(\cc D))$ generates a $\cc{C}_0$-contraction semigroup on ${\rm X}$.

\end{lem}
\begin{proof}
	We first note that $\cc D$ is a dense subset of ${\rm X}$ as it is a core for the generator of a $\cc C_0$-contraction semigroup on ${\rm X}$. Next, by the Hille-Yosida Theorem, $(\alpha I-B)(\cc D)$ is dense in ${\rm X}$ for any $\alpha>0$. Since $\Lambda$ is bounded and $\Range(\Lambda)$ is dense in ${\rm X}$, it follows that both $\Lambda(\cc D)$ and $\Lambda((\alpha I-B)(\cc D))$ are dense subsets of ${\rm X}$. Now, using \eqref{eq:weak_sim_AB}, we have $\Lambda((\alpha I-B)(\cc D))=(\alpha I-A)(\Lambda(\cc D))$, which shows that the latter subset is dense in ${\rm X}$ for any $\alpha>0$. Since $A$ is also dissipative on $\Lambda(\cc D)$, the proof is completed by applying the Hille-Yosida Theorem on the operator $(A,\Lambda(\cc D))$.
\end{proof}

\subsection{Bernstein-gamma functions and a functional equation}\hlabel{bernstein}
We recall, from \eqref{eq:phi_def}, that $\phi\in\B$ can be written as $\phi(z)=\phi(0)+\ttt{d}z+\int_0^\infty (1-e^{-zy})\nu(dy), \  z\in\bb{C}_{[0,\infty)}$, where $\nu$ is a positive measure such that $\int_0^\infty (y\wedge 1)\nu(dy)<\infty$ and $\phi(0),\ttt{d}\ge 0$. Next, we define
$$W_{\phi}(z)=\frac{e^{-\gamma_\phi z}}{\phi(z)}\prod_{k=1}^\infty\frac{\phi(z)}{\phi(k+z)}e^{\frac{\phi'(k)}{\phi(k)}}, \ \ z\in\bb{C}_{(0,\infty)},$$
where $\gamma_\phi=\lim_{n\to\infty}\left(\sum_{k=1}^n\frac{\phi'(k)}{\phi(k)}-\ln\phi(n)\right)\in\left[-\ln\phi(1),\frac{\phi'(1)}{\phi(1)}-\ln\phi(1)\right]$.
From \cite[Section 4]{patie2018} and \cite[Theorem 4.2]{patiesavov}, it is known that $W_\phi$ is a solution to the functional equation
\begin{equation}\hlabel{eq:fe}
f(z+1)=\phi(z)f(z), \ f(0)=1,  \ \ \ z\in\bb{C}_{(0,\infty)},
\end{equation}
and, satisfies the following important properties.
\begin{enumerate}[(W1)]
\item\hlabel{p1} $W_\phi$ is the unique positive definite solution to the above functional equation \eqref{eq:fe}, that is it is the unique solution in the set of Mellin transforms of probability measures on $(0,\infty)$
\item\hlabel{p2} $W_\phi$ is zero-free on $\bb{C}_{(0,\infty)}$ and $\ov{W}(z)=W(\ov{z})$ for all $z\in\bb{C}_{(0,\infty)}$
\item\hlabel{p3} $W_\phi$ is analytic on $\bb{C}_{(0,\infty)}$
\item\hlabel{p4} For any $z=a+\mathrm{i}\xi$ with $a>0$
\begin{equation}\hlabel{eq:bernstein_gamma_bound}
|W_\phi(z)|=\frac{\sqrt{\phi(1)}}{\sqrt{\phi(a)\phi(1+a)\hptg{C29}{|\phi(z)|}}}e^{G_\phi(a)-A_\phi(z)-E_\phi(z)-R_\phi(a)}
\end{equation}
where $G_\phi(a)=\int_1^{1+a}\ln\phi(u)\,du\le a\ln\phi(1+a)$, $A_\phi(z)=\int_a^\infty\ln\left(\frac{|\phi(u+{\rm{i}}\xi)   |}{\phi(u)}\right)du\in [0,\hptg{C30}{\frac{\pi}{2}|\xi|}]$ and $\sup\limits_{z\in\bb{C}_{(d,\infty)}}|E_\phi(z)|+\sup\limits_{a>d}|R_\phi(a)|<\infty$ for all $d>0$. We again refer to  \cite[Theorem 3.1 and Theorem 3.2]{patie2018} for the definition of the quantities $G_\phi, A_\phi, E_\phi,R_\phi$.
\end{enumerate}
$W_{\phi}$ is called the Bernstein-gamma function corresponding to $\phi$. These functions have been studied extensively in \cite{patie2018}. The next proposition derives the analyticity and growth bounds on the multiplier functions defined in \eqref{eq:Mult}. We deal with the one-dimensional case here, which is the main ingredient for all the results in this paper.
\begin{prop}\hlabel{cor:multiplier}
Let $\psi\in\mathbf{N}(\bb{R})$ be such that $\psi(\xi)=\phi_+(-\rm{i}\xi)\phi_-(\rm{i}\xi)$ for all $\xi\in\bb{R}$ and $\psi_\beta(\xi)=\xi^2-{\rm{i}}\beta\xi$, $\beta>0$. Then, the function $\WBpsi:\bb{S}_{(0,1)}\to\bb{C}$ defined by
\begin{align*}
	 \WBpsi(z)=\frac{\Gamma(1+{\rm{i}}z+\beta) W_{\phi_+}(-{\rm{i}}z)}{\Gamma(-{\rm{i}}z)W_{\phi_-}(1+{\rm{i}}z)}
\end{align*}
satisfies the functional equation
\begin{align}\hlabel{eq:fneq}
\WBpsi(\xi)\psi(\xi)=\WBpsi(\xi+{\rm{i}})   \psi_\beta(\xi),  \ \  \xi\in\bb{R},
\end{align}
and enjoys the following properties.
\begin{enumerate}[(1)]
\item \hlabel{it:W_1}$\WBpsi$ is analytic in the strip $\bb{S}_{(0,1)}$ and extends continuously on $\bb{S}_{[0,1]}$.
\item \hlabel{it:kap} Let $\psi\in\mathbf{N}_+(\bb{R})$. Then,
\begin{align}\hlabel{eq:ratio_+_bdd}
	\sup_{a\in [0,1]}\left|W^{(\beta)}_\psi(\xi+\mathrm{i}a)\right|=\mathrm{O}(|\xi|^u) \textrm{ for some $u>\beta$}.
\end{align}
\end{enumerate}
\end{prop}
\begin{rem}\hlabel{rem:fneq_rem}
From the proof of Proposition~\ref{cor:multiplier}, it is clear that for any $\psi_1,\psi_2\in\mathbf{N}(\bb{R})$, with $\psi_1(\xi)=\phi^{(1)}_+(-{\rm{i}}\xi  )\phi^{(1)}_-({\rm{i}}\xi  )$, $\psi_2(\xi)=\phi^{(2)}_+(-{\rm{i}}\xi  )\phi^{(2)}_-({\rm{i}}\xi  )$, if we assume that
$$ \{z\in {\rm{i}}\bb{R};  \ \phi^{(2)}_\pm(z)=0\}\subseteq\{0\} \text{ and } \phi^{(1)}_-(0)>0,$$ then the solution of the functional equation
\begin{align}
f(\xi)\psi_2(\xi)=f(\xi+{\rm{i}})   \psi_1(\xi) \  \mbox{ for all } \xi\in\bb{R}\setminus\{0\}
\end{align}
is given by $$f(\xi)=W[\psi_2;\psi_1](\xi)=\frac{W_{\phi^{(2)}_+}(-{\rm{i}}\xi  )}{W_{\phi^{(1)}_+}(-{\rm{i}}\xi  )}\frac{W_{\phi^{(1)}_-}(1+{\rm{i}}\xi  )}{W_{\phi^{(2)}_-}(1+{\rm{i}}\xi  )}. $$
Also, $W[\psi_2; \psi_1]\in\A_{(0,1]}$. In particular, taking $\psi_2=\psi$, $\psi_1\equiv 1$, the function $W(\xi)=\frac{W_{\phi_+}(-{\rm{i}}\xi  )}{W_{\phi_-}(1+{\rm{i}}\xi  )}$ satisfies
\begin{align}
W(\xi+{\rm{i}})   =\psi(\xi)W(\xi) \ \ \ \mbox{ for all } \xi\in\bb{R}\setminus\{0\},
\end{align}
and $W\in\A_{(0,1]}$. 
\end{rem}
\subsection{Proof of Proposition~\ref{cor:multiplier}} Let $\psi$ be such that $$\psi(\xi)=\phi_+(-\rm i\xi)\phi_-(\rm i\xi)$$ for all $\xi\in\bb{R}$ with $\phi_+,\phi_-\in\B$.
Let us define the following pair of functions
\begin{align*}
\begin{cases}
W_+(z)=\frac{W_{\phi_+}(-{\rm i} z)}{\Gamma({-\rm i}z)}, &  z\in\bb{S}_{(0,\infty)} \vspace{.3 cm} \\
W_-(z)=\frac{\Gamma(1+{\rm i}z+\beta)}{W_{\phi_-}(1+{\rm i}z)}, &  \mbox{for any $\beta>0$ and $z\in\bb{S}_{(-\infty,1)}$}.
\end{cases}
\end{align*}
Now, for $z\in\bb{S}_{(0,\infty)}$, we have
\begin{align*}
W_+(z+{\rm i})&=\frac{W_{\phi_+}(-{\rm i}z+1)}{\Gamma(-{\rm i}z+1)}=\frac{\phi_+(-{\rm i}z)}{-{\rm i}z}\frac{W_{\phi_+}(-{\rm i}z)}{\Gamma(-{\rm i}z)}=\frac{\phi_+(-{\rm i}z)}{-{\rm i}z}W_+(z)
\end{align*}
i.e.~$W_+(z)=W_+(z+\i)\frac{-{\rm i}z}{\phi_+(-{\rm i}z)}$. Along the same vein, we have, for all $z\in\bb{S}_{(-\infty,0)}$,
\begin{align*}
W_-(z+\i)=W_-(z)\frac{\phi_-({\rm i}z)}{{\rm i}z+\beta}.
\end{align*}
 Next, using  the property (W\ref{p3}) for $W_{\phi_+}$, we deduce that $W_+$ is analytic on $\hptg{C32}{\bb{S}_{(0,\infty)}}$ and extends  continuously to the punctured line $\bb{R}\setminus\{0\}$. On the other hand, $W_-$ is also analytic on $\bb{S}_{(-\infty,1)}$ and extends continuously on the line $\bb R+{\rm{i}}$ for any value of $\phi_-(0)$. To deal with the extension of $W_+$ to the entire real line, we observe that
 \begin{align}\hlabel{eq:phi_derivative}
 	\lim_{\substack{z\to 0 \\ \Im(z)>0}}\frac{-{\rm i}z}{\phi_+(-{\rm i}z)}=\lim_{\substack{z\to 0 \\ \Re(z)>0}}\frac{z}{\phi_+(z)}=\begin{cases} 0, & \mbox{if $\phi_+(0)>0$} \\
 \frac{1}{\phi_+'(0)}<\infty, & \mbox{if $\phi_+(0)=0$}
 \end{cases}
 \end{align}
and, hence $W_+$ extends continuously to $\bb{S}_{[0,\infty)}$. Now, we define $W^{(\beta)}_\psi(z)=W_+(z)W_-(z)$ for all $z\in\bb{S}_{(0,1)}$. Clearly, the above results entail that $W^{(\beta)}_\psi$ is analytic on $\bb{S}_{(0,1)}$, extends continuously to its boundary and solves \eqref{eq:fneq}. {Next, to prove \eqref{eq:ratio_+_bdd}, we resort to Phragm\'en-Lindel\"of principle. In the next two lemmas we obtain some elementary bounds of Bernstein functions, and an exponential bound of $W^{(\beta)}_\psi$ in the strip $\bb{S}_{[0,1]}$.
  	\begin{lem}\hlabel{lem:bound_bernstein}
  		For any $\phi\in\B$, $a,b\ge 0$, and for all $\xi\in\bb{R}$ one has
  		\begin{align}
  			|\phi(a+\i\xi)-\phi(b+\i\xi)|\le \phi(|a-b|)-\phi(0). \notag
  		\end{align}
  		Thus, for any $a,b>0$, as $|\xi|\to\infty$, $\phi(a+\i\xi)\asymp\phi(b+\i\xi)$.
  	\end{lem}
  	\begin{proof}
  		Let $\phi$ be of the form \eqref{eq:phi_def}. Then, for any $a,b\ge 0$ and $\xi\in\bb{R}$, we can write
  		\begin{align*}
  			\phi(a+\i\xi)-\phi(b+\i\xi)=\ttt{d}(a-b)+\int_{\bb{R}}\left(e^{-(a+\i\xi)y}-e^{-(b+\i\xi)y}\right)\nu(dy)
  		\end{align*}
  		Now, let us assume that $a<b$. Applying triangle inequality to the right-hand side of the above equality, we obtain
  		\begin{align*}
  			|\phi(a+\i\xi)-\phi(b+\i\xi)|\le \ttt{d}(b-a)+\int_{\bb{R}} (1-e^{-(b-a)y}) e^{-ay}\nu(dy)\le\phi(b-a)-\phi(0).
  		\end{align*}
  		Since $|\phi(a+\i\xi)|\ge \phi(a)>0$ for any $a>0$, from the above inequality we obtain that
  		\begin{align*}
  			\left|\frac{\phi(a+\i\xi)}{\phi(b+\i\xi)}-1\right|\le\frac{\phi(|a-b|)-\phi(0)}{\phi(b)}, \ \ \left|\frac{\phi(b+\i\xi)}{\phi(a+\i\xi)}-1\right|\le \frac{\phi(|a-b|)-\phi(0)}{\phi(a)}
  		\end{align*}
  		which completes the proof of the lemma.
  	\end{proof}
  \hpt{R22}{
  \begin{lem}\hlabel{lem:W_beta_bound}
  	Let $\psi\in\NN(\bb R)$. Then, for any $\epsilon>0$ there exists $C_\eps>0$ independent of $z$ such that $|W^{(\beta)}_\psi(z)|\le C_\eps e^{(\frac{\pi}{2}+\eps)|z|}$ all $z\in\bb{S}_{[0,1]}$.
  \end{lem}
\begin{proof}
	From the definition of the function $W^{(\beta)}_\psi$ and exploiting the recursion in \eqref{eq:fneq} along with the representation in \eqref{eq:bernstein_gamma_bound}, it suffices to prove that there exist $u,C>0$ such that
	\begin{align}\hlabel{eq:phi_bound}
		\left|\frac{z}{\phi_+(z)}\right|\le  C|z|^{u} \quad \text{for all $z\in \bb{C}_{[0,1]}$},
	\end{align}
	where we recall that $z\in\bb{C}_{[0,1]}=\{z\in\bb{C};\ \Re(z)\in [0,1]\}$. Now, for any $a>0$ and $\xi\in\bb R$, applying Lemma~\ref{lem:bound_bernstein} we get
	\begin{align*}
		\left|\frac{\phi_+(\i\xi)}{\phi_+(a+\i\xi)}-1\right|\le \frac{\phi_+(a)-\phi_+(0)}{|\phi_+(a+\i\xi)|}\le\frac{\phi_+(a)-\phi_+(0)}{\phi_+(a)}\le 1,
	\end{align*}
which shows that $|\phi_+(a+\i\xi)|\ge\frac{1}{2}|\phi_+(\i\xi)|$ for all $a>0, \xi\in\bb R$. Therefore, invoking \eqref{N_+} it follows that $\liminf_{|\xi|\to\infty} |\xi|^\kappa|\phi_+(a+\i\xi)|>0$, uniformly with respect to $a\in [0,1]$. Since by \eqref{eq:phi_derivative}, the mapping  $z\mapsto |\frac{z}{\phi_+(z)}|$ is bounded near $0$, \eqref{eq:phi_bound} holds with $u=\kappa+1$.
\end{proof}
}
To achieve the upper bound in \eqref{eq:ratio_+_bdd} using Phragm\'en-Lindel\"of principle, Lemma~\ref{lem:W_beta_bound} implies that it is enough to prove \eqref{eq:ratio_+_bdd} when $a=0,1$ respectively. Let us deal with the case $a=1$. Converting everything in terms of the function $W^{(\beta)}_\psi$, we note that
\begin{align*}
	W^{(\beta)}_\psi(\xi+\i)=\frac{W_{\phi_+}(1+\mathrm{i}\xi)\Gamma(\beta-\mathrm{i}\xi)\phi_-(-\mathrm{i}\xi)}{W_{\phi_-}(1-\mathrm{i}\xi)\Gamma(1+\mathrm{i}\xi)}~.
\end{align*}
By Stirling asymptotic formula for the gamma function and recalling from \cite[Proposition~3.1(4)]{patie2018} that $\phi_-(-\i\xi)=-\i\ttt{d}_-\xi+\mathrm{o}(|\xi|)$,
we get
\begin{align}\hlabel{eq:gamma_phi_ratio}
\frac{\Gamma(\beta-\mathrm{i}\xi)\phi_-(-\mathrm{i}\xi)}{\Gamma(1+\mathrm{i}\xi)}=\mathrm{O}(|\xi|^\beta).
\end{align}
So, it is enough to study the asymptotic behavior of the ratio $\frac{W_{\phi_+}(1+\mathrm{i}\xi)}{W_{\phi_-}(1-\mathrm{i}\xi)}$. For this, we proceed by proving the following lemma.

Now, coming back to the proof of the proposition, from the property (W\ref{p4}) and using the above lemma, we obtain
\begin{align}\hlabel{eq:bounded_ratio}
	\left|\frac{W_{\phi_+}(1+\mathrm{i}\xi)}{W_{\phi_-}(1-\mathrm{i}\xi)}\frac{W_{\phi_-}\left(1/2+\mathrm{i}\xi\right)}{W_{\phi_+}\left(1/2-\mathrm{i}\xi\right)}\right|\lesssim e^{A_{\phi_-}(1+\mathrm{i}\xi)-A_{\phi_-}\left(1/2+{\rm{i}}\xi\right)+A_{\phi_+}\left(1/2+{\rm{i}}\xi\right)-A_{\phi_+}(1+\mathrm{i}\xi)}
\end{align}
where $f\lesssim g$ means that  $\frac{f(\xi)}{g(\xi)}$ is bounded above as $|\xi|\to\infty$.
Since $\hptg{C34}{\lim_{|\xi|\to\infty}\frac{\phi_+(u+i\xi)}{u+i\xi}=\ttt{d}_+}$ locally uniformly in $u>0$, from the definition of $A_{\phi_+}$ in \eqref{eq:bernstein_gamma_bound}, we have that
\hptg{C35}{
\begin{align*}
	\left|A_{\phi_+}(1+\mathrm{i}\xi)-A_{\phi_+}\left(1/2+{\rm{i}}\xi\right)\right|=\int_{\frac{1}{2}}^1\ln\left(\frac{|\phi_+(u+\i\xi)|}{\phi_+(u)}\right)du\le\int_{\frac{1}{2}}^1\ln\left(\frac{|\phi_+(u+\i\xi)|}{\phi_+(\frac{1}{2})}\right)du
\end{align*}}
where the right most integral grows at most at the order of $\ln|\xi|$ as $|\xi|\to\infty$.
On the other hand, $A_{\phi_-}(1+\i\xi)\le A_{\phi_-}(1/2+\i\xi)$ for all $\xi\in\bb R$, see \cite[Theorem~3.2(1)]{patie2018}. Therefore, by the boundedness of the ratio $\frac{W_{\phi_+}\left(\frac{1}{2}-{\rm{i}}\xi\right)}{W_{\phi_-}\left(\frac{1}{2}+{\rm{i}}\xi\right)}$ in \eqref{eq:bounded_ratio} along with the estimate in \eqref{eq:gamma_phi_ratio}, it follows that $W^{(\beta)}_\psi(1+{\rm{i}}\xi)=\mathrm{O}(|\xi|^u)$ for some $u>\beta$. The case $a=0$ can be handled in the similar fashion after observing that $W_{\phi_\pm}(-\i\xi)=W_{\phi_\pm}(1-\i\xi)/\phi_\pm(-\i\xi)$ for any $\xi\in\bb{R}$. Since there exists $\kappa>0$ such that $\liminf_{|\xi|\to\infty} |\xi|^\kappa|\phi_+(-\i\xi)|>0$, using the similar idea as before, one can show that the following ratio
\begin{align*}
	\left|\frac{W_{\phi_+}(-\i\xi)}{W_{\phi_-}(1+\mathrm{i}\xi)}\frac{W_{\phi_-}\left(\frac{1}{2}+\mathrm{i}\xi\right)}{W_{\phi_+}\left(\frac{1}{2}-\mathrm{i}\xi\right)}\right|
\end{align*}
grows at most polynomially as $|\xi|\to\infty$. Hence, the proof of the proposition is completed by applying the Phragm\'en-Lindel\"of principle to the function $z\mapsto(2+z^2)^{-\frac{u}{2}}W^{(\beta)}_\psi(z)$ in the strip $\bb{S}_{[0,1]}$. \qed

\subsection{Proof of Proposition~\ref{prop:ratio_criterion}} \hlabel{sec:prop2.2} From \cite[Theorem 6.2(1)]{patiesavov}, we know that for any $\xi \in \bb{R}$,
	\begin{align}\hlabel{eq:ratio-gene2}
		\left|W_{\phi}\left(\frac{1}{2}+{\rm{i}}\xi\right)   \right|\asymp\frac{\sqrt{\phi(\frac{1}{2})}W_{\phi}(\frac{1}{2})}{\sqrt{|\phi\left(\frac{1}{2}+{\rm{i}}\xi\right)   |}}e^{-|\xi|\Theta_\phi(|\xi|)}.
	\end{align}
	Clearly, this proves the estimate in \eqref{eq:ration-gene}.
	When condition \eqref{it:ratio_bdd} holds, the ratio \[\left|\frac{W_{\phi_+}\left(\frac{1}{2}+{\rm{i}}\xi\right)}{W_{\phi_-}\left(\frac{1}{2}+{\rm{i}}\xi\right)   }\right|\] decays exponentially, thanks to \eqref{eq:ration-gene}. Moreover, we note that $\liminf_{|\xi|\to\infty} |\phi_+(\i\xi)|=\infty$ if $\ttt{d}_+>0$. Now, writing $\ov{\nu}_\pm(r)=\nu_\pm(r,\infty)$, \cite[Theorem 3.3(2)]{patie2018} ensures that if $\ttt{d}_+>0$, $\ttt{d}_-=0$ and $\overline{\nu}_-\in\mathbf{RV}(\alpha)$ with $r\mapsto\frac{\overline{\nu}_-(r)}{r^\alpha}$ being quasi-monotone, then $\lim_{|\xi|\to\infty}\Theta_{\phi_+}(|\xi|)=\pi/2$ and $\lim_{|\xi|\to\infty}\Theta_{\phi_-}(|\xi|)=\pi\alpha/2$. When $\ttt{d}_+=\ttt{d}_-=0$, $\ov{\nu}_\pm\in\mathbf{RV}(\alpha_\pm)$ and $r\mapsto\frac{\ov{\nu}_\pm(r)}{r^{\alpha_\pm}}$ are quasi-monotone, the previous reference entails that $\lim_{|\xi|\to\infty}\Theta_{\phi_\pm}(|\xi|)=\pi\alpha_\pm/2$. Since $\alpha_+>\alpha_-$, applying \eqref{eq:ration-gene} it follows that $\psi\in\mathbf{N}_+(\bb{R})$ and $m_\psi\in\bmrm{L}(\bb R)$. This proves \eqref{it:ratio_bdd}. To prove \eqref{it:p3}, invoking \cite[Proposition~6.2]{patie2018}, we know that if $\ttt{d}_-=0$ or $\ov{\nu}_{-}(0)=\infty$, then for any $u>0$,
	\begin{align*}
		\lim_{|\xi|\to\infty}|\xi|^u\left|\frac{\Gamma\left(\frac{1}{2}+\mathrm{i}\xi\right)}{W_{\phi_-}\left(\frac{1}{2}+\mathrm{i}\xi\right)}\right|=0.
	\end{align*} On the other hand, if $\ttt{d}_+>0$ and $\ov{\nu}_+(0)<\infty$, then
	\begin{align*}
		\lim_{|\xi|\to\infty}|\xi|^u\left|\frac{\Gamma\left(\frac{1}{2}+\mathrm{i}\xi\right)}{W_{\phi_+}\left(\frac{1}{2}+\mathrm{i}\xi\right)}\right|=\infty
	\end{align*}
whenever $u>\frac{\phi_+(0)+\ov{\nu}_+(0)}{\ttt{d}_+}$. This proves \eqref{it:p3}.

In both of the above two cases, we note that $m_{\psi}$ and $m_{\ov\psi}$ cannot be in $\bmrm{L}(\bb R)$ simultaneously as $m_{\ov\psi}(\xi)=1/\ov{m}_{\psi}(\xi)$ for all $\xi\in\bb R$.

\section{Proof of Theorem~\ref{thm1:main_thm_1}: the one dimensional case} \hlabel{sec:proof} Since the proof of this theorem  is rather long, we split it into two parts  and  start by sketching the main ideas. We first prove it in the one dimensional case, that is $d=1$, and thus $M=\rm{Id}$. The general case  will follow by tensorization and  similarity transform techniques and it is postponed to Section \ref{sec:proof2}. In Theorem~\ref{thm:intertwining_pdo}, we prove the one dimensional case where we first derive the weak similarity relations at the level of pseudo-differential operators. Then, using the results from Subsection~\ref{subsec:generator}, we show that those pseudo-differential operators generate Markov semigroups and lift the weak similarity relations  at the level of semigroups. Subsequently, using Dynkin's theorem (see Lemma~\ref{lem:Dynkin}) and uniqueness of semigroups generated by operators defined on a core, we show that these semigroups are indeed the ones corresponding to the log-self-similar Markov processes. As a by-product of the weak similarity identity \eqref{eq:bdd_intertwining}, we obtain the core and integro-differential representation of the $\LRe$-generators of these semigroups, see~Proposition~\ref{prop:core_ido}, by using some approximation techniques motivated from \cite{patie2018}.

Recall from Subsection~\ref{ss:lamperti} that for $\psi\in\NN_b(\bb R)$, $P[\psi]$ stands for the $\cc{C}_0$-contraction semigroup on $\LRe$, which coincides with the log-self-similar Feller semigroup corresponding to the L\'evy-Khintchine exponent $\psi$, when restricted on $\C_0(\bb R)$. We also denote the $\LRe$-generator of $P[\psi]$ by $A_\ttt{2}[\psi]$.
\begin{thm}\hlabel{thm:intertwining_pdo}
\begin{enumerate}[(1)]
\item\hlabel{main_thm_1:it:0a} Let $\psi_0(\xi)=\xi^2, \: \xi \in \bb{R}$, then the closure of the operator $(A_{\ttt{PDO}}[\psi_0],\C^\infty_c(\bb R))$ in $\LRe$ generates the log-squared Bessel semigroup $Q$ on $\bmrm{L}(\bb{R},e)$.
\item\hlabel{main_thm_1:it:1a} For any $\psi\in\NN_b(\bb R)$, the shifted Fourier multiplier operator $\Lambda_\psi\in\mathscr{M}_e$, where $m_{\Lambda_\psi}(z)=\frac{W_{\phi_+}(-{\rm{i}}z   )}{W_{\phi_-}(1+{\rm{i}}z)   }\frac{\Gamma(1+{\rm{i}}z)   }{\Gamma(-{\rm{i}}z   )}$ for all $z\in\bb{S}_{(0,1)}$. In particular, when $\psi\in\NN_+(\bb R)$, $\Lambda_\psi\in\mathscr{B}(\LRe)$.
\item\hlabel{main_thm_1:it:1b} For any $\psi\in\NN_b(\bb R)$, we have $A_\ttt{PDO}[\psi]\in\mathrm{WS}(A_\ttt{PDO}[\psi_0])$. More specifically, recalling that $\cc{D}(\bb R) = \{f\in \LRe;\: \cc{F}_f \textrm{ is entire and satisfies \eqref{eq:udec}}\}$,  we have $\cc{D}(\bb R)$ is dense in $\LRe$, and,
\begin{align}\hlabel{eq:intertwining_pdo_0_D}
A_\ttt{PDO}[\psi]\Lambda_\psi =\Lambda_\psi A_\ttt{PDO}[\psi_0]\ \mbox{ on } \ \cD(\bb{R}).
\end{align}
In particular, if $\psi\in\NN_+(\bb R)$ then \eqref{eq:intertwining_pdo_0_D} holds on $\C^\infty_c(\bb R)$.
\item \hlabel{main_thm_1:it:1bc} If $\psi\in\NN_+(\bb R)$ then the closure of the operator $(A_\ttt{PDO}[\psi],\Lambda_\psi(\C^\infty_c(\bb R)))$ in $\LRe$ generates the semigroup $P[\psi]$ on $\LRe$. Therefore, in this case, $\Lambda_\psi(\C^\infty_c(\bb R))$ is a core for $A_\ttt{2}[\psi]$. In general, for any $\psi\in\NN_b(\bb R)$, we have
\begin{align}\hlabel{eq:A_2=A_PDO}
	A_\ttt{2}[\psi]=A_{\ttt{PDO}}[\psi] \ \ \text{on} \ \ \Lambda_\psi(\cc{D}(\bb R))
\end{align}
and $\Lambda_\psi(\cc{D}(\bb R))$ is dense in $\LRe$.
\item \hlabel{main_thm_1:it:1d} For any $\psi\in\NN_b(\bb R)$ and  $t\ge 0$, we have
\begin{align}\hlabel{eq:bdd_intertwining}
P_t[\psi]\Lambda_\psi=\Lambda_\psi Q_t \ \mbox{ on } \ \D(\Lambda_\psi).
\end{align}
If $\Hpsi$ is the shifted Fourier multiplier operator with $m_{\psi}(z):=m_{\Hpsi}(z)=\frac{W_{\phi_+}(-{\rm{i}}z)}{W_{\phi_-}(1+{\rm{i}}z)}$, then $(\Hpsi,\D(\Hpsi))\in\mathscr{M}_e$, and, for all $t\ge 0$, we have
\begin{align}\hlabel{eq:spect_decomp}
	P_t[\psi] =\Hpsi e_t\Hpsinv \ \mbox{ on  } \ \hptg{C37}{\D(\Hpsinv)}
\end{align}
where  $(e_t)_{t\ge 0}$ denotes the multiplication semigroup on $\bmrm{L}(\bb{R},e)$, i.e.~ $\hptg{C36}{e_t f(x)=e^{-te^{-x}}f(x)}$. Note that $\Hpsinv=\Hpsid$ as they correspond to the same multiplier.
\end{enumerate}
\end{thm}
This theorem is proved in Section \ref{ss:pf_intertwining_pdo}. \\

\noindent
The item~\eqref{main_thm_1:it:1bc} in the above theorem provides us with a core for $A_\ttt{2}[\psi]$ when $\psi\in\NN_+(\bb R)$. The general case is more involved as $\C^\infty_c(\bb R)$ may not be included in the domain of $\Lambda_\psi$ when $\psi\in\NN_-(\bb R)$, as $m_{\Lambda}(\cdot+{\rm{i}}/2)$ can grow exponentially. Although the set $\Lambda_\psi(\cD(\bb R))$ mentioned in item~\eqref{main_thm_1:it:1bc} is dense in $\LRe$, it is not a core for $A_\ttt{2}[\psi]$ in general. To deal with this issue, we find a core for the generator of the semigroup $(e_t)_{t\ge 0}$, which consists of smooth functions with Fourier transforms decaying exponentially fast, and also invariant under the semigroup $(e_t)_{t\ge 0}$. Then, by imposing a mild condition on the ratio $\frac{\Wphip}{\Wphim}$, we show that those smooth functions are in the domain of $\Hpsi$ defined in item~\eqref{main_thm_1:it:1d}, and their image under $\Hpsi$ forms a core for $A_\ttt{2}[\psi]$. Additionally, we obtain the integro-differential representation of $A_\ttt{2}[\psi]$. These results are formally stated in the next proposition.

\begin{prop}\hlabel{prop:core_ido} Let $\psi\in\mathbf{N}_b(\bb{R})$ and denote by $A_\ttt{2}[\psi]$ the $\LRe$-generator of $P[\psi]$.
\item \hlabel{it:main_2}
\begin{enumerate}[(1)]
\item \hlabel{it:2d_new} If $\psi(\xi)=-\i\ttt{d}\xi$ for some $\ttt{d}>0$, then $\C^\infty_c(\bb{R})$ is a core for $A_\ttt{2}[\psi]$. Otherwise, let us assume that
\begin{align}\hlabel{eq:mpsi_bound}
	\eta_\psi:=\sup\left\{\eta\in\bb R; \: \xi\mapsto|\xi|^{\eta+\frac{1}{2}}\frac{W_{\phi_+}\left(\frac{1}{2}-\i\xi\right)}{W_{\phi_-}\left(\frac{1}{2}+\i\xi\right)} e^{-\frac{\pi}{2}|\xi|}\in\bmrm{L}(\bb R)\right\}\in (0,\infty].
\end{align}
Writing, for $\epsilon>0$, $\cc{E}(\eps)=\Span\{\heb; \beta>0\}$ where \hptg{C61}{$\heb(x)=e^{-(\frac{1}{2}+\eps)x}e^{-\beta e^{-x}}, x \in \bb{R}$},  the following holds.
\begin{enumerate}
\item \hlabel{it:2c1} For all $\eps>0$, $\cc{E}(\eps)$ is a dense subset of $\LRe$.
\item \hlabel{it:2c2} For all $0<\eps<\eta_\psi$, $\cc{E}(\eps)\subset\D(\Hpsi)$ and $\cc{E}_\psi(\eps):=H_\psi(\cc{E}(\eps))$ is a core for $A_\ttt{2}[\psi]$.
\end{enumerate}
\item\hlabel{it:2d} Let $\psi\in\mathbf{N}_b(\bb R)$ be such that, for all $n\in\bb N$,
\begin{align}\hlabel{eq:mpsi_bound_2}
	\xi\mapsto |\xi|^n\frac{\Wphip}{\Wphim} e^{-\frac{\pi}{2}|\xi|}\in\bmrm[1]{L}(\bb R).
\end{align}

Moreover, if $\psi$ satisfies one of the following conditions
\begin{enumerate}[(i)]
	\item\hlabel{conditions_i} $\phi_+(0)>0$
	\item\hlabel{conditions_iii} $\int_{|y|>1} |y|\mu(dy)<\infty$
\end{enumerate}
then $\cc{E}_\psi=\!\!\bigcup\limits_{0<\eps<\infty}\!\!\!\!\cc{E}_\psi(\eps)$ is core for  $A_\ttt{2}[\psi]$. Moreover, on $\cc{E}_\psi$,  $A_\ttt{2}[\psi]$ admits the representation as the following integro-differential operator
\begin{eqnarray}\hlabel{eq:integro_diff}
	A_\ttt{2}[\psi]f(x)&=&e^{-x}\left(\sigma^2 f''(x)+\ttt{b}f'(x)-\psi(0)f(x)\right. \nonumber\\
	&+ & \left. \int_\bb{R}(f(x+y)-f(x)-y\bbm{1}_{\{|y|\le 1\}}f'(x))\mu(dy) \right)
\end{eqnarray}
where $(\psi(0),\ttt{b},\sigma^2,\mu)$ is the quadruplet  determining $\psi$. 
\end{enumerate}
\end{prop}
This proposition is proved in Subsection~\ref{ss:prop:core_ido}.
\begin{rem}
	The condition \eqref{eq:mpsi_bound} is needed only to show that $\cc{E}_\psi(\eps)$ is in the domain of the $\LRe$-generator of $P[\psi]$. It is always satisfied when $\psi\in\NN_+(\bb R)$. The inclusion $\cc{E}_\psi(\eps)\subset\D(H_\psi)$ and the density of $H_\psi(\cc{E}(\eps))$ in $\LRe$ is true whenever $\eps>0$ is small enough, depending on $\psi$, see \eqref{eq:l2_domain} and \eqref{eq:delta_psi}.
\end{rem}
\begin{rem}
	We note that the condition \eqref{eq:mpsi_bound_2} is always satisfied when $\psi\in\NN_+(\bb R)$. More generally, \eqref{eq:mpsi_bound_2} holds whenever $\psi(\xi)\neq -\i\ttt{d}\xi$ and either $\ttt{d}_-=0$, or $\ttt{d}_+>0$, or $\int_{\bb R}\mu(dy)=\infty$, see \cite[Theorem~2.3(1)]{patie2018}, where $\ttt{d}_\pm$ is defined in Proposition~\ref{prop:ratio_criterion}, and $\mu$ is defined in \eqref{eq:defLKe}. However, when $\psi\in\NN_+(\bb R)$, we will consider $\Lambda_\psi(\C^\infty_c(\bb R))$ as core for $A_\ttt{2}[\psi]$, since we have a PDO representation of $A_\ttt{2}[\psi]$ on this set, see \eqref{eq:intertwining_pdo_0}, and it does not require the conditions \eqref{conditions_i}-\eqref{conditions_iii}.
\end{rem}


\subsection{Proof of Theorem~\ref{thm:intertwining_pdo}}\hlabel{ss:pf_intertwining_pdo} We begin by proving the following lemma. 
\begin{lem}\hlabel{lem:dissipativity}
	For any $\psi\in\mathbf{N}(\bb{R})$, the operator $A_\ttt{PDO}[\psi]$ is dissipative on the domain $\D_\psi=\{f\in\SS(\bb{R})\cap\bmrm{L}(\bb{R},e);\ A_\ttt{PDO}[\psi]f\in\bmrm{L}(\bb{R},e)\}$.
\end{lem}
\begin{proof}
For any  $f\in\D_\psi \subset \SS(\bb{R})$, we have
\begin{align} \label{eq:PDO_L}
	A_\ttt{PDO}[\psi]f(x)=-\frac{e^{-x}}{\sqrt{2\pi}}\int_\bb{R} e^{{\rm{i}}\xi  x}\psi(\xi)\mathcal{F}_{f}(\xi)\,d\xi=e^{-x}L[\psi]f(x)
\end{align}
where $L[\psi]$ is the $\bmrm{L}(\bb{R})$-generator as well as the Feller-generator  of the L\'evy process with L\'evy-Khintchine exponent $\psi$, see \cite[Section 3.4.1 and Theorem 3.3.3]{Applebaum}. It is known that an operator $(L,\D(L))$ on a Hilbert space $H$ is dissipative if and only if $\Re\langle Lf,f\rangle\le 0$ for all $f\in\D(L)$. To check this condition for $A_\ttt{PDO}[\psi]$ on $\D_\psi\subset\bmrm{L}(\bb{R},e)$, we observe that, for all $f\in\D(L)$,
\begin{align}
	&\Re \langle A_\ttt{PDO}[\psi]f, f\rangle_{\bmrm{L}(\bb{R},e)}=\Re\langle L[\psi]f, f\rangle_{\bmrm{L}(\bb{R})}\le 0
\end{align}
as $L[\psi]$ is the $\bmrm{L}(\bb{R})$-generator of the L\'evy process, hence dissipative. Thus, $(A_\ttt{PDO}[\psi],\D_\psi)$ is dissipative.
\end{proof}
\begin{rem} \label{rem:PDO_Dynk}
We point out that although the identity \eqref{eq:PDO_L} relates, on the dense subset $\D_\psi$, $A_\ttt{PDO}$ with the $\bmrm{L}(\bb{R},e)$- and  Feller generator of a L\'evy process, we can not yet conclude that  $A_\ttt{PDO}$ coincides on this set with the $\bmrm{L}(\bb{R}, e)$-generator of the log-Lamperti semigroup, as a representation of such generator is not available. Indeed, Lamperti \cite{Lamperti1972} provides only the representation of the generator as a Dynkin operator due to the delicate application  of the Volskonskii's formula in this context. However, under the condition \eqref{eq:mpsi_bound_2} on the Wiener-Hopf factors of $\psi$, we shall show that the Dynkin representation coincides with the Feller (and also the $\bmrm{L}(\bb{R}, e)$) one. To deal with the remaining cases, we will prove and use the fact  that our class of PDO's as well as the log-Lamperti class is stable by taking the adjoint in $\bmrm{L}(\bb{R}, e)$.
\end{rem}
\subsubsection{Proof of~Theorem~\ref{thm:intertwining_pdo}\eqref{main_thm_1:it:0a}.} To prove this result, we change back to $\bb R_+$ via the homeomorphism $x\mapsto e^x$, which will transform the operator $(A_\ttt{PDO}[\psi_0],\C^\infty_c(\bb R))$ to
\begin{align*}
	\widetilde{A} f(r)=r f''(r)+f'(r), \ f\in\C^\infty_c(\bb R_+).
\end{align*}
It is enough to show that the operator above is essentially self-adjoint with respect to $\C^\infty_c(\bb R)$. Note that writing $F(x,y)=f(x^2+y^2)$, one has $Af(x^2+y^2)=\frac{1}{2}\Delta F(x,y)$, where $\Delta$ is the Laplacian on $\bb R^2$. Therefore, $(\widetilde{A},\C^\infty_c(\bb R_+))$ coincides with the operator $(\Delta,\mathbf{C}^\infty_c(\bb R^2)\cap\cc{S})$, where $\cc{S}$ is the set of all radially symmetric functions on $\bb R^2$. Since the latter is essentially self-adjoint in $\bmrm{L}(\bb R^2)\cap\cc{S}$, we conclude that $(\widetilde{A},\C^\infty_c(\bb R_+))$ is essentially self-adjoint in $\bmrm{L}(\bb R_+)$.
Thus, $(A_\ttt{PDO}[\psi_0],\C^\infty_c(\bb R))$ is closable and the closure $(\overline{A}_\ttt{PDO}[\psi_0],\D(\overline{A}_\ttt{PDO}[\psi_0]))$ is also self-adjoint. Using Lemma~\ref{lem:dissipativity}, we get that ${A}_\ttt{PDO}[\psi_0]$ is dissipative, so is its closure. By the  theory of self-adjoint dissipative operators, we infer that $\overline{A}_\ttt{PDO}[\psi_0]$ generates a $\cc{C}_0$-contraction semigroup on $\bmrm{L}(\bb{R},e)$. On the other hand, from Lemma~\ref{lem:feller=L2} and Lemma~\ref{lem:Dynkin}, we observe that
\begin{align*}
A_\ttt{2}[\psi_0]=A_\ttt{PDO}[\psi_0]  \textrm{ on  } \C^\infty_c(\bb{R})
\end{align*}
where $A_\ttt{2}[\psi_0]$ is the $\LRe$-generator of the log-squared Bessel semigroup. As $\C^\infty_c(\bb{R})$ is dense in $\bmrm{L}(\bb{R},e)$, by uniqueness, we conclude that $(\overline{A}_\ttt{PDO}[\psi_0],\D(\overline{A}_\ttt{PDO}[\psi_0]))$ generates the log-squared Bessel semigroup $(Q_t)_{t\ge 0}$ on $\bmrm{L}(\bb{R},e)$.

\subsubsection{Proof of~Theorem~\ref{thm:intertwining_pdo}\eqref{main_thm_1:it:1a}.} For any $\phi_+,\phi_-\in\B$, we claim that $\frac{\Wphip}{\Wphim}\Gamma(\frac{1}{2}+\i\xi)$ is always bounded with respect to $\xi$. Indeed, we first note that $\xi\mapsto\Wphip$ is bounded since it is the Mellin transform of a random variable, see \cite[Theorem~6.1(3)]{patiesavov}, and the ratio $\frac{\Gamma(\frac{1}{2}+\i\xi)}{\Wphim}$ is also the Mellin transform of the exponential functional associated to the subordinator with Laplace exponent $\phi_-$, see \cite[Theorem~2.2]{hirsch_yor}. Now, recalling  that $|\Gamma(\frac{1}{2}+\i\xi)|\sim e^{-\frac{\pi}{2}|\xi|}$ as $|\xi|\to\infty$, the above argument entails that for all $\xi\in\bb R$,
\begin{align*}
	\left|\frac{\Wphip}{\Wphim}\right|\le Ce^{\frac{\pi}{2}|\xi|}
\end{align*}
which shows that $\Lambda_\psi\in\mathscr{M}_e$. In particular, when $\psi\in\NN_+(\bb R)$, from \eqref{N_+} and the definition of $m_{\Lambda_\psi}$, the mapping $\xi\mapsto m_{\Lambda_\psi}(\xi+{\rm{i}}/2)$ is bounded, since $|\Gamma({1 \over 2}+\i\xi)|=|\Gamma({1\over 2}-\i\xi)|$ for all $\xi\in\bb R$.
 Therefore, $\Lambda_\psi\in\mathscr{B}(\LRe)$ if $\psi\in\NN_+(\bb R)$.

\subsubsection{Proof of~Theorem~\ref{thm:intertwining_pdo}\eqref{main_thm_1:it:1b}.} 
We first prove the second statement of this item.
\begin{prop}\hlabel{prop:intertwining_test}
	For any $\psi\in\NN_+(\bb R)$ we have
	\begin{align}\hlabel{eq:intertwining_pdo_0}
		A_\ttt{PDO}[\psi]\Lambda_\psi =\Lambda_\psi A_\ttt{PDO}[\psi_0]\ \mbox{ on } \ \C^\infty_c(\bb R).
	\end{align}
\end{prop}
\begin{proof}
 Since the proof is quite technical, we start by outlining the main ideas of its proof. If $m_{\Lambda_\psi}$ extends continuously on $\bb{S}_{[0,1]}$ and the conditions of Proposition~\ref{multiplier_result} are met for $f\in\C^\infty_c(\bb R)$ with $\gamma=1$, then we get
\begin{align*}
	\fou_{\! A_\ttt{PDO}[\psi]\Lambda_\psi f}(\xi+\i)=\fou_{L[\psi]\Lambda_\psi f}(\xi)=m_{\Lambda_\psi}(\xi)\psi(\xi)\fou_f(\xi)
\end{align*}
where $L[\psi]$ is the generator of the L\'evy process associated to $\psi$. On the other hand, if one takes $\beta=0$ in Proposition~\ref{cor:multiplier}, it is expected that $m_{\Lambda_\psi}$ should satisfy the same functional equation in \eqref{eq:fneq} with $\beta=0$, given that $m_{\Lambda_\psi}(\xi+\i)$ is well defined for all $\xi\in\bb R$. As a result, a straightforward application of Proposition~\ref{multiplier_result} would imply that
\begin{align*}
	\fou_{\Lambda_\psi A_\ttt{PDO}[\psi_0] f}(\xi+\i)=m_{\Lambda_\psi}(\xi+\i)\fou_{\! A_\ttt{PDO}[\psi_0] f}(\xi+\i)=m_{\Lambda_\psi}(\xi+\i)\psi_0(\xi)\fou_{f}(\xi)
\end{align*}
which would establish that $\fou_{\! A_\ttt{PDO}[\psi]\Lambda_\psi f}=\fou_{\Lambda_\psi A_\ttt{PDO}[\psi_0]f}$. However, the continuous extension of $m_{\Lambda_\psi}$ on the line $\bb R+\i$ is not always guaranteed and this is why we perturb the operator $\Lambda_\psi$ by a parameter $\beta>0$, and resort to Proposition~\ref{cor:multiplier} to prove \eqref{eq:intertwining_pdo_0}. To this end, we recall that $\psi_\beta(\xi)=\xi^2-{\rm{i}}\beta\xi$ for $\beta>0$ and $\xi \in \bb{R}$. Also, let us define
\begin{align*}
m_{\Lambda^{\!(\beta)}}(z)=\frac{\Gamma(1+{\rm{i}}z   )}{\Gamma(1+\beta+{\rm{i}}z)}, \ z\in\bb{S}_{(-\infty,1)}.
\end{align*}
Clearly, $m_{\Lambda^{\!(\beta)}}$ is analytic and zero-free in $\bb{S}_{[0,1)}$. Let  $\Lambda^{\!(\beta)}$ be the shifted Fourier multiplier associated to $m_{\Lambda^{\!(\beta)}}$. Since $m_{\Lambda^{\!(\beta)}}$ is bounded on the line $\bb{R}+{\rm{i}}/2$, $\Lambda^{\!(\beta)}\in\mathscr{M}\cap\mathscr{B}(\bmrm{L}(\bb{R},e))$.  We note that for any $f\in \C^\infty_c(\bb{R})$, $\mathcal{F}_{f}$ is an entire function. Moreover, $m_{\LL^{\!(\beta)}}\in\A_{(-\infty,1)}$ and by the Stirling asymptotic formula of Gamma functions, $m_{\LL^{\!(\beta)}}$ is uniformly bounded on the \hpt{R24}{strip $\bb{S}_{(-\infty, a/2]}$ for any $a<2$}. \hptg{C39}{Also, $\fou_f(\cdot+\frac{\i a}{2})\in\bmrm{L}(\bb R)$ for all $a<2$, and therefore by Proposition~\ref{multiplier_result}, it follows that $\Lambda^{\!(\beta)}f\in\bmrm{L}(\bb{R},e^{ax}dx)$ for any $a<2$.}

\begin{lem}\label{lem:beta}
For any $\beta>0$ we have
\begin{align*}
A_{\ttt{PDO}}[\psi_\beta]\Lambda^{\!(\beta)} =\Lambda^{\!(\beta)} A_\ttt{PDO}[\psi_0] \ \mbox{ on } \  \C^\infty_c(\bb{R})
\end{align*}
and $\Lambda^{\!(\beta)}(\C^\infty_c(\bb{R}))$ is dense in $\bmrm{L}(\bb{R},e)$.
\end{lem}
\begin{proof}
By the definition of the  operator $\Lambda^{\!(\beta)}$, the fact that $A_\ttt{PDO}[\psi_0](\C^\infty_c(\bb R))\subset\C^\infty_c(\bb R)$ and Proposition~\ref{multiplier_result},
we have, for any $f\in \C^\infty_c(\bb{R})$ and writing $d\tilde{\xi}_x=-e^{{\rm{i}}\xi   x}\frac{d\xi}{\sqrt{2\pi}}$,
\begin{eqnarray}
\Lambda^{\!(\beta)} A_\ttt{PDO}[\psi_0]f(x)&=&-\int_{\bb{R}}e^{{\rm{i}}\xi   x}m_{\Lambda^{(\beta)}}(\xi)(\xi-{\rm{i}})   ^2\mathcal{F}_{f}(\xi-{\rm{i}})   \,\frac{d\xi}{\sqrt{2\pi}} \nonumber \\
&=&\int_\bb{R} \frac{\Gamma(1+{\rm{i}}\xi  )(\xi-{\rm{i}})^2}{\Gamma(1+\beta+{\rm{i}}\xi  )}\mathcal{F}_{f}(\xi-{\rm{i}})   \,d\tilde{\xi}_x \hlabel{eq:step_2}=\int_\bb{R} \frac{\rm{i}\Gamma(2+{\rm{i}}\xi )(\xi-{\rm{i}})}{\Gamma(1+\beta+{\rm{i}}\xi  )}   \mathcal{F}_{f}(\xi-{\rm{i}})   \,d\tilde{\xi}_x \\
&=&e^{-x}\int_\bb{R} \frac{{\rm{i}}\Gamma(1+{\rm{i}}\xi  )}{\Gamma(\beta+{\rm{i}}\xi  )}\xi\mathcal{F}_{f}(\xi)\,d\tilde{\xi}_x\hlabel{eq:contour_1}
 =e^{-x}\int_\bb{R}\frac{\Gamma(1+{\rm{i}}\xi  )}{\Gamma(1+\beta+{\rm{i}}\xi  )}\psi_\beta(\xi)\mathcal{F}_{f}(\xi)\,d\tilde{\xi}_x  \\ &=&A_\ttt{PDO}[\psi_\beta]\Lambda^{\!(\beta)}f(x)\hlabel{eq:intertwining_pdo}
\end{eqnarray}
where \eqref{eq:contour_1} follows  by applying Lemma~\ref{prop:contour}. \hptg{C40}{Now, since the adjoint operator $\widehat{\Lambda}^{\!(\beta)}$ is also a shifted Fourier multiplier operator with $m_{\widehat{\Lambda}^{\!(\beta)}}(\xi+{\rm{i}}/2)=m_{\Lambda^{\!(\beta)}}(-\xi+{\rm{i}}/2)\neq 0$ for a.e.~$\xi$, it is bounded and injective, which implies that the image of any dense set under $\Lambda^{(\beta)}$ is dense in $\LRe$. This proves the lemma.}
\end{proof}
\noindent
Now, let $g\in\Lambda^{\!(\beta)}(\C^\infty_c(\bb{R}))$, that is  $g=\Lambda^{\!(\beta)}f$, $f\in \C^\infty_c(\bb{R})$, and,
\begin{align}\hlabel{eq:W^beta_psi}
m_{\Lambda^{\!(\beta)}_\psi}(z)=\WBpsi(z)=\frac{W_{\phi_+}(-{\rm{i}}z   )\Gamma(1+{\rm{i}}z+\beta)}{\Gamma(-{\rm{i}}z   )W_{\phi_-}(1+{\rm{i}}z   )} \ \ \mbox{for} \ \  z\in\bb{S}_{[0,1)}.
\end{align}
From Proposition~\ref{cor:multiplier}\eqref{it:W_1}, it follows that $m_{\Lambda^{\!(\beta)}_\psi}\in\A_{[0,1]}$. Let $\Lambda^{\!(\beta)}_\psi$ be Fourier operator associated to $m_{\Lambda^{\!(\beta)}_\psi}$. Then, $\Lambda^{\!(\beta)}_\psi\in\mathscr{M}$. Moreover, from our assumption that $\psi\in\mathbf{N}_+(\bb{R})$ and the  Stirling asymptotic of the gamma function, it follows that $\left|m_{\Lambda^{\!(\beta)}_\psi}\left(\xi+{\rm{i}}/2\right)\right|={\rm{O}}(|\xi|^\beta)$. From Proposition~\ref{cor:multiplier}\eqref{it:kap}, one gets that
$\xi\mapsto m_{\Lambda^{\!(\beta)}_\psi}(\xi)$ and  $ \xi\mapsto m_{\Lambda^{\!(\beta)}_\psi}(\xi+{\rm{i}})   $ have at most polynomial growth as $|\xi|\to\infty$. On the other hand, from \eqref{eq:step_2} and \eqref{eq:intertwining_pdo}, one can show that
\begin{align*}
A_\ttt{PDO}[\psi_\beta]g\in\bmrm{L}(\bb{R},e)  \mbox{ and }  \fouho_{A_\ttt{PDO}[\psi_\beta]g}(\xi)={\rm{O}}(|\xi|^{-n}) \ \ \mbox{for all $n\in\bb{N}$}.
\end{align*}
Hence, the mapping $\xi\mapsto m_{\Lambda^{\!(\beta)}_\psi}\left(\xi+{\rm{i}}/2\right)\fouho_{A_\ttt{PDO}[\psi_\beta]g}(\xi)\in\bmrm{L}(\bb{R})\cap\bmrm[1]{L}(\bb{R})$,
which ensures that $A_\ttt{PDO}[\psi_\beta]g\in\D(\Lambda^{\!(\beta)}_\psi)$.
 Therefore, we have,
\begin{align}
\Lambda^{\!(\beta)}_\psi A_{\ttt{PDO}}[\psi_\beta]g(x)&=\frac{e^{-\frac{x}{2}}}{\sqrt{2\pi}}\int_{\bb{R}} e^{{\rm{i}}\xi   x}\fouho _{\Lambda^{\!(\beta)}_\psi A_\ttt{PDO}[\psi_\beta]g}(\xi)\,d\xi \hlabel{line1}\\
&=-\frac{1}{\sqrt{2\pi}}\int_{\bb{R}+\frac{{\rm{i}}}{2}} e^{{\rm{i}}z   x}m_{\Lambda^{\!(\beta)}_\psi}(z)\psi_\beta(z-{\rm{i}})   \fou_{g}(z-{\rm{i}})   \,dz \hlabel{line1'} \\
&=-\frac{1}{\sqrt{2\pi}}\int_{\bb{R}} e^{{\rm{i}}\xi   x}m_{\Lambda^{\!(\beta)}_\psi}(\xi)\psi_\beta(\xi-{\rm{i}})   \fou_{g}(\xi-{\rm{i}})   \,d\xi \hlabel{line2}\\
&=-\frac{e^{-x}}{\sqrt{2\pi}}\int_{\bb{R}}e^{{\rm{i}}\xi   x}m_{\Lambda^{\!(\beta)}_\psi}(\xi+{\rm{i}})   \psi_\beta(\xi)\fou_{g}(\xi)d\xi  \nonumber\\
&=-\frac{e^{-x}}{\sqrt{2\pi}}\int_{\bb{R}}e^{{\rm{i}}\xi   x}m_{\Lambda^{\!(\beta)}_\psi}(\xi)\psi(\xi)\mathcal{F}_{g}(\xi)\,d\xi \ \ \nonumber \\
&= A_\ttt{PDO}[\psi]\Lambda^{\!(\beta)}_\psi g(x). \hlabel{line5}
\end{align}
\hptg{C42}{In the computation above, the second identity follows by a change of variable which is valid since $\psi_\beta$ and $\fou_{g}$ are entire functions, and $\fou_{A_\ttt{PDO}[\psi_\beta]g}(\xi)=\psi_\beta(\xi-\i)\fou_g(\xi-\i)$ for all $\xi\in\bb R$.} \eqref{line2} follows from \eqref{line1'} and Lemma~\ref{prop:contour}, since $\psi_\beta\fou_{g}\in\A_{(-\infty,1)}$ and it decays faster than any polynomial as $|\xi|\to\infty$, and the identity before \eqref{line5} follows from Proposition~\ref{cor:multiplier}\eqref{eq:fneq} and Proposition~\ref{multiplier_result}. Now, recalling the definition of the shifted Fourier multiplier operator $\Lambda_\psi$, we note that $m_{\Lambda^{\!(\beta)}_\psi} m_{\Lambda^{\!(\beta)}}=m_{\Lambda_\psi}$ on $\bb{S}_{(0,1)}$, which implies that $\Lambda_\psi=\Lambda^{\!(\beta)}_\psi\Lambda^{\!(\beta)}$. Then, plugging in $g=\Lambda^{\!(\beta)}f$ in the above computation along with Lemma~\ref{lem:beta}, \eqref{eq:intertwining_pdo_0} follows.
\end{proof}
Now, coming back to the proof of item~\eqref{main_thm_1:it:1b}, we first show the density of $\cc{D}(\bb R)$ in $\LRe$. Indeed, we have the following inclusion
\begin{align*}
	\Span\{x\mapsto e^{-(x-a)^2}; \: a\in\bb R\}\subset \cc{D}(\bb R)
\end{align*}
and the former set is dense in $\LRe$, thanks to Corollary~\ref{corr:wiener}. Next, to prove \eqref{eq:intertwining_pdo_0_D}, we can mimic the same technique used in the proof of Proposition~\ref{prop:intertwining_test} after justifying the following fact.
\begin{lem}
	For any $\psi\in\NN_b(\bb R)$ and $f\in\cc{D}(\bb R)$ we have, for all $u>0$,
	\begin{align}\hlabel{eq:estimate_D_psi}
		\sup_{b\in [0,1]}\left|m_{\Lambda^{\!(\beta)}_\psi}(\xi+\i b)\cc{F}_f(\xi\pm\i b)\right|={\rm O}(|\xi|^{-u}).
	\end{align}
\end{lem}
\begin{proof}
	From \eqref{eq:W^beta_psi} and invoking the fact that, for any $b>0$,
	\begin{align*}
		\left|\frac{W_{\phi_+}(b+\i\xi)}{W_{\phi_-}(b+\i\xi)}\Gamma(b+\i\xi)\right|={\rm O}(1)
	\end{align*}
we deduce that for $b\in\{0,1\}$, $|m_{\Lambda^{\!(\beta)}_\psi}(\xi+\i b)|={\rm O}(|\xi|^l e^{\pi |\xi|/2})$ for some $l>0$. Since for any $f\in\cc{D}(\bb R)$, $\cc{F}_f$ is entire and $\sup_{b\in\{0,1\}}|\cc{F}_f(\xi\pm\i b)|={\rm O}(e^{-(\pi/2+\eps)|\xi|})$ for some $\eps>0$, the statement of the lemma holds for $b=0,1$. \hpt{R25}{Finally, the estimate in \eqref{eq:estimate_D_psi} follows by a straightforward application of Lemma~\ref{lem:W_beta_bound} and the Phragm\'en-Lindel\"of principle as in the proof of Proposition~\ref{cor:multiplier}\eqref{it:kap}.}
\end{proof}
Because of the above lemma, the statement of Lemma~\ref{lem:beta} and the computational steps between \eqref{line1} and \eqref{line5} still go through when $\C^\infty_c(\bb R)$ is replaced by $\cc{D}(\bb R)$. This completes the proof of item~\eqref{main_thm_1:it:1b}.


\subsubsection{Proof of Theorem \ref{thm:intertwining_pdo}\eqref{main_thm_1:it:1bc}}
 We  show that the closure in $\LRe$ of $(A_\ttt{PDO}[\psi],\Lambda_\psi(\C^\infty_c(\bb R)))$ generates the semigroup $P[\psi]$. Using the fact that $A_\ttt{PDO}[\psi_0]$ is dissipative and $\C^\infty_c(\bb{R})$ is its core, thanks to Lemma~\ref{lem:dissipativity} and item~\ref{main_thm_1:it:0a}, we get that for any $\alpha>0$, $(\alpha I-A_\ttt{PDO}[\psi_0])(\C^\infty_c(\bb{R}))$ is dense in $\bmrm{L}(\bb{R},e)$. \hptg{C45}{Moreover, since  $\xi \mapsto m_{\Lambda_\psi}(\xi+\frac{{\rm{i}}}{2})$ is bounded, and  $\overline{m}_{\Lambda_\psi}(\xi+\frac{{\rm{i}}}{2})$ is non-zero on $\bb{R}$, it follows that both $\Lambda_\psi$ and $\widehat{\Lambda}_\psi$ are bounded and injective on $\bmrm{L}(\bb{R},e)$, and from this we infer that $\Lambda_\psi (\alpha I-A_\ttt{PDO}[\psi_0])(\C^\infty_c(\bb{R}))$ is dense in $\bmrm{L}(\bb{R},e)$.} Invoking Lemma~\ref{lem:generator_intertwining}, it follows that $(A_\ttt{PDO}[\psi],\Lambda_\psi(\C^\infty_c(\bb R)))$ generates a $\cc{C}_0$-contraction semigroup on $\LRe$. Next, to show that this semigroup indeed coincides with $P[\psi]$, let $A_\ttt{2}[\psi]$ denote the $\LRe$-generator of $P[\psi]$. \hptg{C47}{We aim to show that the two operators $A_2[\psi]$ and $A_\ttt{PDO}[\psi]$ are identical when restricted to the set of functions $\Lambda_\psi(\C^\infty_c(\bb R))$. For that, we first note that $\Lambda_\psi(\C^\infty_c(\bb R))\subset\C^\infty_0(\bb R)$ for any $\psi\in\mathbf{N}_+(\bb R)$. This follows by observing that for any $\psi\in\mathbf{N}_+(\bb R)$, $f\in\C^\infty_c(\bb R)$, and $u>0$,
\begin{align*}
	\lim_{|\xi|\to\infty}|\xi|^u \sup_{b\in [0,\frac{1}{2}]}\left|m_{\Lambda_\psi}(\xi+\i b)\fou_{f}(\xi+\i b)\right|=0.
\end{align*}
As $m_{\Lambda_\psi}$ is analytic in the strip $\bb{S}_{[0,1)}$, Proposition~\ref{multiplier_result} yields that $\fou_{\Lambda_\psi f}=m_{\Lambda_\psi}\fou_f$ and hence, $\Lambda_\psi f\in\C^\infty_0(\bb R)$ by Riemann-Lebesgue lemma. If $A_\ttt{D}[\psi]$ denotes the  Dynkin characteristic operator for the log-self-similar Feller process, recalling Lamperti's result, we get
\begin{align*}
A_\ttt{D}[\psi]\Lambda_\psi  =A_\ttt{PDO}[\psi]\Lambda_\psi   \ \mbox{ on }  \C^\infty_c(\bb{R}).
\end{align*}}
From \eqref{line1}, \eqref{line2}, \eqref{line5}, we observe that for any $ f \in \C^\infty_c(\bb{R})$, $A_\ttt{PDO}[\psi]\Lambda_\psi f \in \C_0(\bb{R})\cap\bmrm{L}(\bb{R},e)$. Therefore, by Lemma~\ref{lem:Dynkin}, denoting the generator of the Feller semigroup by $A_\ttt{F}[\psi]$, we have $\Lambda_\psi f \in\D(A_\ttt{F}[\psi])$. Finally, using Lemma~\ref{lem:feller=L2}, we infer that for all $f\in \C^\infty_c(\bb{R})$, $\Lambda_\psi f \in\D(A_\ttt{2}[\psi])$ and
\begin{align*}
A_\ttt{2}[\psi]\Lambda_\psi f =A_\ttt{F}[\psi]\Lambda_\psi f =A_\ttt{D}[\psi]\Lambda_\psi f =A_\ttt{PDO}[\psi]\Lambda_\psi f .
\end{align*}
As $\Lambda_\psi(\C^\infty_c(\bb{R}))$ is a core for $A_\ttt{PDO}[\psi]$, the \hptg{C48}{closure of the operators $(A_{\ttt{PDO}}[\psi],\Lambda_\psi(\C^\infty_c(\bb{R})))$ and $(A_\ttt{2}[\psi],\Lambda_\psi(\C^\infty_c(\bb{R})))$ must generate the same semigroup $P[\psi]$}. To finish the proof of this item, it remains to show \eqref{eq:A_2=A_PDO} and the density of $\Lambda_\psi(\cD(\bb R))$. We note that the set $\cc{D}(\bb R)$ is closed under translation, that is, if $f\in\cc{D}(\bb R)$ then $\uptau_a f\in\cc{D}(\bb R)$ for all $a\in\bb R$. Since $\Lambda_\psi$ is a shifted Fourier multiplier, $\Lambda_\psi(\cD(\bb R))$ is also closed under translation. Therefore, the density of $\Lambda_\psi(\cD(\bb R))$ follows from Corollary~\ref{corr:wiener}. Now, from \eqref{eq:intertwining_pdo_0_D} a similar argument involving Lemma~\ref{lem:Dynkin} and Lemma~\ref{lem:feller=L2} as before will lead to the identity \eqref{eq:A_2=A_PDO}.

\subsubsection{Proof of~Theorem~\ref{thm:intertwining_pdo}\eqref{main_thm_1:it:1d}.} Following the proof of item~\eqref{main_thm_1:it:1bc}, we can replace $A_\ttt{PDO}$ by $A_\ttt{2}$ and therefore by \eqref{eq:intertwining_pdo_0}, for all $\alpha>0$ we have,
\begin{align*}
(\alpha I-A_\ttt{2}[\psi])\Lambda_\psi  =\Lambda_\psi(\alpha I-A_\ttt{2}[\psi_0]) \ \textrm{ on } \  \C^\infty_c(\bb{R}).
\end{align*}
Recalling that the resolvent operators $R_\alpha[\psi]=(\alpha I-A_\ttt{2}[\psi])^{-1}$ and $R_\alpha[\psi_0]=(\alpha I-A_\ttt{2}[\psi_0])^{-1}$, we get
\begin{align*}
&R_\alpha[\psi]\Lambda_\psi =\Lambda_\psi R_\alpha[\psi_0]  \ \mbox{ on } \ (\alpha I-A_\ttt{2}[\psi_0])(\C^\infty_c(\bb{R})).\nonumber
\end{align*}
As $\C^\infty_c(\bb{R})$ is a core of $A_\ttt{2}[\psi_0]$, $(\alpha I-A_\ttt{2}[\psi_0])(\C^\infty_c(\bb{R}))$ is dense in $\bmrm{L}(\bb{R},e)$. Using the boundedness of $\Lambda_\psi, R_\alpha[\psi], R_\alpha[\psi_0]$, we have
\begin{align*}
R_\alpha[\psi]\Lambda_\psi=\Lambda_\psi R_\alpha[\psi_0] \ \ \mbox{on \ $\bmrm{L}(\bb{R},e)$}.
\end{align*}
Recalling  that $R_\alpha[\psi]$ (resp.~$R_\alpha[\psi_0]$) is the Laplace transform of the semigroup $P[\psi]$ (resp.~$Q$), we have
\begin{equation*}
\begin{aligned}
\int_0^\infty e^{-\alpha t}P_t[\psi]\Lambda_\psi f dt=\Lambda_\psi \int_0^\infty e^{-\alpha t} Q_t f dt=\int_0^\infty e^{-\alpha t} \Lambda_\psi Q_t f dt
\end{aligned}
\end{equation*}
where the second equality follows from the boundedness of $\Lambda_\psi$. Since $t\mapsto P_t[\psi]\Lambda_\psi f$ and $t\mapsto \Lambda_\psi Q_t f$ are bounded functions, from uniqueness of Laplace transform, we conclude that, for almost every $t\geq 0$,
\begin{align}\hlabel{eq:gen_semigroup_intertwining}
P_t[\psi]\Lambda_\psi=\Lambda_\psi Q_t \ \ \mbox{on \ $\bmrm{L}(\bb{R},e)$.}
\end{align}
\hptg{C49}{By strong continuity of the semigroups, the above identity extends for all $t\ge 0$.} This proves \eqref{main_thm_1:it:1d}. If $\psi\in\mathbf{N}_-(\bb{R})$, then $\ov{\psi}\in\mathbf{N}_+(\bb{R})$ and taking adjoint in \eqref{eq:bdd_intertwining} and using the fact $\widehat{P}[\ov{\psi}]=P[\psi]$, we get
\begin{align*}
	&Q_t\widehat{\Lambda}_{\ov{\psi}}=\widehat{\Lambda}_{\ov{\psi}} \widehat{P}_t[\ov{\psi}]
	\iff  Q_t\widehat{\Lambda}_{\ov{\psi}}=\widehat{\Lambda}_{\ov{\psi}} P_t[\psi]. \nonumber
\end{align*}
\hptg{C51}{From the definition of $\Lambda_\psi$ and $\Lambda_{\ov{\psi}}$, we note that $m^{-1}_{\LL_\psi}=\overline{m}_{\LL_{\overline{\psi}}}$, which implies that $\widehat{\Lambda}_{\ov{\psi}}=\Lambda^{-1}_\psi$.} Hence, using Proposition~\ref{intertwining_equiv}, we get $P_t[\psi]\Lambda_\psi=\Lambda_\psi Q_t$ on $\D(\Lambda_\psi)$. To prove \eqref{eq:spect_decomp}, we first observe that the semigroup $(e_t)_{t\ge 0}$ corresponds to $\psi\equiv 1$. Thus,  from \eqref{eq:gen_semigroup_intertwining},  we have
\begin{align*}
	e_t H=HQ_t \ \ \mbox{on \ $\bmrm{L}(\bb{R},e)$}
\end{align*}
where  $m_H(z)=\frac{\Gamma(1+{\rm{i}}z   )}{\Gamma(-{\rm{i}}z   )}$ for all $z\in\bb{S}_{(0,1)}$. Noting that $m_H\left(\xi+{\rm{i}}/2\right)=\frac{\Gamma(1/2+{\rm{i}}\xi  )}{\Gamma(1/2-{\rm{i}}\xi  )}$ for all $\xi\in\bb R$, $H$ turns out to be a unitary operator as $\left|m_{H}\left(\xi+{\rm{i}}/2\right)\right|=1$ for all $\xi\in\bb R$. Hence, $Q$ is unitary similar to the semigroup $(e_t)_{t\ge 0}$. Let us define $\Hpsi=\Lambda_\psi H$. By Proposition~\ref{m-group}, $\Hpsi\in\mathscr{M}_e$ with $m_{\Hpsi}(z)=\frac{W_{\phi_+}(-{\rm{i}}z   )}{W_{\phi_-}(1+{\rm{i}}z   )}$ on $\bb{S}_{(0,1)}$ and $\D(\Hpsi)=\D(\Lambda_\psi)$. Therefore, for all $f\in\D(\Hpsi)$,
\begin{align}\hlabel{eq:pf_spectral_decomp}
	P_t[\psi]\Hpsi f=P_t[\psi]\Lambda_\psi Hf=\Lambda_\psi Q_t Hf=\Lambda_\psi H e_t f=\Hpsi e_t f.
\end{align}
Also, $m_{H_{\ov{\psi}}}=\frac{1}{\ov{m}_{\Hpsi}}$ on $\bb{R}+\frac{\i}{2}$. Thus, $\Hpsinv=\Hpsid$, which, from \eqref{eq:pf_spectral_decomp} yields $P_t[\psi]=\Hpsi e_t\Hpsid$ on $\D(\Hpsid)$.

\subsection{Proof of Proposition~\ref{prop:core_ido}}\hlabel{ss:prop:core_ido}
\subsubsection{Proof of Proposition~\ref{prop:core_ido}\eqref{it:2d_new}}\hlabel{ss:thm_prop_core_ido} When $\psi(\xi)=-\i\ttt{d}\xi$ for some $\ttt{d}>0$, the corresponding semigroup $P_t[\psi]$ is given by
\begin{align*}
	P_t[\psi] f(x)=f(\ln(e^x+\ttt{d}t)).
\end{align*}
From the above identity, it follows that for any $f\in\C^\infty_c(\bb R)$, $\lim_{t\downarrow 0}(P_t[\psi]f-f)/t$ exists in $\LRe$ and hence, $\C^\infty_c(\bb R)\subset\D(A_{\ttt{2}}[\psi])$. Now, trivially, $P_t[\psi](\C^\infty_c(\bb{R}))\subset\C^\infty_c(\bb{R})$, which shows that $\C^\infty_c(\bb{R})$ is a core for the generator of $P[\psi]$.

For the claim \eqref{it:2c1}, we first note that for any $\epsilon,\beta>0$, one has for all $\xi\in\bb{R}$,
\begin{align}\hlabel{eq:fourier_trans}
	\fou_{\heb}(\xi)=\frac{\beta^{-\frac{1}{2}-\eps-\i\xi}}{\sqrt{2\pi}}\Gamma\left(\frac{1}{2}+\eps+\i\xi\right).
\end{align}
Since $\fou_{\heb}$ is non-zero everywhere and $\heb\in\bmrm{L}(\bb{R},e)$, writing $\beta=e^a$ for $a\in\bb{R}$, the density of $\cc{E}(\eps)$ follows from Corollary~\ref{corr:wiener}. Next, for \eqref{it:2c2}, we show that there exists $\delta_\psi>0$ such that for all $\beta>0$,
\begin{align}\hlabel{eq:l2_domain}
	\xi\mapsto m_{H_\psi}\left(\xi+{\rm{i}}/2\right)\fouho_{\heb}\left(\xi\right)\in\bmrm{L}(\bb{R})
\end{align}
whenever $0<\eps<\delta_\psi$. Since $\psi(\xi)\neq -\i\ttt{d}\xi$ for some $\ttt{d}>0$, from \cite[Theorem~2.3(1)]{patie2018} we know that there exists $\delta_\psi>0$ such that, for any  $\delta_1<\delta_\psi<\delta_2$,
\begin{align}\hlabel{eq:delta_psi}
	\lim_{|\xi|\to\infty}|\xi|^{\delta_1}\left|\frac{W_{\phi_+}\left(\frac{1}{2}-\i\xi\right)}{W_{\phi_-}\left(\frac{1}{2}+\i\xi\right)}\Gamma\left(\frac{1}{2}+\i\xi\right)\!\right|=0, \  \lim_{|\xi|\to\infty}|\xi|^{\delta_2}\left|\frac{W_{\phi_+}\left(\frac{1}{2}-\i\xi\right)}{W_{\phi_-}\left(\frac{1}{2}+\i\xi\right)}\Gamma\left(\frac{1}{2}+\i\xi\right)\!\right|=\infty.
\end{align}
Since for all $\xi\in\bb{R}$, $\fouho_{\heb}(\xi)=\fou_{\heb}\left(\xi+{\rm{i}}/2\right)=\frac{\beta^{-\eps-\i\xi}}{\sqrt{2\pi}}\Gamma\left(\eps+\i\xi\right)$, by definition of $\tred{\Hpsi}$ and using the above estimate along with the Stirling approximation, we have for any $\delta_1<\delta_\psi<\delta_2$,
\begin{align*}
	\lim_{|\xi|\to\infty}|\xi|^{\delta_1+\frac{1}{2}-\eps}\left|m_{H_\psi}\left(\xi+{\rm{i}}/2\right)\fouho_{\heb}\left(\xi\right)\right|=0, \ \ \lim_{|\xi|\to\infty}|\xi|^{\delta_2+\frac{1}{2}-\eps}\left|m_{H_\psi}\left(\xi+{\rm{i}}/2\right)\fouho_{\heb}\left(\xi\right)\right|=\infty.
\end{align*}
Clearly, this implies \eqref{eq:l2_domain}  whenever $0<\eps<\delta_\psi$, which proves that $\cc{E}_\psi(\eps)\subset \D(H_\psi)$. To show the density of $\cc{E}_\psi(\eps)$ in $\LRe$, let $f\in\bmrm{L}(\bb{R},e)$ be such that $f\in\Hpsi(\cc{E}(\eps))^\perp$. Then, by the isometry of shifted Fourier transform, we have
\begin{align}\hlabel{eq:inner_prod}
	\fouho _{f}\perp \left\{m_{\Hpsi}\left(\cdot+\frac{{\rm{i}}}{2}\right)\fouho_{\heb}; \beta>0\right\}.
\end{align}
Recalling that $\fouho_{\heb}(\xi)=\frac{\beta^{-\eps-\i\xi}}{\sqrt{2\pi}}\Gamma(\eps+\i\xi)$ and choosing $\beta=e^a$, from \eqref{eq:inner_prod}, we get
\begin{align*}
	\int_{\bb{R}}e^{-a\eps-{\rm{i}} a\xi}m_{\Hpsi}\left(\xi+\frac{{\rm{i}}}{2}\right)\Gamma(\eps+\i\xi)d\xi=0 \ \ \mbox{ for all }  a\in\bb{R}.
\end{align*}
This implies that $\fouho_{f}(\xi)=0$ a.e., which means that $f=0$ a.e. This completes the proof of the density of $\cc{E}_\psi(\eps)$.


Now, it remains to show that $\Hpsi(\cc{E}(\eps))$ is a core for $A_\ttt{2}[\psi]$, and, to this end, we need the following two lemmas.
\begin{lem}\hlabel{lem:core_invariant}
The operator $(I_{e},\bmrm{L}(\bb{R}, e^{|x|})\tred{)}$, where $I_{e}f(x)=-e^{-x}f(x)$, is the $\bmrm{L}(\bb{R},e)$-generator of $(e_t)_{t\ge 0}$, and for any $\eps>0$, $\cc{E}(\eps)\subset\bmrm{L}(\bb{R}, e^{|x|})$. Moreover, for any $\eps>0$, $\cc{E}(\eps)$ is invariant under the semigroup $(e_t)_{t\ge 0}$ and hence
it is a core of its  generator.
\end{lem}
\begin{proof} If $f\in\D(I_e)\cap\LRe$, then $\lim_{t\to 0}(e_t f-f)/t$ must exist in $\LRe$. After observing that $\lim_{t\to 0}(1-e^{-te^{-x}})/t=e^{-x}$ for all $x\in\bb{R}$, by Fatou's lemma, we have
	\begin{align*}
		\int_{\bb{R}} |f(x)|^2 e^{-x} dx\le\liminf_{t\to 0}\int_{\bb R}\left(\frac{e_t f(x)-f(x)}{t}\right)^2 e^x dx=\|I_e f\|^2_{\LRe}<\infty.
	\end{align*}
This shows that $\D(I_e)\subseteq\bmrm{L}(\bb{R}, e^{|x|})$. On the other hand, if $f\in\bmrm{L}(\bb{R},e^{|x|})$, we first observe that
\begin{align*}
\lim_{t\to 0}\frac{e_t f(x)-f(x)}{t}=I_{e} f(x)=-e^{-x} f(x)
\end{align*}
pointwise. Next, recalling the inequality $(1-e^{-te^{-x}})/t\le e^{-x}$ for all $t>0$, the dominated convergence theorem yields
\begin{align*}
	\lim_{t\to 0}\int_{\bb R}\left(\frac{e_t f(x)-f(x)}{t}-I_e f(x)\right)^2 e^x dx=0
\end{align*}
which shows that $\bmrm{L}(\bb{R}, e^{|x|})\subseteq\D(I_e)$, and therefore $\D(I_e)=\bmrm{L}(\bb{R}, e^{|x|})$. Next, from the definition of $\heb$, it follows that, for any $\eps,\beta>0$, $\heb\in\bmrm{L}(\bb{R}, e^{|x|})$. Hence, $\cc{E}(\eps)\subset\D(I_{e})$ for any $\eps>0$.  Now, for any $\eps,\beta>0$, $e_t\heb(x)=\mf{h}_{\eps,\beta+t}$. By linearity of $e_t$, $\cc{E}(\eps)$ is indeed invariant. Since $\cc{E}(\eps)$ is dense in $\bmrm{L}(\bb{R},e)$, by \cite[Proposition~1.7]{engel-nagel}, it is a core for the generator of $(e_t)_{t\ge 0}$.
\end{proof}


Coming back to the proof of the claim \eqref{it:2c2}, we have already shown in Theorem~\ref{thm:intertwining_pdo}\eqref{main_thm_1:it:1d} that $P_t[\psi]\Hpsi=\Hpsi e_t$ on $\D(H_\psi)$ for any $\psi\in\mathbf{N}_b(\bb{R})$. Therefore, $P_t[\psi]\Hpsi(\cc{E}(\eps))=\Hpsi e_t(\cc{E}(\eps))$. By the invariance of $\cc{E}(\eps)$ under $e_t$, we get
\begin{align*}
P_t[\psi](\Hpsi(\cc{E}(\eps)))\subseteq\Hpsi(\cc{E}(\eps)).
\end{align*}
In other words, $\cc{E}_\psi(\eps)=\Hpsi(\cc{E}(\eps))$ is invariant under the semigroup $P[\psi]$. In the next lemma, we show that $\cc{E}_\psi(\eps)\subset\D(A_\ttt{2}[\psi])$. This is where the assumption \eqref{eq:mpsi_bound} is crucial.
\begin{lem}\hlabel{lem:gen_int_2}
If $\psi(\xi)\neq -\i\ttt{d}\xi$ for some $\ttt{d}>0$ then for any $0<\eps<\eta_\psi$, $\Hpsi (\cc{E}(\eps))\subseteq \D(A_\ttt{2}[\psi])$ and $A_\ttt{2}[\psi]\Hpsi =\Hpsi I_{e}$ on $\cc{E}(\eps)$.
\end{lem}
\begin{proof}
From \eqref{eq:bdd_intertwining} we know, since  $\cc{E}(\eps)\subset \D(\Hpsi)$, that
\begin{align}\hlabel{eq:random}
P_t[\psi]\Hpsi =\Hpsi e_t  \mbox{ on } \cc{E}(\eps).
\end{align}
For any $f\in\cc{E}(\eps)\subset\D(I_e)$, we claim that
 \begin{align*}
 \lim_{t\downarrow 0}\Hpsi\frac{e_t f-f}{t}=\Hpsi I_{e}f.
 \end{align*}
It is enough to prove our claim for $f=\heb$ for $\beta>0$ and $0<\eps<\eta_\psi$. First, we recall that $e_t\heb=\mathfrak{h}_{\eps,\beta+t}$ for all $t\ge 0$. As $\Hpsi$ is a shifted Fourier multiplier operator (hence closed), it suffices to show that the function
 \[\xi\mapsto m_{\Hpsi}\left(\xi+\frac{{\rm{i}}}{2}\right)\left[\frac{\fouho_{\mathfrak{h}_{\eps,\beta+t}}(\xi)-\fouho_{\heb}(\xi)}{t}\right]\]
 converges in $\bmrm{L}(\bb{R})$ as $t \downarrow 0$. From \eqref{eq:fourier_trans}, using the Stirling asymptotic of the gamma function, we have, as $|\xi|\to\infty$,
 \begin{equation*}
 \begin{aligned}
 \sup_{t\in [0,1]}\left|\frac{d}{dt}\fouho_{\mathfrak{h}_{\eps,\beta+t}}(\xi)\right|
 =\frac{1}{\sqrt{2\pi}}\sup_{t\in [0,1]}\left|(\beta+t)^{-\eps-\i\xi-1} (\eps+\i\xi)\Gamma(\eps+\i\xi) \right|\asymp  |\xi|^{\eps+\frac{1}{2}} e^{-\frac{\pi}{2}|\xi|}.
\end{aligned}
 \end{equation*}
 Now, recalling that $m_{H_\psi}\left(\xi+{\rm{i}}/2\right)=\frac{W_{\phi_+}\left(\frac{1}{2}-\i\xi\right)}{W_{\phi_-}\left(\frac{1}{2}+\i\xi\right)}$ and invoking the condition \eqref{eq:mpsi_bound} with the above bound, it follows that for all $0<\eps<\eta_\psi$,
 \begin{align*}
 	\sup_{t\in [0,1]}\int_{\bb R}\left|m_{\Hpsi}\left(\xi+\frac{{\rm{i}}}{2}\right)\right|^2\left|\frac{\fouho_{\mathfrak{h}_{\eps,\beta+t}}(\xi)-\fouho_{\heb}(\xi)}{t}\right|^2d\xi<\infty.
 \end{align*}
 Hence, our claim follows from the dominated convergence theorem. Therefore, from \eqref{eq:random},
 \begin{align*}
 \lim_{t\downarrow 0}\frac{P_t[\psi]\Hpsi f-\Hpsi f}{t}=\lim_{t\downarrow 0}\Hpsi\frac{e_t f-f}{t}=\Hpsi I_{e} f
 \end{align*}
which implies that $\Hpsi f\in \D(A_\ttt{2}[\psi])$ and $A_\ttt{2}[\psi]\Hpsi f=\Hpsi I_{e}f$ for all $f\in\cc{E}(\eps)$, which completes the proof of the lemma.
\end{proof}
\noindent
Since the set $\cc{E}_\psi(\eps)$ is dense in $\LRe$, and invariant under $P[\psi]$, by \cite[Proposition~1.7]{engel-nagel}. it is a core for $A_\ttt{2}[\psi]$.

\subsubsection{Proof of Proposition~\ref{prop:core_ido}\eqref{it:2d}} We proceed by observing that the condition \eqref{eq:mpsi_bound_2} implies that $\delta_\psi=\infty$, where $\delta_\psi$ is defined in \eqref{eq:delta_psi}. This further implies that $\eta_\psi=\infty$, where $\eta_\psi$ is defined in \eqref{eq:mpsi_bound}. Let us first assume condition~\eqref{conditions_i}, that is $\phi_+(0)>0$, which ensures that $m_{H_\psi}$ extends continuously on $\bb S_{[0,1]}$. Now, with the bound in \eqref{eq:mpsi_bound_2} and \hpt{R28}{a similar argument involving the Phragm\'en-Lindel\"of principle as in the proof of Proposition~\ref{cor:multiplier}\eqref{it:kap},} one can establish that, for any $n\in\bb N$,
\hptg{C69}{
\begin{align*}
	\xi\mapsto\sup_{b\in [0,1]}|\xi+\i b|^n|m_{H_\psi}(\xi+\i b)\fou_{\heb}(\xi+\i b)|\in\bmrm[1]{L}(\bb R)\cap\bmrm{L}(\bb R).
\end{align*}
}
Hence, applying Proposition~\ref{multiplier_result}, we get $H_\psi\heb\in\bmrm{L}(\bb R)$ for all $\eps,\beta>0$ and
\begin{align*}
	\fou_{H_\psi\heb}=m_{H_\psi}\fou_{\heb}.
\end{align*}
Since the function on the right-hand side of the above equation decays faster than any polynomial, we obtain that $H_\psi\heb\in\C^\infty_0(\bb R)$. Also, from Remark~\ref{rem:fneq_rem}, we have $m_{\Hpsi}(\xi+{\rm{i}})   =m_{\Hpsi}(\xi)\psi(\xi)$ for all $\xi\in\bb{R}$. Therefore, from Lemma~\ref{lem:gen_int_2}, for any $\beta,\eps>0$, we get
\begin{align}
A_\ttt{2}[\psi]\Hpsi \heb&=\Hpsi I_e \heb \nonumber \\
&=-\frac{e^{-\frac{x}{2}}}{\sqrt{2\pi}}\int_\bb{R} e^{{\rm{i}}\xi   x}m_{\Hpsi}\left(\xi+\i/2\right)\fou_{\heb}\left(\xi-\i/2\right)\,d\xi \nonumber \\
&=-\frac{e^{-x}}{\sqrt{2\pi}}\int_{\bb{R}-\frac{{\rm{i}}}{2}}m_{\Hpsi}(z+{\rm{i}})   \fou_{\heb}(z)e^{{\rm{i}}z   x}\,dz \nonumber \\
&=-\frac{e^{-x}}{\sqrt{2\pi}}\int_\bb{R} m_{\Hpsi}(\xi+{\rm{i}})   \fou_{\heb}(\xi)e^{{\rm{i}}\xi   x}\,d\xi \nonumber \\
&=-\frac{e^{-x}}{\sqrt{2\pi}}\int_{\bb{R}}m_{\Hpsi}(\xi)\psi(\xi)\fou_{\heb}(\xi)e^{{\rm{i}}\xi   x}\,d\xi \nonumber \\
&= A_\ttt{PDO}[\psi]\Hpsi\heb=A_\ttt{IDO}[\psi]\Hpsi\heb \label{eq:pdo=ido}
\end{align}
\hptg{C69}{where we have repeatedly used Lemma~\ref{prop:contour} to change the line of integration.} Also note that the last equality holds as $\xi\mapsto\xi^2 m_{\Hpsi}(\xi)\fou_{H_\psi\heb}(\xi)\in\bmrm[1]{L}(\bb R)$. Finally, by linearity of the operators, the last identity extends to $\cc{E}_\psi(\eps)$, and hence to $\cc{E}_\psi$. \hptg{C74}{Next, to get rid of the condition $\phi_+(0)>0$, we approximate $\psi$ by its small perturbations $\psi_q$ for which we are able to apply the technique discussed above. More formally, we define $\psi_q(\xi)=\psi(\xi) +q$, $q\ge 0$. Let $\psi_q(\xi)=\phi^{(q)}_+(-{\rm{i}}\xi  )\phi^{(q)}_-({\rm{i}}\xi  )$ be the Wiener-Hopf factorization of $\psi_q$. From  \cite[Lemma~7.3, (7.66)]{patie2018}, taking  $\mathfrak{r}=1$ with the notation therein, we have, for all $n\in\bb N$,
\begin{equation}\hlabel{eq:ratio_approx}
\limsup_{|\xi|\to\infty}|\xi|^ne^{-\frac{\pi}{2}|\xi|}\sup_{0\le q\le 1}\left|m_{H_{\psi_q}}\left(\xi+{\rm{i}}/2\right)-m_{H_{\psi}}\left(\xi+{\rm{i}}/2\right)\right|=0.
\end{equation}
Note that with the notations of the aforementioned paper, $\ttt{N}_{\Psi}$ coincides with $\delta_\psi$ in our case, which equals $\infty$ due to \eqref{eq:mpsi_bound_2}. Also, we have replaced $\Gamma(\frac{1}{2}+\i\xi)$ in \cite[(7.66)]{patie2018} simply by $e^{-\frac{\pi}{2}|\xi|}$, since they are asymptotically equivalent by the Stirling approximation formula.}
Therefore, by the dominated convergence theorem and Lemma~\ref{lem:gen_int_2}, we get
\begin{equation}
\begin{aligned}
A_\ttt{2}[\psi]\Hpsi\heb(x)&=\Hpsi \hptg{C65}{I_e}\heb(x) \hlabel{eq:ido_1} \\
&=-\frac{e^{-\frac{x}{2}}}{\sqrt{2\pi}}\int_\bb{R}m_{\Hpsi}\left(\xi+\i/2\right)\fou_{\heb}\left(\xi-\i/2\right)e^{{\rm{i}}\xi   x}\,d\xi \\
&=-\frac{e^{-\frac{x}{2}}}{\sqrt{2\pi}}\lim_{q\downarrow 0}\int_\bb{R} m_{H_{\psi_q}}\left(\xi+\i/2\right)\fou_{\heb}\left(\xi-\i/2\right)e^{{\rm{i}}\xi   x}\,d\xi.
\end{aligned}
\end{equation}
Since $\phi^{(q)}_+(0)>0$, \eqref{eq:pdo=ido} yields
\begin{equation}
\begin{aligned}
&\lim_{q\downarrow 0}\frac{1}{\sqrt{2\pi}}\int_{\bb{R}}m_{H_{\psi_q}}\left(\xi+\i/2\right)\fou_{\heb}\left(\xi-\i/2\right)e^{{\rm{i}}\xi   x}\,d\xi \\
&=-e^{-\frac{x}{2}}\lim_{q\downarrow 0}\frac{1}{\sqrt{2\pi}}\int_\bb{R}\psi_q(\xi)m_{H_{\psi_q}}(\xi)\fou_{\heb}(\xi) e^{{\rm{i}}\xi   x}\,d\xi \\
&=-e^{\frac{x}{2}}\lim_{q\downarrow 0}A_\ttt{IDO}[\psi_q]H_{\psi_q}\heb(x). \hlabel{eq:ido_2}
\end{aligned}
\end{equation}
Finally, we need the following lemma to conclude the proof of this theorem.
\begin{lem} If $\psi$ satisfies \eqref{conditions_iii}, then, for any $n\in\bb N\cup\{0\}$, one has
\begin{align*}
\lim_{q\downarrow 0}(H_{\psi_q}\heb)^{(n)}(x)&=(\Hpsi\heb)^{(n)}(x) \\
\lim_{q\downarrow 0}\int_\bb{R}\ttt{F}H_{\psi_q}\heb(x,y)\mu(dy)
&=\int_\bb{R}\ttt{F}\Hpsi\heb(x,y)\mu(dy) \nonumber
\end{align*}
where $\ttt{F}f(x,y)=f(x+y)-f(x)-y\bbm{1}_{\{|y|\le 1\}}f'(x)$.
\end{lem}
\begin{proof}
By definition of $H_{\psi_q}$, for any $\beta>0$,
\begin{align*}
\lim_{q\downarrow 0}H_{\psi_q}\heb(x)=\frac{1}{\sqrt{2\pi}}e^{-\frac{x}{2}}\lim_{q\downarrow 0}\int_\bb{R} e^{{\rm{i}}\xi   x}m_{H_{\psi_q}}\left(\xi+{\rm{i}}/2\right)\fouho_{\heb}(\xi)\,d\xi,
\end{align*}
which, by \eqref{eq:ratio_approx} and the dominated convergence theorem, equals
$$\frac{1}{\sqrt{2\pi}}e^{-\frac{x}{2}}\int_\bb{R} e^{{\rm{i}}\xi   x}m_{\Hpsi}\left(\xi+\i/2\right)\fouho_{\heb}(\xi)\,d\xi=\Hpsi\heb(x). $$
For the derivatives, we note that, for any $n\in\bb{N}$, $x\in\bb{R}$,
\begin{align*}
(\Hpsi\heb)^{(n)}(x)&=\sqpi e^{-\frac{x}{2}}\int_{\bb{R}} \left({\rm{i}}\left(\xi+\i/2\right)\right)^n m_{\Hpsi}\left(\xi+\i/2\right)\fouho_{\heb}(\xi)\,d\xi \\
&=\lim_{q\downarrow 0}e^{-\frac{x}{2}}\sqpi\int_\bb{R} e^{{\rm{i}}\xi   x} m_{H_{\psi_q}}\left(\xi+\i/2\right)\left({\rm{i}}\left(\xi+\i/2\right)\right)^n\fouho_{\heb}(\xi)\,d\xi \\
&=\lim_{q\downarrow 0}(H_{\psi_q}\heb)^{(n)}(x).
\end{align*}
On the other hand, using the analyticity of the function $z\mapsto (-\i z)^n m_{\Hpsi}(z)=\frac{(-\i z)^n}{\phi_+(-\i z)}\frac{W_{\phi_+}(1-\i z)}{W_{\phi_-}(1+\i z)}$ in $\bb{S}_{[0,1]}$ for $n\ge 1$ and recalling Proposition~\ref{multiplier_result}, we get
\begin{align*}
(\Hpsi\heb)^{(n)}(x)=\sqpi\int_\bb{R} e^{{\rm{i}}\xi   x} ({\rm{i}}\xi  )^nm_{\Hpsi}(\xi)\fou_{\heb}(\xi)\,d\xi.
\end{align*}
Since the modulus of the integrand above decays faster than any polynomial with respect to $\xi$, we deduce that $(\Hpsi\heb)^{(n)}\in \C^\infty_0(\bb{R})$ for any $n\ge 1$, and using \eqref{eq:ratio_approx}, it is easy to see that, as $q \to 0$, $(H_{\psi_q}\heb)^{(n)}\to (\Hpsi\heb)^{(n)}$ in the uniform topology. Therefore, denoting
\begin{align*}
\ttt{I}f(x)=\int_\bb{R}\ttt{F}f(x,y)\mu(dy) \text{  and  }
D_q(x)=H_{\psi_q}\heb(x)-\Hpsi\heb(x)
\end{align*}
we have, as $q \to 0$,
\begin{equation*}
\begin{aligned}
&|\ttt{I}H_{\psi_q}\heb(x)-\ttt{I}\Hpsi\heb(x)|=|\ttt{I}D_q(x)|\\
&\le \int_{|y|\le 1}\left|D_q(x+y)-D_q(x)-y D'_q(x)\right|\mu(dy)+\int_{|y|> 1} \left|D_q(x+y)-D_q (x)\right|\mu(dy)\\
&\le \|D''_q\|_{\infty}\int_{|y|\le 1} y^2\mu(dy)+\|D'_q\|_\infty\int_{|y|>1} |y|\mu(dy)\to 0
\end{aligned}
\end{equation*}
by our assumption on $\psi$. This proves the lemma.
\end{proof}
Therefore, using \eqref{eq:ido_1} and \eqref{eq:ido_2} and the previous lemma, we conclude that
\begin{align*}
A_\ttt{2}[\psi]\Hpsi\heb=A_\ttt{IDO}[\psi]\Hpsi\heb.
\end{align*}
By linearity of the operators, the above identity extends to $\cc{E}_\psi$ which completes the proof of (\ref{it:2d}).

\section{Proof of Theorem~\ref{thm1:main_thm_1}: the multidimensional case}\hlabel{sec:proof2}
The proof of  Theorem~\ref{thm1:main_thm_1}  follows from Theorem~\ref{thm:intertwining_pdo} by a combination of tensorization and similarity transform techniques that we first describe in the general context of semigroups.
\subsection{Tensorization and similarity transform of semigroups and their generators}\hlabel{tens_sim} For $d\in\bb{N}$, consider the $\cc{C}_0$-contraction semigroups $P^{(1)}, P^{(2)},\ldots, P^{(d)}$ defined on the Hilbert space $\bmrm{L}(\bb{R},e)$. For each $t\ge 0$, define $\bm{P}_t$ to be the tensor product of $P^{(1)}_t,\ldots, P^{(d)}_t$. It is plain that $(\bm{P}_t)_{t\geq0}$ is a $\cc{C}_0$-contraction semigroup on $\bmrm{L}(\bb{R}^d,\mathbf{e})$ where $\mathbf{e}(\bm x)=\otimes_{k=1}^d e(\bm x)=e^{\langle\bm x,\bm 1\rangle}$. If $(A^{(k)},\D(A^{(k)}))$ denotes the generator of $P^{(k)}$, $k=1,\ldots,d$, the generator of $(\bm{P}_t)_{t\geq0}$ is given by
\begin{align*}
\hptg{C76}{\bm{A}=\sum_{k=1}^d I\otimes\ldots\otimes I\otimes A^{(k)}\otimes I\otimes\ldots\otimes I.}
\end{align*}
Since we could not find a proper reference regarding the core for the tensor product of generators, we provide the result in the next lemma.
\begin{lem}\hlabel{lem:tensor_core}
$\otimes_{k=1}^d\D(A^{(k)})$ forms a core for $\bm{A}$. In particular, for $k=1,\ldots, d$, if $\cc{D}_k\subseteq\D(A^{(k)})$ are such that $\cc{D}_k$ is a core for $P^{(k)}$, then $\bm{\cD}=\otimes_{k=1}^d\cc{D}_k$ forms a core for $\bm{A}$.
\end{lem}
\begin{proof}
	First note that $\otimes_{k=1}^d\D(A^{(k)})\subset\D(\bm A)$, and it is invariant under the semigroup $\bm P$. Therefore, the first part of the statement follows by \cite[Proposition 1.7]{engel-nagel}. Next, to show that $\bm{\cD}=\otimes_{k=1}^d\cc{D}_k$ is a core for $\bm{A}$, it suffices to show that
	\begin{enumerate}[(i)]
		\item \hlabel{it:dense} $\bm{\cc D}$ is dense in $\LLRe$,
		\item \hlabel{it:image_dense} $(\alpha\bm{I}-\bm{A})(\bm{\cc D})$ is dense in $\LLRe$ for any $\alpha>0$.
	\end{enumerate}
Item~\eqref{it:dense} is true since $\cc{D}_k$ is dense in $\LRe$ for each $k=1,\ldots,d$. For item~\eqref{it:image_dense}, note that for any $\alpha>0$, one has
\begin{align*}
	(\alpha\bm{I}-\bm{A})(\bm{\cc D})=\sum_{k=1}^d \cc{D}_1\otimes\cdots\otimes (\alpha I-A^{(k)})(\cc{D}_k)\otimes\cdots\otimes\cc{D}_d.
\end{align*}
Since $\cc{D}_k$ is a core for $A^{(k)}$, $(\alpha I-A^{(k)})(\cc{D}_k)$ is dense in $\LRe$. Therefore, the set on the right hand side of the above equation is also dense in $\LLRe$. This proves the lemma.
\end{proof}

Once  a semigroup on $\bmrm{L}(\bb{R}^d,\mathbf{e})$ is defined,  one can construct a class of semigroups via similarity transform of the coordinates of $\bb{R}^d$. The resulting semigroup becomes similar to the original one. More specifically, let  $(\bm{P}_t)_{t\geq 0}$ be a semigroup on $\bmrm{L}(\bb{R}^d,\mathbf{e})$, and for any Borel bijection $\mathfrak{g}:\bb{R}^d\to\bb{R}^d$ one can define a new semigroup by similarity transform as follows, for any $t\geq0$,
\begin{align*}
\bm{P}^\mathfrak{g}_t =\mathrm{d}_{\mathfrak{g}}\bm{P}_t \mathrm{d}^{-1}_{\mathfrak{g}}
\end{align*}
where $\mathrm{d}_\mathfrak{g}:\bmrm{L}(\bb{R}^d,\mathbf{e})\to\bmrm{L}(\bb{R}^d,\mathbf{e}_\mathfrak{g})$ with $\mathrm{d}_\mathfrak{g} \f=\f\circ\mathfrak{g}^{-1}$ is an isometry.
The semigroup $(\bm{P}^\mathfrak{g}_t)_{t\geq 0}$ is defined on a new $\bmrm{L}$-space, namely, $\bmrm{L}(\bb{R}^d,\mathbf{e}_\mathfrak{g})$, where $\mathbf{e}_\mathfrak{g}$ is the push forward measure induced by $\mathfrak{g}$. If  $(\bm{P}_t)_{t\geq0}$  is a $\cc{C}_0$-contraction semigroup, so is $(\bm{P}^\mathfrak{g}_t)_{t\geq0}$. Furthermore, denoting by $\bm{A},\bm{A}_\mathfrak{g}$  the generators of $ (\bm{P}_t)_{t\geq0}, (\bm{P}^\mathfrak{g}_t)_{t\geq0}$ respectively, we have
\begin{align*}
\bm{A}_\mathfrak{g} =\mathrm{d}_{\mathfrak{g}}\bm{A} \mathrm{d}^{-1}_{\mathfrak{g}} \mbox{ on } \D(\bm{A}_\mathfrak{g})=\left\{\f\in\bmrm{L}(\bb{R}^d,\mathbf{e}_\mathfrak{g}); \: \f\circ\mathfrak{g}\in\D(\bm{A})\right\}.
\end{align*}
\subsection{Tensorization of Fourier multiplier operators}\hlabel{sec:tens-sim} Since Fourier multiplier operators are closed, it is easy to define their tensor product on multi-dimensional $\bmrm{L}$ spaces. For instance, let $\{(\LL_k,\D(\LL_k))\}_{k=1}^d\subset\mathscr{M}$ and $\bm{\LL}$ be the shifted Fourier multiplier on $\otimes_{k=1}^d\D(\LL_k)$ defined by
\begin{align*}
\fouh_{\bm{\LL}\f}(\bm{\xi})=\prod_{k=1}^d m_{\LL_k}\left(\xi_k+\i/2\right)\fouho_{f_k}(\xi_k)
\end{align*}
where $\f=\otimes_{k=1}^d f_k$, $f_k\in\D(\LL_k)$ for $k=1,\ldots,d$ and $\bm{\xi}=(\xi_1,\ldots, \xi_d) \in \bb{R}^d$. Note that in this case $\fouh$ stands for the multi-dimensional shifted Fourier transform, i.e.~
\begin{align*}
\fouh_{\f}(\bm{\xi})=(2\pi)^{-\frac{d}{2}}\int_{\bb{R}^d} e^{-{\rm{i}}\langle\bm{x},\bm{\xi}+\frac{{\rm{i}}}{2}\bm{1} \rangle}\f(\bm{x})\,d\bm{x}
\end{align*}
where we  recall that $\bm{1}$ is the vector in $\bb{R}^d$ consisting of all $1$'s. We call $\bm{\LL}$  the tensor product of the $\LL_k$'s and write
\begin{align*}
\bm{\LL}=\otimes_{k=1}^d\LL_k.
\end{align*}
The operator $\bm{\LL}$ is closable when restricted on $\otimes_{k=1}^d\D(\LL_k)$ and its closure is the shifted Fourier multiplier operator on $\bmrm{L}(\bb{R}^d,\mathbf{e})$ with the multiplier function
\begin{align*}
\mm_{\bm{\LL}}\left(\bm{\xi}+\frac{{\rm{i}}}{2}\bm{1}\right)=\prod_{k=1}^d m_{\LL_k}\left(\xi_k+\frac{{\rm{i}}}{2}\right).
\end{align*}
When taking the tensor product of Fourier multipliers, we will always mean the shifted Fourier multiplier operator mentioned above. Just like the one-dimensional Fourier multiplier operator, Propositions~\ref{dense_range_intertwining}, \ref{m-group} and \ref{intertwining_equiv} hold for tensor products of shifted Fourier multiplier operator  and tensor product of bounded operators as well. Next, we prove the following lemma regarding the core for the tensor product of Fourier multiplier operators.

\begin{lem}\hlabel{lem:core_tensor}
	Let $(\Lambda_k)_{k=1}^d$ be a sequence of shifted Fourier multiplier operators on $\LRe$ such that $\cc{D}_k$ is a core for $\Lambda_k$ for each $k=1,\ldots, d$. Then, $\bm{\cc{D}}=\otimes_{k=1}^d\cc{D}_k$ is a core for $\bm{\Lambda}=\otimes_{k=1}^d\Lambda_k$.
\end{lem}
\begin{proof}
	We recall, from the proof of Proposition~\ref{dense_range_intertwining}, that $\cc{D}_k$ is a core for $\Lambda_k$ if and only if $\fouho_{\cc{D}_k}$ is dense in $\bmrm{L}(\bb R,1+|m^e_{\Lambda_k}|^2)$. Let us write $\mathbf{m}=\prod_{k=1}^d (1+|m^e_{\Lambda_k}|^2)$. Then, it follows that
\begin{align*}
	\bigotimes_{k=1}^d \bmrm{L}(\bb R,1+|m^e_{\Lambda_k}|^2)=\bmrm{L}(\bb R^d,\mathbf{m})
\end{align*}
and therefore, $\bm{\cc D}$ is dense in $\bmrm{L}(\bb R^d,{\bf m})$. Since $\mathbf{m}\ge 1+\prod_{k=1}^d |m^e_{\Lambda_k}|^2=1+|\mathbf{m}^{\bf e}_{\bm{\Lambda}}|^2$, this implies that $\bm{\cc D}$ is also dense in $\bmrm{L}(\bb R^d, 1+|\mathbf{m}^{\bf e}_{\bm{\Lambda}}|^2)$. Therefore, $\bm{\cD}$ is a core for $\bm{\Lambda}$.
\end{proof}

\begin{prop}\hlabel{prop:tensor_intertwining}
Let, for each $k=1, \ldots,d$, $A_k, B_k \in \mathscr{D}(\bmrm{L}(\bb{R},e))$ and $(\Lambda_k,\D(\Lambda_k)) \in \mathscr{M}$, and, assume that there exists a dense subset $\cc{D}_k$ of $\LRe$ such that
$ A_k\Lambda_k=\Lambda_k B_k  \mbox{ on } \cc{D}_k$. Then,
\begin{align}\hlabel{eq:tensor_intertwining_id}
\bm{A}\bm{\Lambda}=\bm{\Lambda}\bm{B} \ \ \ \mbox{on \ $\bm{\cc D}$}
\end{align}
where $\bm{\cc D}=\otimes_{k=1}^d \cc{D}_k$, $\bm{A}=\otimes_{k=1}^d A_k, \ \bm{B}=\otimes_{k=1}^d B_k$, and $\bm{\Lambda}=\otimes_{k=1}^d \Lambda_k$. In particular, when the $A_k,B_k$'s are bounded operators and $\cc{D}_k$ is a core for $\Lambda_k$ for each $k$, the above identity extends to $\D(\bm{\Lambda})$.
\end{prop}
\begin{proof} Since any $\f\in\bm{\cc D}=\otimes_{k=1}^d\cc{D}_k$ \hptg{C78}{is a linear combination of functions of the form} $\f(\bm{x})=\prod_{k=1}^d f_k(x_k)$,
	where $\bm{x}=(x_1,\ldots, x_d)$ and $f_k\in\cc{D}_k$, by the definition of the tensor product of operators, we have
	\begin{equation}\hlabel{eq:tens_int_pf}
		\begin{aligned}
			\bm{A}\bm{\Lambda}\f&=\bm{A}(\otimes_{k=1}^d\Lambda_k f_k)
			=\otimes_{k=1}^d A_k\Lambda_k f_k=\otimes_{k=1}^d \Lambda_k B_k f_k =\bm{\Lambda}\bm{B}\f,
		\end{aligned}
	\end{equation}
which proves \eqref{eq:tensor_intertwining_id}. When $A_k,B_k$ are bounded operators, so are $\bm{A}$ and $\bm{B}$. Moreover, if $\cc{D}_k$ is a core for $\Lambda_k$ for $k=1,\ldots ,d$, then by Lemma~\ref{lem:core_tensor}, $\bm{\cc D}=\otimes_{k=1}^d\cc{D}_k$ is a core for $\bm{\Lambda}$. Therefore, boundedness of $\bm{A}, \bm{B}$ and \tred{closedness} of $\bm{\Lambda}$ ensure that \eqref{eq:tens_int_pf} extends for all $\f\in\D(\bm{\Lambda})$.
\end{proof}
\begin{prop}\hlabel{prop:core}
	Let $\bm{\psi}=(\psi_1,\psi_2,\ldots, \psi_d)\in\mathbf{N}^d_b(\bb{R})$ such that $\psi_k(\xi)=\phi_{+,k}(-\mathrm{i}\xi)\phi_{-,k}(\mathrm{i}\xi)$ and  $\psi_k(\xi)\neq -\i\ttt{d}\xi$ for any $k=1,2,\ldots, d$. Assume that $M \in {\rm{GL}}_d(\bb{R})$ and let $\bm{H}^M_{\bm{\psi}}$ be defined as in Theorem~\ref{th:spectral_decomp}. Then,
	 $\bm{H}^M_{\bm{\psi}}:\D(\bm{H}^M_{\bm{\psi}})\to\bmrm{L}(\bb{R}^d,{\mathbf{e}})$ is a linear operator which is densely defined, closed, injective and has dense range with
	\begin{align}\hlabel{eq:domain_Hpsi}
		\D(\bm{H}^M_{{\bm{\psi}}})=\left\{\f\in\bmrm{L}(\bb{R}^d,{\mathbf{e}}_M);\ \bm{\xi}=(\xi_1,\ldots,\xi_d) \mapsto\fouh_{\f\circ M^{-1}}(\bm{\xi})\prod_{k=1}^d\frac{W_{\phi_{+,k}}(\frac{1}{2}-{\rm{i}}\xi  _k)}{W_{\phi_{-,k}}(\frac{1}{2}+{\rm{i}}\xi  _k)}\in\bmrm{L}(\bb{R}^d)\right\}.
	\end{align}
	Moreover, for sufficiently small $\eps>0$, $\bm{\cc{E}}(\eps)\subset\D(\bm{H}^M_{\bm{\psi}})$ and $\bm{\cc{E}}^M_{\bm{\psi}}(\eps):=\bm{H}^M_{\bm{\psi}}(\bm{\cc{E}}(\eps))$ is dense in $\bmrm{L}(\bb{R}^d,{\mathbf{e}}_M)$ where
	\begin{equation*}
		\bm{\cc{E}}(\eps)=\bigotimes_{k=1}^d\cc{E}(\eps) \mbox{ with } \  \cc{E}(\eps)=\Span\{\heb; \: \beta>0 \}
	\end{equation*}
and $\heb(x)=e^{-(\frac{1}{2}+\eps)x}e^{-\beta e^{-x}}$, $x\in\bb R$.
\end{prop}
\begin{proof}
It is   enough to prove the above proposition  for $d=1$. The general case will follow by tensorization and similarity transform technique. When $d=1$, \eqref{eq:domain_Hpsi} follows from Section~\ref{ss:fourier_mult} and the fact that $\cc{E}_\psi(\eps)$ is a core for $A_\ttt{2}[\psi]$ for sufficiently small $\epsilon>0$, follows from Proposition~\ref{prop:core_ido}\eqref{it:2c2}.
\end{proof}
 We are now ready to prove Theorem~\ref{thm1:main_thm_1}.

\subsection{Proof of Theorem~\ref{thm1:main_thm_1}} \label{ss:main_thm_1} Since the weak similarity operators are Fourier multiplier operators, and weak similarity is stable by the similarity transform, it suffices to prove the theorem when $M={\rm Id}$. From Theorem~\ref{thm:intertwining_pdo}\eqref{main_thm_1:it:1b}, it follows that for each $k=1,\ldots, d$, $A_\ttt{PDO}[\psi_k]\Lambda_{\psi_k}=\Lambda_{\psi_k}A_\ttt{PDO}[\psi_0]$ on $\cc{D}(\bb R)$. Since $\bm{A}_{\ttt{PDO}}[\bm \psi]=\otimes_{k=1}^d A_{\ttt{PDO}}[\psi_k]$ and $\bm{\Lambda}_{\bm\psi}=\otimes_{k=1}^d \Lambda_{\psi_k}$, Proposition~\ref{prop:tensor_intertwining} yields that
\begin{align*}
	\bm{A}_\ttt{PDO}[\bm\psi]\bm{\Lambda}_{\bm\psi}=\bm{\Lambda_\psi}\bm{A}_{\ttt{PDO}}[\bm{\psi}_0] \ \text{ on }\: \bm{\cc{D}}(\bb R^d)
\end{align*}
where $\bm{\cc{D}}(\bb R^d)=\otimes_{k=1}^d\cc{D}(\bb R)$.

When $\bm{\psi}\in\NN^d_+(\bb R)$, it means that $\psi_k\in\NN_+(\bb R)$ for each $k=1,\ldots, d$. Therefore, $\Lambda_{\psi_k}\in\mathscr{B}(\LRe)$ for each $k=1,\ldots, d$. Hence, $\bm{\Lambda_\psi}\in\mathscr{B}(\LLRe)$. This completes the proof of the theorem.

\subsection{Proof of Theorem~\ref{thm:semigroup_generation}} \hlabel{ss:pf_thm:semigroup_generation}
\subsubsection{Proof of Theorem~\ref{thm:semigroup_generation}\eqref{it:bessel_generation}} We first observe that $\bm{A}_{\ttt{PDO}}[\bm\psi]=\otimes_{k=1}^d A_\ttt{PDO}[\psi_k]$ for any $\bm{\psi}=(\psi_1,\ldots,\psi_k)\in\NN^d_b(\bb R)$. Next, from Theorem~\ref{thm:intertwining_pdo}\eqref{main_thm_1:it:0a}, we know that the closure (in $\LRe$) of $(A_\ttt{PDO}[\psi_0],\C^\infty_c(\bb R))$ generates the log-squared-Bessel semigroup $Q$ on $\LRe$, and its restriction on $\C_0(\bb R)\cap\LRe$ extends to a Feller semigroup on $\C_0(\bb R)$, whose generator is the closure (in $\C_0(\bb R)$) of $(A_0,\C^\infty_c(\bb R))$. Note that the tensor product $\bm{Q}=\otimes_{k=1}^d Q$ is generated by the closure of $\bm{A}_{\ttt{PDO}}[\bm\psi_0]$ restricted on $\otimes_{k=1}^d \C^\infty_c(\bb R)$, thanks to Lemma~\ref{lem:tensor_core}. Since $\otimes_{k=1}^d \C^\infty_c(\bb R)\subset\C^\infty_c(\bb R^d)$ and both of these two sets are dense in $\LLRe$, the closure of $(\bm{A}_\ttt{PDO}[\bm\psi_0],\C^\infty_c(\bb R^d))$ also generates $\bm{Q}$. Moreover, due to tensorization, $\bm{Q}\in\cc{C}^+_0(\LLRe)$. This completes the proof of item~\eqref{it:bessel_generation}.

\subsubsection{Proof of Theorem~\ref{thm:semigroup_generation}\eqref{it:bijection}} For any $\bm{\psi}=(\psi_1,\psi_2,\ldots,\psi_d)\in\mathbf{N}^d_b(\bb{R})$, we define
\begin{align*}
\bm{P}[\bm{\psi}]=\otimes_{k=1}^d P[\psi_i].
\end{align*}
Clearly, $\bm{P}[\bm{\psi}]\in\cc{C}_0(\bmrm{L}(\bb{R}^d,\mathbf{e}))$. Also, from Theorem~\ref{thm:intertwining_pdo}\eqref{eq:A_2=A_PDO}, we infer that
\begin{align}\hlabel{eq:A_2=A_PDO_mult}
	\bm{A}_\ttt{2}[\bm\psi]=\bm{A}_{\ttt{PDO}}[\bm\psi] \ \ \text{on} \ \ \bm{\Lambda_\psi}(\bm{\cD}(\bb R^d)).
\end{align}
Since $\bm{\Lambda_\psi}(\bm{\cD}(\bb R^d))=\otimes_{k=1}^d \Lambda_{\psi_k}(\cD(\bb R))$ is the tensor product of dense subsets of $\LRe$, it is therefore dense in $\LLRe$. Next, for $M\in\mathrm{GL}_d(\bb R)$ and  $\f\in\bmrm{L}(\bb{R}^d,{\mathbf{e}}_M)$, let us define
\begin{align*}
\FsemiM\f=\bm{P}[\bm{\psi}](\f\circ M)\circ M^{-1}.
\end{align*}
From the discussion in Subsection~\ref{tens_sim}, it follows that $\FsemiM\in\cc{C}_0(\bmrm{L}(\bb R^d,\mathbf{e}_M))$ and $\bm{A}^M_\ttt{2}[\bm\psi]=\mathrm{d}_M\bm{A}_\ttt{2}[\bm\psi]\mathrm{d}_{M^{-1}}$ is its $\bmrm{L}(\bb R^d,\mathbf{e}_M)$-generator. Hence, \eqref{eq:A_2=A_PDO_mult} yields that
\begin{align*}
	\bm{A}^M_\ttt{2}[\bm\psi]=\bm{A}^M_{\ttt{PDO}}[\bm\psi] \ \ \text{on} \ \ \mathrm{d}_{M}\bm{\Lambda_\psi}(\bm{\cD}(\bb R^d))
\end{align*}
and $\mathrm{d}_{M}\bm{\Lambda_\psi}(\bm{\cD}(\bb R^d))$ is also dense in $\LLReM$ as $\mathrm{d}_M$ is an isometry between $\LLReM$ and $\LLRe$.
The second part of this item follows from the proof of Theorem~\ref{thm:intertwining_pdo}\eqref{main_thm_1:it:1d}.
\subsubsection{Proof of Theorem~\ref{thm:semigroup_generation}\eqref{it:ws_semigroups}} From Theorem~\ref{thm:intertwining_pdo}\eqref{eq:bdd_intertwining}, we know that for any $\psi\in\mathbf{N}_b(\bb{R})$, $P[\psi]\Lambda_\psi=\Lambda_\psi Q$ on $\D(\Lambda_\psi)$, where $Q$ is the log-squared-Bessel semigroup. Let $M=\mathrm{Id}$. Then, $\bm{P}^M[\bm{\psi}]=\otimes_{k=1}^d P[\psi_k]$ and therefore, using Proposition~\ref{prop:tensor_intertwining}, we conclude that
\[\bm{P}[\bm{\psi}]\bm{\Lambda}_{\bm{\psi}}=\bm{\Lambda}_{\bm{\psi}}\bm{Q} \ \text{ on } \ \D(\bm{\Lambda_\psi}) \] where $\bm{\Lambda_\psi}=\otimes_{k=1}^d\Lambda_{\psi_k}$ and $\bm{Q}=\otimes_{k=1}^d Q$. This proves \eqref{eq:ws} when $M=\mathrm{Id}$. The general case will follow from the similarity transform argument.

\subsubsection{Proof of Theorem~\ref{thm:semigroup_generation}\eqref{it:thm2adj}} Let us first assume that $M={\rm Id}$. Then, by tensorization,  the adjoint of $\bm{P}[\bm\psi]=\otimes_{k=1}^d P[\psi_k]$ in $\LLRe$ equals
\[\widehat{\bm{P}}[\bm\psi]=\otimes_{k=1}^d\widehat{P}[\psi_k]=\otimes_{k=1}^d P[\ov{\psi}_k]=\bm{P}[\ov{\bm\psi}].\]
Indeed, $\widehat{\bm{P}}[\bm\psi]=\bm{P}[\bm\psi]$ if and only if $\ov{\psi}_k=\psi_k$ for each $k=1,\ldots,d$. Next, we show that for any $M\in\mathrm{GL}_d(\bb R)$, the adjoint of $\bm{P}^M[\bm\psi]$ in $\LLReM$ equals ${\bm P}^M[\ov{\bm\psi}]$. We recall that the operator $\mathrm{d}_M:\bmrm{L}(\bb{R}^d,{\mathbf{e}}_M)\to\bmrm{L}(\bb{R}^d,\mathbf{e})$ defined by $\mathrm{d}_M \f=\f\circ M^{-1}$ is an isometry and $\widehat{\mathrm{d}}_M=\mathrm{d}_{M^{-1}}$. Now, for any $\mathbf{f},\mathbf{g}\in\LLReM$ and $t\ge 0$, we have
\begin{align}
	\int_{\bb{R}^d}\bm{P}^M_t[\bm{\psi}] \f(\bm{x}) \g(\bm{x}){\mathbf{e}}_M(\bm{x})\,d\bm{x}&=\int_{\bb{R}^d}\bm{P}_t[\bm{\psi}](\f\circ M)(M^{-1}\bm{x}) \g(\bm{x})\mathbf{e}(M^{-1}\bm{x})\,d\bm{x} \nonumber \\
	&=\det(M)\int_{\bb{R}^d} \bm{P}_t[\bm{\psi}](\f\circ M)(\bm{y})\g(M\bm{y})\mathbf{e}(\bm{y})\,d\bm{y} \nonumber \\
	&=\det(M)\int_{\bb{R}^d}\bm{P}_t[\ov{\bm{\psi}}](\g\circ M)(\bm{y}) \f(M\bm{y})\mathbf{e}(\bm{y})\,d\bm{y} \nonumber \\
	&=\int_{\bb{R}^d}\bm{P}^M_t[\overline{\bm{\psi}}]\g(\bm{y}) \f(\bm{y}){\mathbf{e}}_M(\bm{y})\,d\bm{y}. \nonumber
\end{align}
Therefore, $\widehat{\bm P}^M[\bm\psi]={\bm P}^M[\ov{\bm\psi}]$ and ${\bm P}^M[\bm\psi]$ is self-adjoint in $\LLReM$ if and only if $\bm\psi=\ov{\bm\psi}$.

\subsection{Proof of Theorem~\ref{th:spectral_decomp}}\hlabel{sec:proof_th:spectral_decomp} Again, we assume that $M=\mathrm{Id}$ as the general case will follow by similarity transform. We recall from  Theorem~\ref{thm:intertwining_pdo}\eqref{eq:spect_decomp} that for any $k=1,\ldots, d$, we have $P_t[\psi_k]H_{\psi_k}=H_{\psi_k} e_t$ on $\D(H_{\psi_k})$ where $H_{\psi_k}\in\mathscr{M}$ associated to
\begin{align*}
	m^e_{\psi_k}(\xi)=m_{H_{\psi_k}}(\xi+\i/2)=\frac{W_{\phi_+}(\frac{1}{2}-\i\xi)}{W_{\phi_-}(\frac{1}{2}+\i\xi)}.
\end{align*}
Since $\bm{e}_t=\otimes_{k=1}^d e_t$, applying Proposition~\ref{prop:tensor_intertwining} with $\bm{\Hpsi}=\otimes_{k=1}^dH_{\psi_i}$, the theorem follows.

\subsection{Proof of Theorem~\ref{thm:spectrum}}\hlabel{sec:proof4} Let us prove this theorem for $M=\mathrm{Id}$ as the general case  follows by the similarity transform. We first prove items \eqref{it:point_spect} and \eqref{it:residual_spect} followed by item \eqref{it:spect}. If $m_{\Hpsi}(\cdot+{\rm{i}}/2)\in\bmrm{L}(\bb{R})$, there exists a unique function $\Jp\in\bmrm{L}(\bb{R},e)$ such that $\fouho_{\!\Jp}=m_{\Hpsi}(\cdot+{\rm{i}}/2)$. Thus, for any $f\in\D(\Hpsi)$, we have
\begin{align}
\Hpsi f=\Jp\ast f \nonumber
\end{align}
where $\ast$ stands for additive convolution. Also, $m_{\Hpsi}\left(\cdot+{\i \over 2}\right)\in\bmrm{L}(\bb{R})$ implies that $\C^\infty_c(\bb{R})\subset\D(\Hpsi)$. Now, for any $ f \in \C^\infty_c(\bb{R})$,
\begin{align}\hlabel{eq:conv_l2}
\int_\bb{R} (|\Jp|\ast| f |)^2 (x) e^x dx &=\int_\bb{R}\left(\int_\bb{R}|\Jp(x-y)|| f (y)| dy\right)^2 e^x dx \\ &\le\left(\int_\bb{R}\left(\int_\bb{R}|\Jp(x-y)|^2 e^x dx\right)^{\frac{1}{2}}| f (y)| dy\right)^2 \nonumber \\
&=\|\Jp\|^2_{\bmrm{L}(\bb{R},e)}\left(\int_\bb{R} e^{\frac{y}{2}}| f (y)|dy\right)^2<\infty \nonumber
\end{align}
where the first inequality follows from Minkowski's integral inequality. Also, for all $ f \in \C^\infty_c(\bb{R})$ and $g\in\bmrm{L}(\bb{R},e)$,
\begin{align}\hlabel{eq:eigenfunction_fubini}
\langle P_t[\psi]\Hpsi f , g\rangle_{\bmrm{L}(\bb{R},e)}=\langle\Hpsi f , \widehat{P}_t[\psi] g\rangle_{\bmrm{L}(\bb{R},e)}=\int_{\bb{R}}\left(\int_{\bb{R}}\Jp(x-y) f (y)\,dy\right)\widehat{P}_t[\psi]g(x)e^x\,dx.
\end{align}
Since \eqref{eq:conv_l2} yields that $|\Jp|\ast| f |\in\bmrm{L}(\bb{R},e)$, by Cauchy-Schwarz inequality, we have
\begin{align*}
\int_\bb{R}\int_\bb{R}|\Jp (x-y) f (y)\widehat{P}_t[\psi] g(x)| e^x dx\ dy <\infty.
\end{align*}
Applying Fubini Theorem  on the right-hand side of \eqref{eq:eigenfunction_fubini}, for all $ f \in \C^\infty_c(\bb{R})$ and $g\in\bmrm{L}(\bb{R},e)$, we have
\begin{align}\hlabel{spectral_lhs}
\langle P_t[\psi]\Hpsi  f , g\rangle_{\bmrm{L}(\bb{R},e)}=\int_{\bb{R}}\left(\int_{\bb{R}}\Jp(x-y)\widehat{P}_t[\psi]g(x) e^x\,dx\right)  f (y)\,dy.
\end{align}
On the other hand,
\begin{align}
\langle P_t[\psi]\Hpsi f , g\rangle_{\bmrm{L}(\bb{R},e)}&=\langle \Hpsi e_t  f ,g\rangle_{\bmrm{L}(\bb{R},e)} \nonumber \\
&=\int_{\bb{R}}\Hpsi e_t  f (x)g(x)e^x\,dx=\int_{\bb{R}}\left(\int_{\bb{R}}\Jp(x-y) e_t  f (y)\,dy\right)  g(x)e^x\,dx\nonumber\\
&=\int_{\bb{R}}\left(\int_{\bb{R}}\Jp(x-y) f (y)e^{-te^{-y}}\,dy\right)g(x)e^x\,dx\nonumber \\
&=\int_{\bb{R}}\left(\int_{\bb{R}}\Jp(x-y)g(x)e^x\,dx\right) f(y)e^{-te^{-y}}\,dy \hlabel{spectral_rhs}
\end{align}
where for the last identity we have invoked again Fubini Theorem.
From \eqref{spectral_lhs}, \eqref{spectral_rhs} and using the density of $\C^\infty_c(\bb{R})$ in $\bmrm{L}(\bb{R},e)$, we obtain for almost every $y$,
\begin{align*}
\int_{\bb{R}}\Jp(x-y)\widehat{P}_t[\psi]g(x)e^x\,dx=e^{-te^{-y}}\int_{\bb{R}}\Jp(x-y)g(x)e^x\,dx.
\end{align*}
\hpt{R30}{Since $y\mapsto\uptau_{-y}\Jp= \Jp(\cdot - y)$ is continuous in $\bmrm{L}(\bb R,e)$, both sides of the above identity are continuous functions with respect to $y$. Therefore, for all $y\in\bb{R}$ and $g\in\bmrm{L}(\bb{R},e)$,}
\begin{align*}
\int_{\bb{R}}P_t[\psi]\tau_{-y}\Jp(x) g(x) e^x\,dx=e^{-te^{-y}}\int_{\bb{R}}\tau_{-y}\Jp(x)g(x)e^x\,dx, \nonumber
\end{align*}
which further implies that for all $y\in\bb{R}$,
\begin{align*}
P_t[\psi]\tau_{-y}\Jp=e^{-te^{-y}}\tau_{-y}\Jp \quad \text{in  $\bmrm{L}(\bb R,e)$}. \nonumber
\end{align*}
When $d>1$, $\Mpsi\in\bmrm{L}(\bb{R}^d)$ if and only if $m_{H_{\psi_i}}\in\bmrm{L}(\bb{R})$ for all $k=1,\ldots,d$. This implies
\begin{equation*}
\widehat{\fouh}_{\!\!\!\!\!\Mpsi}=\otimes_{k=1}^d \widehat{\fouho}_{\!\!m_{H_{\psi_i}}}.
\end{equation*}
Therefore, if $\Mpsi\in\bmrm{L}(\bb{R}^d)$ then, for any $t>0$, $\bm{P}_t[\bm{\psi}]\tau_{-\tred{\bm y}}\mathbf{J}_{\bm{\psi}}=
e^{-t\langle e(-\bm{y}),\bm{1}\rangle}\mathbf{J}_{\bm{\psi}}$ for all $y\in\bb{R}$, where $\mathbf{J}_{\bm{\psi}}=\otimes_{k=1}^d J_{\psi_i}$. Since $\Mpsi$ is bounded, $\bm{\Hpsi}$ is a bounded operator. Taking the adjoint of the intertwining relation $\bm{P}_t[\bm{\psi}]\bm{\Hpsi}=\bm{\Hpsi}\bm{e}_t$, we get
\begin{equation*}
\bm{e}_t\widehat{\bm{H}}_{\bm\psi}=\widehat{\bm{H}}_{\bm\psi}\widehat{\bm{P}}_t[\bm{\psi}].
\end{equation*}
Since $\sigma_r(\bm{P}_t[\bm{\psi}])\subseteq\sigma_p(\widehat{\bm{P}}_t[\bm{\psi}])$ and  $\sigma_p(\bm{e}_t)=\emptyset$, we must have $\sigma_p(\widehat{\bm{P}}_t[\bm \psi])=\emptyset$. Therefore, $\sigma_r(\bm{P}_t[\bm{\psi}])=\emptyset$, which concludes the proof of item (\ref{it:point_spect}).

Item (\ref{it:residual_spect}) follows from item (\ref{it:point_spect}) by replacing $\bm{P}_t[\bm{\psi}]$ by $\bm{P}_t[\ov{\bm{\psi}}]$. Now coming back to item \eqref{it:spect}, if one of the conditions in items \eqref{it:point_spect} and \eqref{it:residual_spect} holds, then item \eqref{it:spect} is trivially true. We again prove this result for $d=1$ and the general case would be the routine application of tensorization and similarity transform. Let us assume that $\mpsi(\cdot+{\rm{i}}/2)$ is bounded and not in $\bmrm{L}(\bb{R})$. We consider $\varphi\in \C^\infty_c(\bb{R}), \varphi\ge 0$ with $\supp\varphi\subset (-1,1)$ and for any $y\in\bb{R}$, \hpt{R32}{$\varphi^{(y)}_n(x)={n}\varphi(n(x-y))$}. We first show that $(\varphi^{(y)}_n)_{n\ge 1}$ (when normalized to have unit norm) are approximate eigenfunctions of $e_t$ corresponding to the approximate eigenvalue $e^{-te^{-y}}$. We note that, for any $y\in\bb{R}$,
\begin{align}
\|e_t\phin-e^{-te^{-y}}\phin\|^2_{\bmrm{L}(\bb{R},e)} \nonumber
&=\int_{\bb{R}}n^2e^{-2te^{-y}}\left[e^{-te^y(e^{-(x-y)}-1)}-1\right]^2\varphi(n(x-y))^2e^xdx \nonumber \\
&=e^{-2te^{-y}+y}\int_{\bb{R}}n\left[e^{-te^y(e^{-\frac{x}{n}}-1)}-1\right]^2\varphi(x)^2e^{\frac{x}{n}}dx. \hlabel{eq:app_dct}
\end{align}
Using dominated convergence theorem, it follows that the expression in \eqref{eq:app_dct} converges to $0$ as $n$ tends to infinity.  Also, for all $y\in\bb R$ and $n\ge 1$,
\begin{align*}
\|\varphi^{(y)}_n\|^2_{\LRe}=\int_{\bb R} n^2\varphi(n(x-y))^2 e^{x} dx=ne^y\int_{-1}^1 \varphi(x)^2 e^{\frac{x}{n}}dx
\end{align*}
which implies that $\|\varphi^{(y)}_n\|_{\bmrm{L}(\bb R,e)}\to\infty$ as $n\to\infty$. As a result, these functions, when normalized, form approximate eigenfunctions of $e_t$ corresponding to the approximate eigenvalue $e^{-te^{-y}}$. Now, we claim that the functions $\Hpsi\phin$, when normalized, are approximate eigenfunctions of $P_t[\psi]$ corresponding to the approximate eigenvalue $e^{-te^{-y}}$. \hptg{C81}{To prove this, it suffices to show that} \hpt{R31}{$\lim_{n\to\infty}\|P_t[\psi]\Hpsi\phin-e^{-te^{-y}}\Hpsi\phin\|_{\LRe}=0$} and $\{\|\Hpsi\phin\|_{\LRe}\}_{n\ge 1}$ is bounded below.  We first show that $\|\Hpsi\phin\|_{\bmrm{L}(\bb{R},e)}\to\infty$ as $n\to\infty$. Indeed, we observe that, for all $n\ge 1$,
\begin{align*}
\fouho_{\phin}(\xi)=e^{-\mathrm{i}y+\frac{y}{2}}\fou_\varphi\left(\frac{\xi}{n}+\frac{{\rm{i}}}{2n}\right).
\end{align*}
Therefore,
\begin{align*}
\fouho_{\Hpsi\phin}(\xi)=e^{-\mathrm{i}y+\frac{y}{2}}\mpsi\left(\xi+\frac{{\rm{i}}}{2}\right)\fou_\varphi\left(\frac{\xi}{n}+\frac{{\rm{i}}}{2n}\right).
\end{align*}
Using the isometry of the shifted Fourier transform followed by Fatou's lemma, we obtain
\begin{align}
\liminf_{n\to\infty}\|\Hpsi\phin\|^2_{\bmrm{L}(\bb{R},e)}&=\liminf_{n\to\infty}\|\fouho_{\Hpsi\phin}\|^2_{\bmrm{L}(\bb{R})} \nonumber\\
&=e^y\liminf_{n\to\infty}\int_\bb{R}\left|\mpsi\left(\xi+\frac{{\rm{i}}}{2}\right)\right|^2\left|\fou_\varphi\left(\frac{\xi}{n}+\frac{{\rm{i}}}{2n}\right)\right|^2d\xi \nonumber\\
&\ge e^y\int_{\bb{R}}\left|\mpsi\left(\xi+\frac{{\rm{i}}}{2}\right)\right|^2\liminf_{n\to\infty}\left|\fou_\varphi\left(\frac{\xi}{n}+\frac{{\rm{i}}}{2n}\right)\right|^2d\xi \nonumber \\
&=\|\varphi\|^2_{\bmrm[1]{L}(\bb{R})}\int_{\bb{R}}\left|\mpsi\left(\xi+\frac{{\rm{i}}}{2}\right)\right|^2d\xi=\infty. \hlabel{eq:app_infty}
\end{align}
With the aid of the intertwining relationship $P_t[\psi]\Hpsi=\Hpsi e_t$ and the fact that  $e^{-te^{-y}}$ is an approximate eigenvalue for $e_t$, one deduces that
\begin{align*}
	\lim_{n\to\infty}\|P_t[\psi]\Hpsi\phin-e^{-te^{-y}}\Hpsi\phin\|_{\bmrm{L}(\bb{R},e)}=0.
\end{align*}
Therefore, \eqref{eq:app_infty} ensures that the sequence $(\Hpsi\phin)_{n\geq1}$, when normalized, are approximate eigenfunctions of $P_t[\psi]$ corresponding to the approximate eigenvalue $e^{-te^{-y}}$. When $m_{H_\psi}(\cdot+\frac{{\rm{i}}}{2})^{-1}$ is bounded, we have $\widehat{P}_t[\psi]H_{\ov{\psi}}=H_{\ov{\psi}}e_t$. Therefore, following the same argument as before, we have $e^{t\bb{R}_-}\subseteq\sigma_{ap}(\widehat{P}_t[\psi])$. Since $\overline{\sigma(P_t[\psi])}=\sigma(\widehat{P}_t[\psi])$, we conclude that for all $t\ge 0$, $e^{t\bb{R}_-}\subseteq{P_t[\psi]}$.

For the item (\ref{it:cont_spect}), let $\bm{P}_t[\bm{\psi}]$ be self-adjoint, i.e.~ $\bm{\psi}=\bm{\ov{\psi}}$. Then, for all $k=1,2\ldots, d$, $m_{H_{\psi_k}}=(\ov{m}_{H_{\psi_k}})^{-1}$, which means that $H_{\psi_k}$ is unitary. Hence, $\bm{\Hpsi}$ is unitary as well and therefore, $\sigma(\bm{P}[\bm{\psi}])=\sigma_c(\bm{e}_t)=e^{t\bb{R}_-}$ for all $t\ge 0$.

\subsection{Proof of Theorem~\ref{thm:core_multi_d}} \label{sec:thm29} In the proof of Theorem~\ref{thm:intertwining_pdo}\eqref{main_thm_1:it:1bc} we have seen that for any $\psi\in\NN_+(\bb R)$, the closure of $(A_\ttt{PDO}[\psi],\Lambda_\psi(\C^\infty_c(\bb R))$ generates $P[\psi]$, which is the $\LRe$-extension of the log-self-similar Feller semigroup associated to $\psi$. By duality, $\widehat{P}[\psi]=P[\ov{\psi}]$ also corresponds to the log-self-similar Feller semigroup associated with $\ov{\psi}$. Hence, for any $\psi\in\NN_b(\bb R)$ and $f\in\C_0(\bb R)\cap\LRe$ we have for all $t\ge 0$ and $x\in\bb R$,
\begin{align*}
	P_t[\psi] f(x)=\bb{E}_x[f(X_t)]
\end{align*}
where $(X_t)_{t\ge 0}$ is the log-self-similar Feller process associated to $\psi$. Therefore, item~\eqref{it:feller_semigroup} is a direct consequence of tensorization and similarity transform.

Item~\eqref{it:core} again follows from Proposition~\ref{prop:core_ido}\eqref{it:2d} followed by tensorization, similarity transform, and Proposition~\ref{prop:core}.

\section{Examples}\hlabel{examples}
\subsection{Spectrally negative self-similar processes}\hlabel{ex:spectrally_neg}
Let $\psi\in\mathbf{N}(\bb{R})$ be the L\'evy-Khintchine exponent of a conservative spectrally negative L\'evy process with a non-negative mean, that is, $\psi(0)=0, \psi'(0)\geq0$ and  $\mu(0,\infty)=0$ in \eqref{eq:defLKe}. Note that this entails that the underlying L\'evy process, and, hence the associated self-similar process,  do not have  positive jumps. In this case, the Wiener-Hopf factorization of $\psi$ is \[\psi(\xi)=-{\rm{i}}\xi   \phi({\rm{i}}\xi   ), \: \xi\in\bb{R}, \] where $ \phi\in\B$ with a L\'evy measure which is absolutely continuous with a non-increasing density. Let us write $m_{\Lambda_\psi}(z)=\frac{\Gamma(1+{\rm{i}}z)   }{W_{\phi}(1+{\rm{i}}z)   }$. \hptg{C84}{From the properties of Bernstein-gamma functions recalled in Section \ref{bernstein}, it is plain that $m_{\Lambda_\psi}$ is analytic in the strip $\bb{S}_{[0,1)}$ and for $z\in\bb{S}_{(0,\infty)}$, $m_{\Lambda_\psi}(-\i z)$ is the Mellin transform of a positive random variable, denoted by $I_\psi$,  which has finite moments of all order strictly larger than $-1$, see e.g.~\cite{Patr} and references therein}. \hptg{C85}{Therefore, $|m_{\Lambda_\psi}(\cdot+{\rm{i}}/2)|$ is bounded on $\bb{R}$, and the shifted Fourier multiplier operator $\Lambda_\psi$ with $m_{\Lambda_\psi}$ as the multiplier is a bounded operator, and belongs to $\mathscr{M}$.}  Moreover, by  Theorem~\ref{thm:intertwining_pdo}\eqref{main_thm_1:it:1d}, $P_t[\psi]\Lambda_\psi=\Lambda_\psi Q_t$ on $\bmrm{L}(\bb{R},e)$, $Q$ being the log-squared Bessel semigroup. In the next lemma, we show that $\Lambda_\psi$ is a convolution kernel with respect to a probability measure, which is the additive convolution analogue of  \cite[Theorem 7.1]{patiesavov}.
\begin{prop}\hlabel{C0_intertwining}
If $\psi(\xi)=-{\rm{i}}\xi   \phi({\rm{i}}\xi   )$, $\xi\in\bb{R}$, then, the operator  $\Lambda^b_\psi \in \mathscr{B}(\C_b(\bb{R}))$ where
\begin{align}\hlabel{eq:Lambda_psi}
\Lambda^b_\psi f(x)=\bb{E}[f(x-\log I_\psi)]
\end{align}
 is bounded and
\[\Lambda^b_\psi {}_{\left. \right|{\C_b(\bb{R})\cap\bmrm{L}(\bb{R},e)}}=\Lambda_\psi {}_{\left.\right|{\C_b(\bb{R})\cap\bmrm{L}(\bb{R},e)}}.\]
In particular, $\Lambda^b_\psi$ maps $\C_0([-\infty,\infty))$ to itself.
\end{prop}
\begin{proof}
We define $\Lambda^b_\psi$ on $\C_b(\bb{R})$ as in \eqref{eq:Lambda_psi}, where ${I}_\psi$ is such that $\frac{1}{\sqrt{2\pi}}\bb{E}[e^{-{\rm{i}}\xi \log I_\psi}]=m_{\Lambda_\psi}(\xi)$ for all $\xi\in\bb{R}$. It easy to see that $\Lambda^b_\psi$ is bounded (with respect to the supremum topology). Now, for any $f\in \C^\infty_c(\bb{R})$,
\begin{align*}
\Lambda^b_\psi f(x)=\frac{1}{\sqrt{2\pi}}\int_{\bb{R}} e^{{\rm{i}}\xi  x}m_{\Lambda_\psi}(\xi)\mathcal{F}_{f}(\xi)\,d\xi=\hpt{R36}{\frac{e^{-\frac{x}{2}}}{\sqrt{2\pi}}\int_{-\infty-\frac{{\rm{i}}}{2}}^{\infty-\frac{{\rm{i}}}{2}}m_{\Lambda_\psi}\left(z+\frac{{\rm{i}}}{2}   \right)\fouho_{f}(z) e^{{\rm{i}}zx} dz}.
\end{align*}
 This implies that
$$\Lambda^b_\psi {}_{\left.\right|{\C^\infty_c(\bb{R})}}=\Lambda_\psi {}_{\left.\right|{\C^\infty_c(\bb{R})}}.$$
By density of $\C^\infty_c(\bb{R})$ in $\bmrm{L}(\bb{R},e)$, the above identity ensures that $$\Lambda^b_\psi {}_{\left.\right|{\C_b(\bb{R})\cap\bmrm{L}(\bb{R},e)}}=
\Lambda_\psi {}_{\left.\right|{\C_b(\bb{R})\cap\bmrm{L}(\bb{R},e)}}.$$
\end{proof}
\begin{prop}\hlabel{prop:one_sided_jump}
Let $\psi(\xi)=-{\rm{i}}\xi \phi({\rm{i}}\xi   )$, $\xi\in\bb{R}$, and define
\begin{align}\hlabel{eq:eigenfunction_1}
\Jpo(x)=\sum_{n=0}^\infty (-1)^n\frac{1}{W_{\phi}(n+1)}\frac{e^{nx}}{n!}.
\end{align}
Then, $\Jpo$  is an entire function and $\Jpo\in \C_0([-\infty,\infty))$. Moreover, if\[\xi\mapsto m_{H_\psi}\left(\xi+\frac{{\rm{i}}}{2}   \right)=\frac{\Gamma(\frac{1}{2}-{\rm{i}}\xi   )}{W_{\phi}(\frac{1}{2}+\mathrm{i}\xi)}\in\bmrm{L}(\bb{R}) \] then $\Jpo\in\bmrm{L}(\bb{R},e)$ and $\tau_{-y}\Jpo$ is an eigenfunction of $P_t[\psi]$ corresponding to the eigenvalue $e^{-te^{-y}}.$
\end{prop}
\begin{rem}
\begin{enumerate}[(i)]
\item We mention that the analytic power series $\cc{I}_\psi$, defined as \[J_\psi(z)=\cc{I}_\psi(-e^z), z\in \bb{C},\] was introduced in \cite{PatieAIHP} and a recent thorough study of this class of entire functions has been carried out in \cite{Bart-Patie}. In the latter reference, several instances of this class of power series are provided in connection with well-known special functions. In particular, when $\psi(\xi)=\xi^2$, then $\Jpo$ boils down to the Bessel function of index $0$.
\item From  \cite[Proposition 6.2]{patie2018}, $\xi \mapsto\left|m_{\Hpsi}(\xi+\i/2)\right|=\left|\frac{\Gamma(1/2-{\rm{i}}\xi   )}{W_{\phi}(1/2+\mathrm{i}\xi)}\right|\in\bmrm{L}(\bb{R})$ whenever $\phi\in\B^c_\ttt{d}$ or  $\phi_+\in\B_\ttt{d}$ and $\frac{1}{\ttt{d}}(\phi(0)+\overline{\mu}(0))>\frac{1}{2}$. In particular, this is always satisfied when $\ttt{d}>0$ and $\overline{\mu}(0)=\infty$. However, in the proof, we do not assume any condition on $\phi$ except that $\xi\mapsto\frac{\Gamma(1/2-{\rm{i}}\xi   )}{W_{\phi}(1/2+\mathrm{i}\xi)}\in\bmrm{L}(\bb{R})$.
\end{enumerate}
\end{rem}
\begin{proof} Since $\phi$ is increasing, we have  $W_{\phi}(n+1)\ge\phi(1)^n$ for all $n$ and therefore, the series defining $\Jp$ converges absolutely for all $z\in\bb{C}$. Now, consider the function
\begin{align*}
J(x)=\sum_{n=0}^\infty(-1)^n\frac{e^{nx}}{(n!)^2}=J_0\left(2e^{\frac{x}{2}}\right)
\end{align*}
where $J_0$ is the Bessel function of index $0$. It is well-known that $J_0\in \C_0([0,\infty))$ which implies that $J\in \C_0([-\infty,\infty))$. Also,
\begin{align*}
\Lambda^b_\psi J(x)&=\bb{E}[J(x-{I}_\psi)]
=\bb{E}\left[\sum_{n=0}^\infty(-1)^n\frac{1}{(n!)^2}e^{n(x-{I}_\psi)}\right]
=\sum_{n=0}^\infty (-1)^n\frac{1}{(n!)^2}e^{nx}\bb{E}[e^{-n{I}_\psi}],
\end{align*}
where the interchange of the sum and expectation is justified as the series in the right-hand side is absolutely convergent. Recalling that $\bb{E}[e^{-n{I}_\psi}]=\frac{n!}{W_{\phi}(n+1)}$, we obtain $\Jp=\Lambda^b_\psi J$. From Proposition~\ref{C0_intertwining}, we get that $\Jp\in \C_0([-\infty,\infty))$. Now, from Theorem~\ref{thm:spectrum}, we know that $P_t[\psi]$ has point spectrum when $m_{\Hpsi}(\cdot+\frac{{\rm{i}}}{2}   )\in\bmrm{L}(\bb{R})$ and the eigenfunction corresponding to the eigenvalue $e^{-te^{-y}}$ is given by $\tau_{-y}K_\psi$, where $K_\psi=\widehat{\fouho}(m_{\Hpsi}(\cdot+\frac{{\rm{i}}}{2}))$. The rest of the proof is devoted to show that $K_\psi=\Jp$. For each $\kappa>1$, let $\mf{f}_\kappa(x)=\kappa e^{-\kappa x}e^{-e^{-\kappa x}}$. Then, $\int_\bb{R}\mf{f}_\kappa(x)\,dx=1$. From the proof of Theorem~\ref{thm:spectrum}, we know that $\Hpsi\mf{f}_\kappa=K_\psi\ast\mf{f}_\kappa$. On the other hand,
\begin{align*}
\Hpsi\mf{f}_\kappa(x)&=e^{-\frac{x}{2}}\frac{1}{\sqrt{2\pi}}\int_{\bb{R}}\fouho_ {\Hpsi\mf{f}_\kappa}(\xi)e^{{\rm{i}}\xi x}\,d\xi
=\frac{1}{\sqrt{2\pi}}\int_{\bb{R}}e^{{\rm{i}}\xi x}m_{\Hpsi}\left(\xi+\frac{{\rm{i}}}{2}   \right)\fou_{\mf{f}_\kappa}\left(\xi+\frac{{\rm{i}}}{2}   \right)\,d\xi \\
&=\frac{1}{\sqrt{2\pi}}e^{-\frac{x}{2}}\int_{\bb{R}}e^{{\rm{i}}\xi x}\frac{\Gamma(\frac{1}{2}-{\rm{i}}\xi   )}{W_{\phi}(\frac{1}{2}+\mathrm{i}\xi)}\fou_{\mf{f}_\kappa}\left(\xi+\frac{{\rm{i}}}{2}   \right)\,d\xi\\
&=\frac{1}{\sqrt{2\pi}}\int_{-\infty+\frac{{\rm{i}}}{2}   }^{\infty+\frac{{\rm{i}}}{2}   } e^{{\rm{i}}zx}\frac{\Gamma(-{\rm{i}}z   )}{W_{\phi}(1+{\rm{i}}z)   }\Gamma\left(1+\frac{{\rm{i}}z}{\kappa}\right)\,dz \\
&=\frac{1}{\sqrt{2\pi}}\int_{-\infty+\frac{{\rm{i}}}{2}   }^{\infty+\frac{{\rm{i}}}{2}   }e^{{\rm{i}}zx}\frac{\pi\Gamma(1+\frac{{\rm{i}}z}{\kappa})}{W_{\phi}(1+{\rm{i}}z)   \sin(-i\pi z)\Gamma(1+{\rm{i}}z)   }\,dz \\
&=:\frac{1}{\sqrt{2\pi}}\int_{-\infty+\frac{{\rm{i}}}{2}   }^{\infty+\frac{{\rm{i}}}{2}   }e^{{\rm{i}}zx}G_\kappa(z)\,dz.
\end{align*}
To compute the above integral, we fix $B>0$, a positive odd integer $N$ and consider the rectangular contour formed by the points $-B+\frac{{\rm{i}}}{2}   , B+\frac{{\rm{i}}}{2}   , B-{\rm{i}}{N\over 2}, -B-{\rm{i}}\frac{N}{2}$ with clockwise orientation. The poles of $G_\kappa$ in this rectangle are $\{0,-\i,-2\i,\ldots, -\lfloor\frac{N}{2}\rfloor \i\}$ with residues $(-1)^n\frac{\Gamma(1+\frac{n}{\kappa})}{n!W_{\phi}(n+1)}$, $n=0,1,2,\ldots, \lfloor\frac{N}{2}\rfloor$. We split the contour integral into the following four parts
\begin{align}
&I^{(1)}_B=\int_{-B+\frac{{\rm{i}}}{2}   }^{B+\frac{{\rm{i}}}{2}   }e^{{\rm{i}}zx}G_\kappa(z)\,dz, \  I^{(2)}_{B,N}=\int_{B-{\rm{i}}\frac{N}{2}}^{B+\frac{{\rm{i}}}{2}   }e^{{\rm{i}}zx}G_\kappa(z)\,dz\nonumber \\
&I^{(3)}_{B,N}=\int_{-B-{\rm{i}}\frac{N}{2}}^{-B+\frac{{\rm{i}}}{2}   }e^{{\rm{i}}zx}G_\kappa(z)\,dz , \ I^{(4)}_{B,N}=\int_{-B-{\rm{i}}\frac{N}{2}}^{B-{\rm{i}}\frac{N}{2}}e^{{\rm{i}}zx}G_\kappa(z)\,dz .\nonumber
\end{align}
Since the integrals $I^{(2)}_{B,N}$ and $I^{(3)}_{B,N}$ are similar, we deal with only one of them, say $I^{(2)}_{B,N}$. For any fixed $N$, we observe that, for large values of $B$,
\begin{align*}
 \sup_{t\in [-\frac{N}{2},\frac{1}{2}]}\left|\frac{\Gamma(1-\frac{t}{\kappa}+{\rm{i}}\frac{B}{\kappa})}{W_{\phi}(1-t+\i B)\sin(\pi(t-\i B))\Gamma(1-t+{\rm{i}}\xi   )}\right|\le C_Ne^{-\frac{\pi}{2\kappa}|B|}.
\end{align*}
Therefore, for fixed $N$, $I^{(2)}_{B,N}\to 0, I^{(3)}_{B,N}\to 0$ as $B\to\infty$. For $I^{(4)}_{B,N}$ we proceed as follows. First, we observe that
$$\left|W_{\phi}\left(1+\frac{N}{2}+\mathrm{i}\xi\right)\right|\ge\left|W_{\phi}\left(\frac{1}{2}+\mathrm{i}\xi\right)\right|\phi(1)^{\frac{N}{2}}.$$
Then, for any odd natural number $N$,
\begin{align}
|I^{(4)}_{B,N}|&\le\int_{-B}^B\frac{e^{\frac{Nx}{2}}\left|\Gamma(1+\frac{N}{2\kappa}+\frac{{\rm{i}}\xi }{\kappa})\right|}{|W_{\phi}(\frac{N}{2}+1+\mathrm{i}\xi)||\sin(-\frac{N\pi}{2}-\i\pi\xi)||\Gamma(1+\frac{N}{2}+\mathrm{i}\xi)|}\,d\xi \nonumber \\
&\le e^{\frac{Nx}{2}}\phi(1)^{-\frac{N}{2}}\int_{-\infty}^\infty\left|\frac{\Gamma(1+\frac{N}{2\kappa}+\frac{{\rm{i}}\xi }{\kappa})}{W_{\phi}(\frac{1}{2}+\mathrm{i}\xi)\cosh(\pi\xi)\Gamma(1+\frac{N}{2}+\mathrm{i}\xi)}\right|\,d\xi \hlabel{eq:stirling_1} \\
&\le Ce^{\frac{Nx}{2}}\phi(1)^{-\frac{N}{2}}\left(\frac{N}{2}\right)^{\frac{N}{2}(\frac{1}{\kappa}-1)}\int_{-\infty}^\infty\left|\frac{1}{W_{\phi}(\frac{1}{2}+\mathrm{i}\xi)\cosh(\pi\xi)}\right|\,d\xi\to 0 \hlabel{eq:stirling_2}
\end{align}
as $N\to\infty$ since $\kappa>1$ and \eqref{eq:stirling_2} follows from \eqref{eq:stirling_1} by Stirling approximation. Therefore, $I^{(4)}_{B,N}\to 0$ uniformly in $B$ as $N\to\infty$. Finally, using the residue theorem for meromorphic functions, we conclude that, for all $x\in\bb{R}$,
\begin{align*}
\Hpsi\mf{f}_\kappa(x)=\sum_{n=0}^\infty (-1)^n\frac{\Gamma(1+\frac{n}{\kappa})}{n! W_{\phi}(n+1)}e^{nx} .
\end{align*}
On the other hand, it is easy to see that, for all $x\in\bb{R}$ and $\kappa>1$,
\begin{align*}
\Jpo\ast\mf{f}_\kappa (x)=\sum_{n=0}^\infty (-1)^n\frac{\Gamma(1+\frac{n}{\kappa})}{W_{\phi}(n+1)n!}e^{nx} .
\end{align*}
 Therefore, $K_\psi\ast\mf{f}_\kappa(x)=\Jpo\ast\mf{f}_\kappa(x)$ for all $x\in\bb{R}$ and $\kappa>1$. Finally, the proof of the proposition will be completed by means of the following lemma.
\end{proof}
\begin{lem}
If $g\in\bmrm{L}(\bb{R},e)$ and $h\in \C_b(\bb{R})$ such that $g\ast\mf{f}_\kappa=h\ast\mf{f}_\kappa$ for all $\kappa>1$, then $g=h$ a.e.
\end{lem}
\begin{proof}
We first recall the fact that, for any $g\in\bmrm{L}(\bb{R},e)$, $\lim_{t\to 0}\|\tau_tg-g\|_{\bmrm{L}(\bb R,e)}=0$ and $\hpt{R38}{\|\tau_t g\|_{\bmrm{L}(\bb{R},e)}=e^{-\frac{t}{2}}\|g\|_{\bmrm{L}(\bb{R},e)}}$ for all $t\in\bb{R}$. Now, for any $\kappa>1$, $g\ast\mf{f}_\kappa\in\bmrm{L}(\bb{R},e)$ and
\begin{align}
\|g\ast\mf{f}_\kappa-g\|^2_{\LRe}&=\int_\bb{R} e^x\left(\int_\bb{R}(\tau_{-y}g(x)-g(x))\mf{f}_\kappa(y)\,dy\right)^2\,dx \nonumber \\
&\le\int_\bb{R} e^x\left(\int_\bb{R}(\tau_{-y}g(x)-g(x))^2\mf{f}_\kappa(y)\,dy\right)\,dx \nonumber \\
&=\int_{\bb{R}}\|\tau_{-y}g-g\|^2_{\LRe}\mf{f}_\kappa(y)\,dy=\int_{\bb{R}}\|\tau_{-\frac{y}{\kappa}}g-g\|^2_{\LRe}\mf{f}_1(y)\,dy. \hlabel{dct}
\end{align}
Next,
$$\int_\bb{R}\|\tau_{-\frac{y}{\kappa}}g-g\|^2_{\LRe}\ \mf{f}_1(y)\,dy\le\int_\bb{R}(1+e^{\frac{y}{\kappa}})\|g\|^2_{\LRe}\ \mf{f}_1(y)\,dy=\|g\|^2_{\LRe}\left(1+\Gamma\left(1-\frac{1}{\kappa}\right)\right)$$
and the right-hand side is bounded with respect to $\kappa$. Hence, by applying the dominated convergence theorem, one gets,  as $\kappa\to\infty$,
\begin{align*}
\int_\bb{R} \|\tau_{-\frac{y}{\kappa}} g-g\|^2_{\bmrm{L}(\bb{R},e)}\mf{f}_1(y)dy\to 0.
\end{align*}
From \eqref{dct}, it follows that $\|g\ast\mf{f}_\kappa-g\|_{\bmrm{L}(\bb{R},e)}\to 0$ as $\kappa\to\infty$. Hence, there exists a sequence $(\kappa_n)$ of positive real numbers such that as $n \to \infty$, $\kappa_n\to\infty$ such that $g\ast\mf{f}_{\kappa_n}\to g$ a.e. On the other hand, for $h\in \C_b(\bb{R})$, $h\ast\mf{f}_\kappa(x)=\int_{\bb{R}}h(x-\frac{y}{\kappa})\mf{f}_1(y)\,dy\to h(x)$ as $\kappa\to\infty$ by continuity of $h$ and the bounded convergence theorem. Therefore, letting $n\to\infty$ on the identity $g\ast\mf{f}_{\kappa_n}=h\ast\mf{f}_{\kappa_n}$, the proof of the lemma follows.
\end{proof}

\subsection{Self-similar processes with two sided jumps}
\subsubsection{Example of a non-self-adjoint semigroup with real continuous spectrum}\hlabel{sec:cont_spect_ex} Consider $\psi\in\mathbf{N}(\bb{R})$ such that $\psi(\xi)=\phi_+(-{\rm{i}}\xi   )\phi_-({\rm{i}}\xi   )$ for all $\xi\in \bb{R}$ and
\begin{align*}
\phi_{\pm}(z)=m_\pm+\ttt{d}_\pm z+\int_0^\infty (1-e^{-zy})\mu_\pm(dy), \ \  z\in\bb{C}_{[0,\infty)},
\end{align*}
where $m_\pm\ge 0, \ttt{d}_\pm> 0$ and $\int_0^\infty (y\wedge 1)\mu_\pm(dy)<\infty.$ Define $\m_\pm=m_\pm+\overline{\mu}_\pm(0+)$. Then, we have the following result.
\begin{prop}\hlabel{prop:cs}
Let $\m_\pm, \ttt d_\pm$ be as above. Let $\overline{\nu}_\pm$ denote the tail of the measure  $e^{-ay}\mu_\pm(dy)$. If $\m_\pm<\infty$, $\frac{\m_+}{\ttt{d}_+}=\frac{\m_-}{\ttt{d}_-}$ and $w\mapsto\int_0^\infty\cos(wy)\overline{\nu}_\pm(y)\,dy$ is in $\bmrm[1]{L}(\bb{R})$, then, $P_t[\psi]\Lambda_\psi=\Lambda_\psi Q_t$ with $\Lambda_\psi$ being invertible. Therefore, $\sigma(P_t[\psi])=\sigma_c(P_t[\psi])=e^{t\bb{R}_-}$.
\end{prop}
\begin{rem} From Vigon's theory of philanthropy \cite{vigon2002}, if we assume that $\mu_\pm$ has non-increasing density with respect to Lebesgue measure, $\phi_+$ and $\phi_-$ are always Wiener-Hopf factors of some L\'evy-Khintchine exponent.
\end{rem}
To prove Proposition \ref{prop:cs}, we start by showing  the following lemma.
\begin{lem}
Let $\phi\in\B$ be such that $\m=m+\ov{\mu}(0+)<\infty$, then
\begin{align}\hlabel{eq:w_phi_est}
|W_\phi(a+{\rm{i}}\xi)|\asymp\sqrt{\phi(a)} W_\phi(a)e^{-\frac{\pi}{2}|\xi|} |\xi|^{a+\frac{\m}{\ttt{d}}-\frac{1}{2}}, \ \ \forall a>0,
\end{align}
if and only if $w\mapsto\int_0^\infty\cos(wy)\overline{\nu}(y)\,dy \in \bmrm[1]{L}(\bb{R})$, where $\overline{\nu}$ is the tail of the measure $e^{-ay}\mu(dy)$.
\end{lem}
\begin{proof}
First note that, from  \cite[Theorem 6.2(1)]{patiesavov}, we have, for all $a>0$,
\begin{align*}
|W_{\phi}(a+{\rm{i}}\xi)   |\asymp\frac{\sqrt{\phi(a)}W_\phi(a)}{\sqrt{\phi(a+{\rm{i}}\xi)   }}e^{-|\xi|\Theta_\phi(a,|\xi|)} \ \ \ \mbox{as $|\xi|\to\infty$}
\end{align*}
where
\begin{align*}
|\xi|\Theta_\phi(a,|\xi|)=\int_a^\infty\ln\left(\frac{|\phi(y+\mathrm{i}|\xi|)|}{\phi(y)}\right)\,dy
\end{align*}
which, from the proof of  \cite[Theorem 3.2]{patie2018}, turns out to be
\begin{align}\hlabel{eq:arg}
\Theta_\phi(a,|\xi|)=\int_0^{|\xi|}\arg(\phi(a+\mathrm{i}w))\,dw.
\end{align}
As $\ttt{d}>0$, $|\phi(a+{\rm{i}}\xi)   |\sim \ttt{d}|\xi|$ as $|\xi|\to\infty$. Thus, it is enough to show that $e^{-|\xi|\Theta_\phi(a,|\xi|)}\asymp |\xi|^{\frac{\m}{\ttt{d}}+a}e^{-\frac{\pi}{2}|\xi|}$. If $z=a+\mathrm{i}w$, then
\begin{align*}
\phi(z)=z\left(\frac{\m}{z}+\ttt{d}-\frac{1}{z}\int_0^\infty e^{-zy}\mu(dy)\right)=zg(z).
\end{align*}
Now, $\arg(g(z))=\arctan\left(\frac{\Im(g(z))}{\Re(g(z))}\right)$. We observe that
\begin{align}
\Re(g(a+\mathrm{i}w))&=\ttt{d}+{\rm{O}}\left(\frac{1}{w}\right) \hlabel{eq:real_part}\\
\Im(g(a+\mathrm{i}w))&=-\frac{\m}{w^2+a^2}+\frac{w}{w^2+a^2}\int_0^\infty\sin(wy)e^{-ay}\mu(dy)+{\rm{O}}\left(\frac{1}{w^2}\right). \hlabel{eq:im_part}
\end{align}
From \eqref{eq:real_part} and \eqref{eq:im_part}, we get
\begin{align}
\frac{\Im(g(a+\mathrm{i}w))}{\Re(g(a+\mathrm{i}w))}=-\frac{\m}{\ttt{d}w}+\frac{\frac{w}{w^2+a^2}\int_0^\infty\sin(wy)e^{-ay}\mu(dy)}{\Re(g(a+\mathrm{i}w))}+{\rm{O}}\left(\frac{1}{w^2}\right). \hlabel{eq:ratio_real_im}
\end{align}
Now, let us analyze the second term in the above expression. Using integration by parts (or, equivalently, Fubini's theorem), we have
\begin{align}
\frac{\frac{w}{w^2+a^2}\int_0^\infty\sin(wy)e^{-ay}\mu(dy)}{\Re(g(a+\mathrm{i}w))}=\frac{\frac{w^2}{w^2+a^2}\int_0^\infty\cos(wy)\overline{\nu}(y)\,dy}{\Re(g(a+\mathrm{i}w))}. \hlabel{eq:im_part2}
\end{align}
As $\Re(g(a+\mathrm{i}w))=\ttt{d}+{\rm{O}}\left(\frac{1}{w}\right)$, the right-hand side of \eqref{eq:im_part2} is integrable with respect to $w$ if and only if the function $w\mapsto\int_0^\infty\cos(wy)\overline{\nu}(y)\,dy$ is integrable with respect to $w$. Also, we observe that $\frac{\Im(g(a+\mathrm{i}w))}{\Re(g(a+\mathrm{i}w))}\to 0 \ \ \mbox{as $w\to\infty$}$ implies that
\begin{align}
&\arctan\left(\frac{\Im(g(a+\mathrm{i}w))}{\Re(g(a+\mathrm{i}w))}\right)-\frac{\Im(g(a+\mathrm{i}w))}{\Re(g(a+\mathrm{i}w))}={\rm{O}}\left(\left|\frac{\Im(g(a+\mathrm{i}w))}{\Re(g(a+\mathrm{i}w))}\right|^2\right).\hlabel{eq:arctan_g}
\end{align}
From \eqref{eq:ratio_real_im}, it is not hard to see that $\int_0^\infty\left|\frac{\Im(g(a+\mathrm{i}w))}{\Re(g(a+\mathrm{i}w))}\right|^2\,dw<\infty$. Thus, \eqref{eq:im_part2} and \eqref{eq:arctan_g} yield that
\begin{align*}
\int_0^{|\xi|}\arctan(g(a+\mathrm{i}w))\,dw=-\frac{\m}{\ttt{d}}\ln |\xi|+{\rm{O}}(1).
\end{align*}
Finally, $\lim_{w\to\infty}\Re(g(a+\mathrm{i}w))=\ttt{d}=\lim_{w\to\infty} g(a+\mathrm{i}w)>0$ implies that for large values of $w$, $\arg(\phi(a+\mathrm{i}w))=\arg(a+\mathrm{i}w)+\arg(g(a+\mathrm{i}w))$. Therefore,
\begin{align*}
|\xi|\Theta_\phi(a,|\xi|)&=\int_0^{|\xi|}\arg(\phi(a+\mathrm{i}w))\,dw \\
&=\int_0^{|\xi|}\arg(a+\mathrm{i}w)\,dw+\int_0^{|\xi|}\arg(g(a+\mathrm{i}w))\,dw+{\rm{O}}(1) \\
&=|\xi|\arctan\left(\frac{|\xi|}{a}\right)-\frac{a}{2}\ln\left(1+\frac{\xi^2}{a^2}\right)-\frac{\m}{\ttt{d}}\ln |\xi|+{\rm{O}}(1) \\
&=\frac{\pi}{2}|\xi|-a\ln |\xi|-\frac{\m}{\ttt{d}}\ln |\xi| +{\rm{O}}(1).
\end{align*}
Then, we conclude that
\begin{align*}
e^{-|\xi|\Theta_\phi(a,|\xi|)}\asymp |\xi|^{a+\frac{\m}{\ttt{d}}} e^{-\frac{\pi}{2}|\xi|}
\end{align*}
if and only if $w\mapsto\int_0^\infty\cos(wy)\overline{\nu}(y)\,dy$ is in $\bmrm[1]{L}(\bb{R})$.
This concludes the proof of the lemma.
\end{proof}
Now, coming back to the proof of  Proposition  \ref{prop:cs} if both $\phi_+$ and $\phi_-$ satisfy the conditions of the proposition, it is easy to see, from the estimate \eqref{eq:w_phi_est}, that $\xi \mapsto \left| \frac{W_{\phi_+}(\frac{1}{2}-{\rm{i}}\xi   )}{W_{\phi_-}(\frac{1}{2}+\mathrm{i}\xi)}\right|$ is bounded above and below. Therefore, both $\Lambda_\psi$ and $\Lambda^{-1}_\psi$ are invertible, which concludes the proof.

\subsubsection{Another example with two-sided jumps}\hlabel{ex:two_sided} Consider $\phi_+, \phi_-$ such that
\begin{equation*}
 \phi_+(z)=\frac{\Gamma(\tilde{\alpha}(1+z))}{\Gamma(\tilde{\alpha} z)} \ \ \mbox{and} \ \ \phi_-(z)=\frac{\Gamma(\rho+\alpha+\alpha z)}{\Gamma(\rho+\alpha z)}, \ \ z \in\bb{C}_{(0,\infty)},
 \end{equation*}
 where $\tilde{\alpha},\alpha\in (0,1), \ \rho>0$. From \cite{kp2013}, it is known that $\phi_+, \phi_-\in\B$ and the L\'evy measures of $\phi_+$ and $\phi_-$ are absolutely continuous with  non-increasing densities, namely, for $y>0$,
 \begin{equation*}
 \begin{aligned}
 \mu_+(dy)&=\frac{1}{\Gamma(1-\tilde{\alpha})}\frac{e^{\frac{y}{\tilde{\alpha}}}}{(e^{\frac{y}{\tilde{\alpha}}}-1)^{1+\tilde{\alpha}}}\,dy, \\
 \mu_-(dy)&=\frac{1}{\Gamma(1-\alpha)}\frac{e^{\frac{(1-\rho)y}{\alpha}}}{(e^{\frac{y}{\alpha}}-1)^{1+\alpha}}\,dy.
 \end{aligned}
 \end{equation*}
 Thus, from Vigon's theory of philanthropy \cite{vigon2002}, there is a L\'evy process whose L\'evy-Khintchine exponent is given by
 \begin{align*}
 \psi(\xi)=\phi_+(-{\rm{i}}\xi   )\phi_-({\rm{i}}\xi   ), \ \ \xi\in\bb{R}.
 \end{align*}
It is immediate that the Bernstein-gamma functions corresponding to $\phi_-$ and $\phi_+$ are  given by
\begin{equation*}
\begin{aligned}
&W_{\phi_+}(z)=\frac{\Gamma(\tilde{\alpha} z)}{\Gamma(\tilde{\alpha})}\textrm{ and }
 W_{\phi_-}(z)=\frac{\Gamma(\rho+\alpha z)}{\Gamma(\alpha+\rho)} \ \  z\in\bb{C}_{(0,\infty)}.
\end{aligned}
\end{equation*}
Using Stirling formula for the gamma function, we know that for all $a>0$ and large values of $|\xi|$,
\begin{equation*}
\begin{aligned}
&|W_{\phi_+}(a+{\rm{i}}\xi)   |\asymp |\xi|^{a\tilde{\alpha}-\frac{1}{2}} e^{-\frac{\tilde{\alpha}\pi |\xi|}{2}} \textrm{ and }
|W_{\phi_-}(a+{\rm{i}}\xi)   |\asymp |\xi|^{a\alpha+\rho-\frac{1}{2}}e^{-\frac{\alpha \pi |\xi|}{2}}.
\end{aligned}
\end{equation*}
Also, from the above estimates and the definitions of $\phi_+, \phi_-$, we infer that $\psi\in\mathbf{N}_b(\bb{R})$ for any $\alpha,\tilde{\alpha}\in (0,1), \rho>0$. From Theorem~\ref{thm:spectrum}, the multiplier of the Fourier  operator $\Hpsi$ is given by
\begin{align*}
m_{\Hpsi}\left(\xi+\frac{{\rm{i}}}{2}   \right)=\frac{W_{\phi_+}(\frac{1}{2}-{\rm{i}}\xi   )}{W_{\phi_-}(\frac{1}{2}+{\rm{i}}\xi)},
\end{align*}
and therefore, $\xi\mapsto m_{\Hpsi}\left(\xi+\frac{{\rm{i}}}{2}   \right)\in\bmrm{L}(\bb{R})$ in the following two cases:  when $\tilde{\alpha}>\alpha$ or  when $\alpha=\tilde{\alpha}, \ \rho>\frac{1}{2}$. Similarly,  the reciprocal function
\begin{align*}
\xi\mapsto \frac{1}{m_{\Hpsi}\left(\xi+\frac{{\rm{i}}}{2} \right)} \in\bmrm{L}(\bb{R})
\end{align*}
 when (i') $\tilde{\alpha}<\alpha$. Based on these  observation, we get the following result.
\begin{prop}
\begin{enumerate}[(i)]
\item If $\tilde{\alpha}>\alpha$, or, $\alpha=\tilde{\alpha}$ and  $\rho>\frac{1}{2}$, then $e^{t\bb{R}_-}\subseteq\sigma_p(P_t[\psi])$ and the eigenfunction corresponding to the eigenvalue $e^{-te^{-y}}$ is given by $\tau_{-y}\Jpo$ where
\begin{align*}
\Jpo(x)=\ttt{W}\left(\frac{\alpha}{\tilde{\alpha}},\alpha+\rho; -e^{\frac{x}{\tilde{\alpha}}}\right)
\end{align*}
and $\ttt{W}(\gamma,\beta;z)=\sum_{n=0}^\infty\frac{1}{\Gamma(\gamma n+\beta)}\frac{z^n}{n!}$ is the Wright hypergeometric function, which defines an entire function if $\gamma>-1$.
\item If  $\tilde{\alpha}<\alpha$, then $e^{t\bb{R}_-}\subseteq\sigma_r(P_t[\psi])$ and the co-eigenfunction corresponding to $e^{-te^{-y}}$ is given by $\tau_{-y}J_{\ov{\psi}}$ where
\begin{align}\hlabel{eq:co-eigenfunction_1}
J_{\ov{\psi}}(x)=e^{\frac{\rho x}{\alpha}}\ttt{W}\left(\frac{\tilde{\alpha}}{\alpha},\tilde{\alpha}+\frac{\tilde{\alpha}}{\alpha}\rho; -e^{\frac{x}{\alpha}}\right).
\end{align}
\end{enumerate}
\end{prop}
\begin{proof}
The proof will be again based on the integration of the multiplier function on a suitable contour. Assuming the condition (i) or (ii), from the proof of Theorem~\ref{thm:spectrum}, we know that
\begin{align*}
\Jpo(x)=e^{-\frac{x}{2}}\widehat{\fou}\left(m_{\Hpsi}\left(\cdot+\frac{{\rm{i}}}{2}   \right)\right)(x). \nonumber
\end{align*}
Let us assume that $\tilde{\alpha}>\alpha$. Then, $m_{\Hpsi}\left(\cdot+\frac{{\rm{i}}}{2}   \right)\in\bmrm[1]{L}(\bb{R})$. Therefore, by Fourier inversion we get
\begin{align*}
\Jpo(x)&=\frac{1}{\sqrt{2\pi}}\int_{-\infty+\frac{{\rm{i}}}{2}   }^{\infty+\frac{{\rm{i}}}{2}   }\frac{\Gamma(-\mathrm{i}\tilde{\alpha} z)}{\Gamma(\rho+\alpha(1+{\rm{i}}z)   )}e^{{\rm{i}}zx}\,dz \nonumber
:=\frac{1}{\sqrt{2\pi}}\int_{-\infty+\frac{{\rm{i}}}{2}   }^{\infty+\frac{{\rm{i}}}{2}   } G(z)e^{{\rm{i}}zx}\,dz. \nonumber
\end{align*}
Choosing the rectangular contour with vertices $-R, R, R-{\rm{i}}\frac{N}{2\tilde{\alpha}}, -R-{\rm{i}}\frac{N}{2\tilde{\alpha}}$, where $N$ is an odd natural number, we define the integrals on the four segments as
\begin{equation} \nonumber
\begin{aligned}
&I^{(1)}_R=\frac{1}{\sqrt{2\pi}}\int_{-R}^R G(z)e^{{\rm{i}}zx}\,dz , \ \ \ I^{(2)}_{R,N}=\frac{1}{\sqrt{2\pi}}\int_{-R}^{-R-{\rm{i}}\frac{N}{2\tilde{\alpha}}} G(z)e^{{\rm{i}}zx}\,dz \\
&I^{(3)}_{R,N}=\frac{1}{\sqrt{2\pi}}\int_{R}^{R-{\rm{i}}\frac{N}{2\tilde{\alpha}}} G(z)e^{{\rm{i}}zx}\,dz, \ \ \ I^{(4)}_{R,N}=\frac{1}{\sqrt{2\pi}}\int_{-R-{\rm{i}}\frac{N}{2\tilde{\alpha}}}^{R-{\rm{i}}\frac{N}{2\tilde{\alpha}}} G(z)e^{{\rm{i}}zx}\,dz.
\end{aligned}
\end{equation}
We note that the poles of the function $G$ in the rectangle are $\left\{0,-\frac{\rm{i}}{\tilde{\alpha}},\ldots, -{\rm{i}}\frac{\lfloor\frac{N}{2}\rfloor}{\tilde{\alpha}}\right\}$ with residues $\frac{1}{\tilde{\alpha} n!\Gamma(\frac{n\alpha}{\tilde{\alpha}}+\alpha+\rho)}$.
Arguing as in the proof of Proposition~\ref{prop:one_sided_jump}, for any fixed $N$, one can show $I^{(2)}_{R,N}, I^{(3)}_{R,N}\to 0$ as $R\to\infty$. On the other hand, after a change of variable,
\begin{align}
I^{(4)}_{R,N}&=\frac{1}{\sqrt{2\pi}}e^{\frac{Nx}{2\alpha}}\int_{-R}^R \frac{\Gamma(-\frac{N}{2}-{\rm{i}}\tilde{\alpha}\xi)}{\Gamma(\rho+\alpha+\frac{N}{2}+{\rm{i}}\alpha\xi)}e^{{\rm{i}}\xi x}\,d\xi \nonumber \\
&=\frac{e^{\frac{Nx}{2\alpha}}}{\sqrt{2\pi}}\int_{-R}^R\frac{\pi e^{{\rm{i}}\xi x}}{\cosh(\pi\tilde{\alpha}\xi)\Gamma(1+\frac{N}{2}+\mathrm{i}\tilde{\alpha}\xi)\Gamma(\rho+\alpha+\frac{N\alpha}{2\tilde{\alpha}}+{\rm{i}}\alpha\xi)}\,d\xi. \nonumber
\end{align}
Using Stirling formula, we have for all $\xi\in\bb{R}$,
\begin{equation}
\begin{aligned}
&\left|\Gamma\left(1+\frac{N}{2}+\mathrm{i}\tilde{\alpha}\xi\right)\right|\ge\frac{\Gamma(1+\frac{N}{2})}{\cosh^{\frac{1}{2}}(\tilde{\alpha}\pi\xi)} \nonumber \\
&\left|\Gamma\left(\rho+\alpha+\frac{N\alpha}{2\tilde{\alpha}}+\mathrm{i}\alpha\xi\right)\right|\ge\frac{\Gamma(1+\frac{N\alpha}{2\tilde{\alpha}}+\alpha+\rho)}{\cosh^{\frac{1}{2}}(\alpha\pi\xi)} \nonumber
\end{aligned}
\end{equation}
and, then
\begin{align*}
|I^{(4)}_{R,N}|\le \frac{e^{\frac{Nx}{2\alpha}}}{\sqrt{2\pi}\Gamma(1+\frac{N}{2})\Gamma(\rho+\alpha+\frac{N\alpha}{2\tilde{\alpha}})}\int_{-\infty}^\infty \frac{\cosh^{\frac{1}{2}}(\tilde{\alpha}\pi\xi)\cosh^{\frac{1}{2}}(\alpha\pi\xi)}{\cosh(\tilde{\alpha}\pi\xi)}\,d\xi.
\end{align*}
As $\tilde{\alpha}>\alpha$, the integral on the right-hand side is finite and hence, $I^{(4)}_{R,N}\to 0$ as $N\to\infty$, uniformly in $R$. Thus, invoking Cauchy integral formula, we have
\begin{align*}
\Jpo(x)=\sum_{n=0}^\infty\frac{(-1)^n}{\Gamma(\frac{n\alpha}{\tilde{\alpha}}+\alpha+\rho)}\frac{e^{\frac{nx}{\alpha}}}{n!}.
\end{align*}
When $\tilde{\alpha}=\alpha$ and $\rho>\frac{1}{2}$, the function $m_{\Hpsi}\left(\cdot+\frac{{\rm{i}}}{2}   \right)\notin\bmrm[1]{L}(\bb{R})$ for $\rho\in (\frac{1}{2},1]$. In this case, we can use the same idea as in the proof of Proposition~\ref{prop:one_sided_jump}, where we have used the convolution of the integrand with respect to a class of kernel functions $\{\mf{f}_\kappa, \kappa>1\}$ and recalling the  fact (see \eqref{eq:eigenfunction_1} with $\phi_+(s)=s+\alpha+\rho$) that
\begin{align*}
x\mapsto \sum_{n=0}^\infty \frac{(-1)^n}{\Gamma(n+\alpha+\rho)}\frac{e^{\frac{nx}{\alpha}}}{n!} \in \C_0(\bb{R}),
\end{align*}
we can conclude that
\begin{align*}
\Jpo(x)=\sum_{n=0}^\infty \frac{(-1)^n}{\Gamma(n+\alpha+\rho)}\frac{e^{\frac{nx}{\alpha}}}{n!}.
\end{align*}
When $\tilde{\alpha}<\alpha$, $\frac{1}{m_{\Hpsi}}\left(\cdot+\frac{{\rm{i}}}{2}   \right)\in\bmrm{L}(\bb{R})$ and, from the proof of Theorem~\ref{thm:spectrum}, one gets
\begin{align*}
J_{\ov{\psi}}(x)=\frac{1}{\sqrt{2\pi}}\int_{-\infty-\frac{\i}{2}}^{\infty+\frac{{\rm{i}}}{2}   }\frac{\Gamma(\rho-i\alpha z)}{\Gamma(\tilde{\alpha}(1+{\rm{i}}z)   )}e^{{\rm{i}}zx}\,dz.
\end{align*}
The poles of the function $z\mapsto\frac{\Gamma(\rho-{\rm{i}}\alpha z)}{\Gamma(\tilde{\alpha}(1+{\rm{i}}z)   )}$ are $\{-\frac{n+\rho}{\alpha}{\rm{i}}, \ n\in\bb{N}\cup\{0\}\}$ with residues $\frac{1}{\alpha n!\Gamma(\tilde{\alpha}(1+\frac{n+\rho}{\alpha}))}$. Using the same argument as in the case $\tilde{\alpha}>\alpha$ and Cauchy's theorem of residues, \eqref{eq:co-eigenfunction_1} follows.
\end{proof}
\bibliographystyle{ieeetr}

\end{document}